\documentclass[11pt]{article}

\usepackage{palatino}
\usepackage{amsfonts}
\usepackage{authblk}
\usepackage{graphicx}
\usepackage{bm}
\usepackage{enumerate}
\usepackage{amsmath, amssymb}
\usepackage{amsthm}
\usepackage{color}
\usepackage{subfigure}
\usepackage{caption} % for subfig captions
\usepackage{pgf} % needed for TpX graphics
\usepackage{tikz} % needed for TpX graphics
\usepackage{upgreek}
\usepackage[comma, authoryear]{natbib}
\usepackage[hidelinks]{hyperref}

\numberwithin{equation}{subsection}

\theoremstyle{plain}
  \newtheorem{thm}{Theorem}[section]
\newtheorem{prop}[thm]{Proposition} 
\newtheorem{cor}[thm]{Corollary}
\newtheorem{lemma}[thm]{Lemma}
\theoremstyle{definition}

\newtheorem{remark}[thm]{Remark}

\DeclareMathOperator{\R}{\mathbb{R}}

\def\iid{i.i.d.\ }
\def\bl{\boldsymbol{\lambda}}
\def\Rstar{R^\star_{k^\star, l^\star}}
\def\lsnew{\lambda_{\gamma,\delta}^\star}
\def\blsnew{\boldsymbol{\lambda}_{\gamma, \delta}^\star}
\def\gade{1 + \gamma + \delta}

\def\bpis{\boldsymbol{\pi}^\star}
\def\ls{\lambda^\star}
\def\bls{\boldsymbol{\lambda}^\star}
\def\bz{\mathbf{z}}
\def\As{A^\star}
\newcommand{\e}{\mathrm e}
\newcommand{\De}{\mathrm d}
\theoremstyle{remark} % For an italic header, more subtle than definition style

\begin{document}
\title{Rates of convergence for extremes of geometric random variables and marked point processes}
\author[1]{Alessandra Cipriani}
%\address{Institut f\"ur Mathematik\\ Universit\"at Z\"urich\\ Winterthurerstrasse 190\\ 8057-Zurich, Switzerland}
\affil[1]{Weierstrass Institute\\ Mohrenstrasse 39\\  10117, Berlin, Germany}

\author[2]{Anne Feidt}
\affil[2]{Institut f\"ur Mathematik, Universit\"at Z\"urich\\ Winterthurerstrasse 190\\ 8057, Zurich, Switzerland}
\affil[1]{\href{mailto:Alessandra.Cipriani@wias-berlin.de}{ Alessandra.Cipriani@wias-berlin.de}}
\affil[2]{\href{mailto:Anne.Feidt@gmx.net}{Anne.Feidt@gmx.net}}

\date{\today}

\maketitle
\bibliographystyle{abbrvnat}
\begin{abstract}
We {use the Stein-Chen method to}  
study the extremal behaviour of
%the problem of extremes for 
univariate and bivariate geometric laws. 
{We obtain a rate for the convergence, to the Gumbel distribution, of the law of the maximum of \iid geometric random variables, and show that convergence is faster when approximating by a discretised Gumbel. We similarly find a rate of convergence for the law of maxima of bivariate Marshall-Olkin geometric random pairs when approximating by a discrete limit law.}
%We 
%determine the rate of convergence of the maximum of \iid geometric random variables 
%and of bivariate vectors following the Marshall-Olkin geometric distribution 
%to a Gumbel distribution, and propose a new discretised version as limit.
We 
introduce marked point processes of exceedances (MPPEs), 
{both with univariate and}
%whose \iid marks are 
bivariate Marshall-Olkin geometric variables
{as marks and }%. We 
%use the Stein-Chen method to 
determine 
bounds on the error of the approximation, in an appropriate probability metric,
of the law of the MPPE 
by that of a Poisson process
{with same mean measure}. We then approximate by another Poisson process with an easier-to-use mean measure and estimate the error of this additional approximation. \\
This work contains and extends results contained in the second author's PhD thesis (\cite{Feidt:2013}) under the supervision of Andrew D. Barbour.

\end{abstract}
\maketitle
\section{Introduction}
The problem of determining the behavior of extremes of random variables is a mathematically intriguing problem. Already the case of the maximum or minimum of a sample of $n$ independent and identically distributed ({i.}{i.}d.) random variables
$X_1, \ldots, X_n$ presents interesting aspects, starting from the existence of a domain of attraction under suitable conditions
(\cite{Fisher/Tippett:1928}). In this work it is also highlighted that
{for discrete distributions for which the jump heights continue to be too large,
}no non-degenerate limit distribution may be found for 
$X_{(n)}$, the maximum of the sample (\cite{Leadbetter_et_al:1983}, Theorem 1.7.13). %for some well-known discrete distribution functions, no non-degenerate limit law may be found. So what are the conditions for the existence of non-degenerate limit laws? An answer is given by Theorem 1.7.13 in \cite{Leadbetter_et_al:1983} which basically says that if the jump heights continue to be too large, there is no non-degenerate limit distribution for $X_{(n)}$, the maximum of the sample.
Well-known examples % do not satisfy these conditions 
are the geometric and Poisson distributions \cite[Examples 1.7.14, 1.7.15]{Leadbetter_et_al:1983}. There is, however, a way to partially remedy this. %\cite{Anderson_et_al:1997} and \cite{Nadarajah/Mitov:2002} determined limit laws for the binomial, geometric and Poisson distributions by 
{By} allowing
the distributional parameter (or one of them) to vary with the sample size $n$ at a suitable rate,
{\cite{Anderson_et_al:1997} and \cite{Nadarajah/Mitov:2002} determined extremal limit laws for the Poisson and geometric distributions, respectively}. 
%More specifically, the Poisson distribution was handled in the first work, while the second studied convergence of univariate and bivariate geometric maxima. 
In this paper, we will concentrate on the {geometric distribution. We will consider both univariate and bivariate geometric random variables, as well as point processes with geometric marks, again both univariate and bivariate.

Using a result from the Stein-Chen method for Poisson approximation by \cite{Barbour/Hall:1984}, contained in our work in Theorem~\ref{t: MyMichelMultivariate}, we determine bounds on the error, in the Kolmogorov distance, of the approximation of the maximum, under different normalisations, of \iid geometric random variables by the Gumbel distribution. Our results show that convergence is faster and requires no conditions on the distributional parameter when approximating by a discretised Gumbel distribution. We similarly determine an error bound for the approximation of the joint law of maxima of random pairs, following the bivariate Marshall-Olkin geometric distribution, by an appropriate discrete limit law.} 

{We further use the Stein-Chen method for Poisson process approximation to determine bounds on the error, in a suitable probability metric, of the law of a \textit{marked point process of exceedances} (MPPE), defined by
\begin{equation}\label{d: MPPEu}
\Xi_{u,n} := \sum_{i=1}^n I_{\{X_i >u\}} \delta_{X_i},
\end{equation} 
by that of a Poisson process whose mean measure equals that of the MPPE, where the $X_i$ are \iid geometric random variables and $u$ denotes a threshold, and $\delta_z$ denotes the \textit{Dirac measure} for a point $z \in E$, the state space of the $X_i$'s.}
%latter case. 
%We obtain the rate of approximation of the maximum to a discretised version of the Gumbel distribution in the univariate case, and then generalise the result to bivariate Marshall-Olkin geometrically distributed variables.\\
%We also introduce a new point process to go alongside the more well-known
%point processes of exceedances (PPE's), and marked point processes (MPP's). Such point processes have been treated in detail in, for instance, \cite{Leadbetter_et_al:1983} and \cite{Resnick:1987}, respectively.  We would like to recall roughly the main features of these two processes. Consider first a PPE $\Xi$ on the state space $(0,\,1]$ of the form $X= \sum_{i=1}^n I_{\{X_i >u\}}\delta_{in^{-1}}.$ $\Xi$ recovers the total number of extreme points of an \iid sequence of variables $X_1,\,\ldots,\,X_n$. As for MPP's, they are point processes of the form $\sum_{i=1}^n \delta_{X_i}$, that live on the state space $E$
%of the random variables $X_i$'s. An MPP gives a random configuration of points in space and counts the number of points in any measurable subset of the state space
%that it is applied to. 
%In this article we introduce yet another kind of point process, that we call \textit{marked point process of exceedances} (MPPE) and that we define as follows:
%\begin{equation}\label{d: MPPEu}
%\Xi_{u,n} := \sum_{i=1}^n I_{\{X_i >u\}} \delta_{X_i}.
%\end{equation} 
%Though 
Though the MPPE does not mark the points that exceed $u$ %as %the PPE does,
{in the way that a
\textit{point process of exceedances} (PPE) of the form
$\sum_{i=1}^n I_{\{X_i >u\}}\delta_{in^{-1}}$ does,
it contains more information relevant to the study of extreme values than a \textit{marked point process} (MPP) of the form $\sum_{i=1}^n \delta_{X_i}$,
as it is not only a random configuration of points in space, but specifically a random configuration of points exceeding a threshold. A related approach is that of {\it empirical point processes of exceedances} (EPPEs), which have the form $\sum_{i=1}^n I_{\{(i/n,\,f^{-1}(X_i))\in A\}}$ for $A$ a Borel set of $[0,\,1]\times [0,\,+\infty)$ and $f:\,[0,\,+\infty)\to\R$ a strictly decreasing function. For more details on the study of these processes, we refer to \cite{Leadbetter_et_al:1983}, \cite{Resnick:1987}, \cite{Barbour/Novak/Xia:2002}. In addition to an MPPE with univariate geometric marks, we}
%We 
consider %in the present paper MPPEs 
an MPPE with bivariate marks that follow %a certain bivariate geometric distribution, 
the
Marshall-Olkin
%(MO)
geometric distribution, {for which the MPPE indicates whether they lie in a subset $A$ of extreme values of the marks' state space.}
%Lastly, we present a result on the limiting distribution of maxima which is shown to be different from that obtained by \cite{Mitov/Nadarajah:2005} due to the different dependence of the parameters defining the MO distribution. The core of our strategy in this case relies again on the approximation given by Theorem~\ref{t: MyMichelMultivariate}.\\
%We can readily approximate the law of an MPPE with such marks by that of a Poisson process with the same mean measure. Proposition \ref{t: MO_Geo_dTV_lattice} determines an error estimate in the total variation distance for the approximation of the law of the MPPE by that of a Poisson process with equal mean measure. 
{For both cases, the}
%The 
estimate for 
%the first step, 
the actual ``Poisson approximation" %,
comes easily. The reason for this is that we use \iid samples $X_1, \ldots,X_n$ and \iid indicators $I_{\{X_1 \in A\}}, \ldots, I_{\{ X_n \in A\}}$. This allows us to apply Theorem~\ref{t: MyMichelMultivariate}, which reduces the problem to the approximation of a binomial by a Poisson distribution. However, as the marks have geometric, and thereby discrete margins, the mean measure will live on a lattice and be rather tedious to work with in practial applications. We would therefore prefer to approximate by a further Poisson process with a continuous mean measure. 
%by spreading out the point probabilities of the Marshall-Olkin 
%distribution over the entire space. As this intensity function depends on $n$, 
%we make some assumptions on the parameters of the Marshall-Olkin geometric distribution.
Since the total variation distance 
is too strong for this kind of approximation, we use the weaker $d_2$-distance instead, which is not as sensitive towards small changes in the positions of the points of the point processes. Our main effort thus lies in determining error bounds on the approximation of a Poisson process by another Poisson process. 
As the error given by Theorem~\ref{t: MyMichelMultivariate} is 
only $P(X \in A)$, the error obtained by further approximating by a different Poisson process is typically the bigger of the two,
both in the univariate and bivariate case.
%This might, however, not be the case, if we had a different basic set-up, i.e. if the point process that we consider were different to the MPPE in (\ref{d: MPPEu}). The error arising from approximation by a Poisson process with equal mean measure might then be bigger, and the error from further approximation by a different Poisson process (if not made redundant entirely) might be smaller. 

{For the MPPE with univariate geometric marks, we specifically add, to the first approximation by a Poisson process with mean measure equal to that of the MPPE, a further approximation by a Poisson process with continuous intensity function.  %whose continuous mean measure is the same as that of an MPPE with exponential marks. 
For the MPPE with bivariate marks, we proceed by ''spreading out`` the point probabilities of the Marshall-Olkin 
distribution over the entire space. As the intensity function obtained through this depends on $n$, }
we make some additional assumptions on the parameters of the Marshall-Olkin geometric distribution {and further approximate by a Poisson process with a ''continuous" mean measure independent of $n$}. By adding up the $d_2$-error bounds arising from each step, we give the total error bound for the approximation of the MPPE by the final Poisson process.\\
The structure of the paper is as follows: %in Section~\ref{sec:review} we recall the basic definitions on Poisson point processes and distances between point measure, and the Marshall-Olkin (MO) geometric bivariate distribution. In Section~\ref{sec:max} we first treat maxima for the univariate geometric case and identify the limit distribution for the maximum of the bivariate MO geometric distribution.
{in Section~\ref{sec:review} we recall necessary basic definitions as well as results from the Stein-Chen method. 
%In Section~\ref{sec:max} we first treat maxima for the univariate geometric case and identify the limit distribution for the maximum of the bivariate MO geometric distribution.
In Section~\ref{sec:max} we treat maxima of univariate geometric random variables, as well as joint maxima of random pairs that follow the bivariate Marshall-Olkin geometric distribution. 
%In Section~\ref{Sec: MO_Geo} we construct the approximation of the related MPPE by a Poisson point process on the lattice first, and with a ``continuous'' intensity function in a second step. 
In Section~\ref{Sec: MO_Geo} we first study the MPPE with univariate geometric marks, and then the various steps involved in approximating the MPPE with bivariate geometric marks by a Poisson process with ``continuous" mean measure independent of $n$.}
\section{Background}\label{sec:review}
Throughout, let $E$ be a locally compact separable metric space. 
%By this we mean that every $z \in E$ has a compact neighbourhood, 
%and that $E$ contains a countable dense subset and is Hausdorff. It is also a Polish space.
In later applications, we simply use $E\subseteq \mathbb{R}^d$, $d \ge 1$. 
Let $E$ be equipped with its Borel $\sigma$-algebra $\mathcal{E}:= \mathcal{B}(E)$.   
%Suppose that $\xi$ is a Radon measure, i.e.\ suppose that $\xi(K) < \infty$ for compact sets $K \in \mathcal{E}$.
Suppose that $\xi$ is an integer-valued Radon measure on $\mathcal{E}$.
Denote by $\overline{M}_p(E)$ the space of all 
such
point measures 
$\xi$
on $E$ and equip $\overline{M}_p(E)$ with the $\sigma$-algebra $\overline{\mathcal{M}}_p(E)$
that is the smallest $\sigma$-algebra making the evaluation maps $\xi \to \xi(B)$ from $M_p(E)$ to $[0,\infty]$ measurable for any set $B \in \mathcal{E}$.
Similarly, 
denote by $M_p(E) \subset \overline{M}_p(E)$ the space of all \textit{finite} point measures
and equip
$M_p(E)$ with the $\sigma$-algebra $\mathcal{M}_p(E):=\overline{\mathcal{M}}_p(E)\cap M_p(E). $
%that is the smallest $\sigma$-algebra making the evaluation maps $ \xi \to \xi(B)$ from $M_p(E)$ to $[0,\infty)$ measurable for any set $B \in \mathcal{E}$.  

Let $(\Omega, \mathcal{F}, P)$ be a probability space. We define a \textit{point process} on $E$
by $\Xi: \, (\Omega, \mathcal{F}, P) \to (\overline{M}_p(E), \overline{\mathcal{M}}_p(E))$, $\omega \mapsto \Xi(\omega) = \xi$
and
denote its \textit{intensity measure} or \textit{mean measure} on $\mathcal{E}$ by $\bl$.
% defined,
%for any $B \in \mathcal{E}$, by 
%\[
%\bl(B) = \mathbb{E}\Xi(B) = \int_{\Omega} \Xi(\omega,B) P(d\omega) = \int_{\overline{M}_p(E)} \xi(B)P_{\Xi}(d\xi).
%\]
%We will consider often \textit{Poisson point processes} $\Xi$ on $E$ with mean measure $\bl$ whose law we indicate  by $\mathrm{PRM}(\bl)$, i.e. $\Xi \sim \mathrm{PRM}(\bl)$.
Furthermore, we denote the law of \textit{Poisson point processes} $\Xi$ on $E$ with mean measure $\bl$ by $\mathrm{PRM}(\bl)$.
For a general review of point processes, we refer to \cite{Resnick:1987}.
%{\color{red} To delete - this seems out of context here: 
%We define also a normalised version,
%\begin{equation}\label{d: normalised_MPPE}
%\Xi^\star_{A^\star} := \sum_{i=1}^n I_{\{ X_i^\star \in A^\star\}} \delta_{X^\star_i},
%\end{equation}
%which has state space $E^\star$ and mean measure
%\begin{equation}\label{d: normalised_bl}
%\bl^\star(\,\cdot\,) := \mathbb{E}\Xi^\star_{A^\star}(\,\cdot\,) = nP(X^\star \in A^\star \cap \,\cdot\,), 
%\end{equation}
%on $\mathcal{E}^\star$.\\
%}
%Finally, we would like to recall briefly 
\subsection{Distance between measures}\label{subsec:dist_meas}
%For each of these, suppose that $\mu$ and $\nu$
% are two probability measures on a measurable space $(E, \mathcal{E})$.
%The \textit{total variation distance} between $\mu$ and $\nu$ is defined as follows:
%\[
%d_{TV}(\mu,\nu) = \sup_{h \in H} \left| \int_E hd\mu - \int_E h d\nu  \right| = \sup_{B \in \mathcal{E}} |\mu(B)-\nu(B)|,
%\]
%where the test functions $h$ are indicators of measurable subsets of $E$, i.e.\ $H :=  \{I_B;\, B \in \mathcal{E}\}$, where, for any $x \in E$,
%$I_B(x)=1$ if $x \in B$ and $I_B(x)=0$ if $x \notin B$.
%
% Suppose that $E=\mathbb{R}$ and that the test functions are the indicators of half-lines in $\mathbb{R}$, i.e.\ 
%$\mathcal{H}=\{I_{(-\infty,x]};\, x \in \mathbb{R}\}$. The \textit{Kolmogorov distance} is then defined as follows:
%\[
%d_K(\mu,\nu) = \sup_{h \in \mathcal{H}} \left| \int_E hd\mu - \int_E h d\nu\right| = \sup_{x \in \mathbb{R}} |\mu\{(-\infty,x]\} - \nu\{(-\infty,x]\}|.
%\]
%For two random variables $X \sim \mu$ and $Y \sim \nu$, the Kolmogorov distance is thus given by
%$\sup_{x \in \mathbb{R}} |P(X \le x) - P(Y \le x)|$.
%We recall here some known facts for Poisson process approximation in a 
%metric that is weaker than the total variation metric, contained for example in \cite[Section 3]{Barbour/Brown:1992}.
We recall here some known facts on a probability metric, the $d_2$-metric, that is weaker than the total variation metric, defined for example in \cite[Section 3]{Barbour/Brown:1992}.
Let $d_0$ be a metric on $E$ that is bounded by $1$.
We now define metrics on both the space $M_p(E)$ of finite point measures over $E$ and on the set of probability measures over $M_p(E)$. 
Let $\mathcal{K}$ denote the set of functions $\kappa: E \to \mathbb{R}$ such that 
\[
s_1(\kappa) = \sup_{z_1 \neq z_2 \in E} \frac{|\kappa(z_1) - \kappa(z_2)|}{d_0(z_1, z_2)} \le +\infty,  
\]
which implies that for all $z_1 \neq z_2 \in E$, $|\kappa(z_1) - \kappa(z_2)| \le s_1(\kappa)d_0(z_1, z_2)$. Thus each function $\kappa \in \mathcal{K}$
is Lipschitz continuous with constant $s_1(\kappa)$.
Define a distance $d_1$ between two finite measures $\boldsymbol{\rho}$ and $\boldsymbol{\sigma}$ over $E$ by
% I do not really need finiteness of these two measures, since the distance is defined to be one if they are infinite. Is that correct? 
\begin{equation}\label{d: def_d1}
d_1(\boldsymbol{\rho},\boldsymbol{\sigma}) = \left\{ 
\begin{array}{ll} 
1, & \textnormal{if }  \boldsymbol{\rho}(E) \neq \boldsymbol{\sigma}(E),\\
\displaystyle \frac{1}{\boldsymbol{\rho}(E)}\, \sup_{\kappa \in \mathcal{K}} \frac{\left| \int_E \kappa d\boldsymbol{\rho} - \int_E \kappa d\boldsymbol{\sigma}\right|}{s_1(\kappa)},
& \textnormal{if } \boldsymbol{\rho}(E) = \boldsymbol{\sigma}(E)\neq 0.
\end{array}
\right.
\end{equation}
Note that $d_1$ is bounded by $1$. We can use $d_1$ as distance between point measures in $M_p(E)$. The $d_1$-distance is then a Wasserstein metric induced by $d_0$
over point measures on $E$. 

We next construct a metric $d_2$ that is a Wasserstein metric induced by $d_1$ over probability measures on $M_p(E)$. 
Let $\mathcal{H}$ denote the set of functions $h : M_p(E) \to \mathbb{R}$ such that 
\begin{equation}\label{d: s_2h}
s_2(h) = \sup_{\xi_1 \neq \xi_2 \in M_p(E)} \frac{|h(\xi_1)-h(\xi_2)|}{d_1(\xi_1, \xi_2)} < \infty,
\end{equation}
i.e. each function $h \in \mathcal{H}$ is Lipschitz continuous with constant $s_2(h)$. We define a distance $d_2$ between probability measures 
$\mu$ and $\nu$ over $M_p(E)$ by
\begin{equation}\label{d: def_d2}
d_2(\mu,\nu)  = \frac{1}{s_2(h)}\, \sup_{h \in \mathcal{H}} \left| \int_{M_p(E)}h d\mu - \int_{M_p(E)} h d\nu \right|.
\end{equation}
Note that $d_2$ is bounded by $1$. 
%{\color{red} The following is only needed to show that $d_2 \le d_{TV}$ - I suggest we delete this: By setting $\tilde{h}:= h/s_2(h)$ for each $h \in \mathcal{H}$, we may equivalently write
%\[
%d_2(\mu, \nu) = \sup_{\tilde{h} \in \widetilde{\mathcal{H}}}  \left| \int_{M_p(E)}\tilde{h} d\mu - \int_{M_p(E)} \tilde{h} d\nu \right|,
%\]
%where $\widetilde{\mathcal{H}} = \{ \tilde{h}: M_p(E) \to \mathbb{R}; \, |\tilde{h}(\xi_1) - \tilde{h}(\xi_2)| \le d_1(\xi_1, \xi_2) \le 1, \forall
%\xi_1, \xi_2 \in M_p(E)\}$.}
%It is known that for any probability measures $\mu$ and $\nu$ on $M_p(E)$,
%\begin{equation}\label{t: d2_smaller_dTV}
%d_2(\mu,\nu) \le d_{TV}(\mu,\nu). 
%\end{equation}
It can readily be shown that $d_2(\mu,\nu) \le d_{TV}(\mu,\nu)$.
\subsection{The Stein-Chen method}
We briefly recall
the theory of the Stein-Chen method, which was first worked out by \cite{Chen:1975a}. Let $Z$
%$\sim \mu$
be a random variable with law $\mu$. A characterising operator for $\mu$ is an operator $A_\mu$ on some class of functions $\mathcal F$ such that, for any random variable $X$,
$$
\mathbb E\left[A_\mu f(X)\right]=0\quad\forall\,f\in\mathcal F\quad \mathrm{iff} \,X\sim\mu.
$$
The idea of the method is that $\mathbb E\left[A_\mu f(X)\right]$ determines the ``distance'' between $Z$ and $X$. To develop this idea, given a function $h\in \mathcal H$ (for example, indicator functions of intervals), we need to find a function $f_h$ such that $A_\mu f_h(x)=h(x)-\mathbb E\left[h(Z)\right]$. Hence this yields $\left| \mathbb E \left[A_\mu f_h(X)\right]\right|=\left|\mathbb E\left[h(X)\right]-\mathbb E\left[h(Z)\right]\right|.$ If the left-hand side is small, then $\left|\mathbb E\left[h(X)\right]-\mathbb E\left[h(Z)\right]\right|$ is also small. A detailed description of the Stein-Chen method can be found, other than in the original article, in \cite{Barbour:1997}. We recall also the following results to keep the paper self-contained: let $Z= \{Z_t, t \in \mathbb{R}_+\}$ be an immigration-death process on $E$ with immigration intensity $\bl$\footnote{$\bl$ is a finite measure over $E$ with $\bl(E)= \lambda$.} and unit per-capita death rate. For any bounded 
%$h: M_p(E) \to \mathbb{R}$
$h \in \mathcal{H}$, let the function $\upgamma: M_p(E) \to \mathbb{R}$ be given by
\[
\upgamma(\xi) = - \int_0^\infty \left\{ \mathbb{E}^\xi h(Z_t) - \mathrm{PRM}(\bl)(h) \right\} \De t.
\]
Then
\begin{prop}\label{t: upgamma_welldefined}
$\upgamma(\xi)$ is well defined, and $\sup_{\xi:\, \xi(E)=k} |\upgamma(\xi)| < \infty$ for each $k \in \mathbb{Z}_+$.
\end{prop}
 \begin{lemma}\label{t: Delta_1_upgamma_d2} If $\gamma$ is defined as above for $h: M_p(E) \to \mathbb{R}$ any function in $\mathcal{H}$, define 
 $$\displaystyle \Delta_1 \upgamma  = \sup_{\xi \in M_p(E), z \in E} \left| \upgamma(\xi + \delta_z) - \upgamma(\xi)\right|.$$
 Then 
\[
\Delta_1 \upgamma  \le s_2(h) \min\left\{1 ,\, \frac{1.65}{\sqrt{\lambda}} \right\}. 
\]  \end{lemma}
\begin{prop}\label{t: upgamma_solves_Stein}
The function $\upgamma$ satisfies the Stein equation
\[
(\mathcal{A}\upgamma)(\xi) = h(\xi)-\mathrm{PRM}(\bl)(h),
\]
for all $\xi \in M_p(E)$.\end{prop}
%\begin{proof}
Proofs of the above can be found in 
%This is the content of 
{\cite[Prop. 10.1.1]{Barbour_et_al.:1992}}, {\cite[Lemma 10.2.3]{Barbour_et_al.:1992}} and {\cite[Prop. 10.1.2]{Barbour_et_al.:1992}}.
%\end{proof}
We also remark that, by Proposition \ref{t: upgamma_solves_Stein}, 
\begin{align}\label{d: Stein_eq_d2}
\begin{split}
\left| \mathbb{E}(\mathcal{A}\upgamma)(\Xi)\right| &= \left| \mathbb{E}h(\Xi) - \mathrm{PRM}(\bl)(h)\right|  = \left| \int_{M_p(E)} h\De \mathcal{L}(\Xi) - \int_{M_p(E)} h \De \mathrm{PRM}(\bl)\right|,
\end{split}
\end{align}

\subsection{General results on approximation by Poisson processes}
%At this juncture we wish 
We use an argument first made by \cite{Michel:1988}
to establish a bound on the total variation distance between a point process and an ``easier'' Poisson random measure.
\begin{thm} \label{t: Michel}
For each integer $n \ge 1$, let $I_1, \ldots, I_n$ be Bernoulli random variables with probability of success $P(I_i=1)=p_i \in (0,1)$. 
Let $E$ be a locally compact separable metric space and let
$X_1, \ldots, X_n$ be \iid $E$-valued random variables, independent of the $I_i$'s.
Moreover, let $\Xi = \sum_{i=1}^n I_i\delta_{X_i}$ and let $W = \sum_{i=1}^n  I_i$. Then,
\[
d_{TV}(\mathcal{L}(\Xi), \mathrm{PRM}(\mathbb{E}\Xi)) \le d_{TV}(\mathcal{L}(W), \mathrm{Poi}(\mathbb{E}W)). 
\]
\end{thm}
\begin{proof}
%We only briefly give a sketch of the proof. 
Let $Z_1, \ldots, Z_n$ be \iid random variables with distribution $\mathcal{L}(X_1)$, and let them be independent of $W$. 
Then the process $\sum_{j=1}^W \delta_{Z_j}$ has the same distribution as the process of interest $\Xi$. 
Furthermore, note that a $\mathrm{PRM}(\mathbb{E}\Xi)$ can be realised as $\sum_{j=1}^{W^\star} \delta_{Z_j}$, where $W^\star \sim \mathrm{Poi}(\mathbb{E}W)$ 
is independent of the $Z_j$'s. Then the result follows from this representation of the two processes, using the law of total probability.
%the Poisson approximation to the binomial distribution.}
%where the latter is at most $ P(X \in A)$ by Theorem \ref{t: Poster_univ}.
%%%%% OLD PROOF
%\begin{align*}
%& d_{TV}(\mathcal{L}(\Xi), \mathrm{PRM}(\mathbb{E}\Xi)) \\
%%& = \sup_{R} \left\{ P(\Xi \in R) - P(\mathrm{PRM}(\mathbb{E}\Xi) \in R) \right\}\\
%& =\sup_{R} \left\{ P\left(\sum_{j=1}^{W} \delta_{Z_j} \in R\right) - P\left(\sum_{j=1}^{W^\star} \delta_{Z_j} \in R\right) \right\}\\
%%& =\sup_{R} \left\{ \sum_{l=0}^n P\left( \sum_{j=1}^{W} \delta_{Z_j} \in R \, , \, W=l \right) 
%%- \sum_{l=0}^\infty P\left( \sum_{j=1}^{W^\star} \delta_{Z_j} \in R \, , \, W^\star=l \right)\right\}\\
%& = \sup_{R} \left\{ \sum_{l=0}^n P\left( \sum_{j=1}^{l} \delta_{Z_j} \in R \right)P\left( W=l \right) 
%- \sum_{l=0}^\infty P\left( \sum_{j=1}^{l} \delta_{Z_j} \in R \right)P\left( W^\star=l \right) \right\}\\
%& \le \sup_{R}\sum_{l=0}^n P\left( \sum_{j=1}^{l} \delta_{Z_j} \in R \right)\left\{P(W=l)-  P(W^\star=l) \right\}_+ \le \sum_{l=0}^n \left\{P(W=l)-  P(W^\star=l) \right\}_+,
%\end{align*}
%where $\{\, .\,\}_+ = \max(\, .\,,0)$. Now define $B_0 = \{l \in \{1, \ldots,n\}: \, P(W = l) > P(W^\star = l)\}$. Then 
%\begin{eqnarray*}
%&&\sum_{l=0}^n \left\{P(W=l)-  P(W^\star=l) \right\}_+
% = \sum_{l \in B_0} \left\{P(W=l)-  P(W^\star=l) \right\}\\
%&& = P(W \in B_0)-  P(W^\star \in B_0) \le \sup_{B \subseteq \mathbb{Z}_+} |P(W \in B)-  P(W^\star \in B)| =d_{TV}(\mathcal{L}(W), \mathrm{Poi}(\mathbb{E}W)).
%\end{eqnarray*}
%%The last inequality follows from total variation distance of Binomial and Poisson random variables by \cite{Barbour/Hall:1984}.
\end{proof}
With this theorem at hand, we are able to show
\begin{thm}\label{t: MyMichelMultivariate}
For each integer $n \ge 1$, let $\mathbf{X}_1, \ldots, \mathbf{X}_n$ be \iid copies of a $d$-dimensional random vector $\mathbf{X}$ with state space 
$E \subseteq \mathbb{R}^d$, where $d \ge 1$. For a fixed set $A \in \mathcal{E}$, define
$\Xi_A := \sum_{i=1}^n I_{\{ \mathbf{X}_i \in A\}} \delta_{\mathbf{X}_i}$ 
%be the marked point process of points in $A$ and let
and
$W_A := \sum_{i=1}^n I_{\{\mathbf{X}_i \in A\}}$
%denote the random number of points in $A$
. Then,
\[
d_{TV}(\mathcal{L}(\Xi_A), \mathrm{PRM}(\mathbb{E}\Xi_A)) \le d_{TV}(\mathcal{L}(W_A), \mathrm{Poi}(\mathbb{E}W_A)) \le P(\mathbf{X} \in A). 
\]
\end{thm}
\begin{proof}
%As before, we only give a quick sketch of the proof. We again rewrite $\Xi_A$ as $\sum_{i=1}^n I_{\{\mathbf{X}_i \in A\}}\delta_{\mathbf{X}'_i}$, where $\mathbf X'_i$ have common law $\mathcal{L}(\mathbf{X}|\mathbf{X} \in A)$ and are independent. The process $\sum_{i=1}^n I_{\{\mathbf{X}_i \in A\}}\delta_{\mathbf{X}'_i}$
%is distributed as $\sum_{j=1}^{W_A} \delta_{\mathbf{Z}'_j}$, where the $\mathbf{Z}'_j$ are independent, have common distribution $P_A$, and are independent of $W_A$. 
%Furthermore, note that a $\mathrm{PRM}(\mathbb{E}\Xi_A)$ can be realised as $\sum_{j=1}^{W^\star} \delta_{\mathbf{Z}'_j}$, where $W^\star \sim \mathrm{Poi}(\mathbb{E}W_A)$ is independent of
%$\mathbf{Z}'_j$. The proof of Theorem~\ref{t: Michel} gives that $d_{TV}(\mathcal{L}(\Xi_A), \mathrm{PRM}(\mathbb{E}\Xi_A)) \le d_{TV}(\mathcal{L}(W_A), \mathrm{Poi}(\mathbb{E}W_A))$. Finally, 
The first result is a direct application of Theorem \ref{t: Michel}. An application of Theorem 1 by \cite{Barbour/Hall:1984} yields the %result.
second inequality.
%Let $P_A = \mathcal{L}(\mathbf{X}|\mathbf{X} \in A)$ and 
%define an \iid random sample $\mathbf{X}'_1, \ldots, \mathbf{X}'_n$ with common distribution $P_A$ that is independent of the sample $\mathbf{X}_1, \ldots, \mathbf{X}_n$.
%Then the process $\sum_{i=1}^n I_{\{\mathbf{X}_i \in A\}}\delta_{\mathbf{X}'_i}$ has the same distribution as the process of interest $\Xi_A$.   
%Note that due to the independence of the samples $\mathbf{X}_1, \ldots, \mathbf{X}_n$ and $\mathbf{X}'_1, \ldots, \mathbf{X}'_n$, the process $\sum_{i=1}^n I_{\{\mathbf{X}_i \in A\}}\delta_{\mathbf{X}'_i}$
%is distributed as $\sum_{j=1}^{W_A} \delta_{\mathbf{Z}'_j}$, where the $\mathbf{Z}'_j$ are independent, have common distribution $P_A$, and are independent of $W_A$. 
%Furthermore, note that a $\mathrm{PRM}(\mathbb{E}\Xi_A)$ can be realised as $\sum_{j=1}^{W^\star} \delta_{\mathbf{Z}'_j}$, where $W^\star \sim \mathrm{Poi}(\mathbb{E}W_A)$ is independent of
%$\mathbf{Z}'_j$. 
%%We conclude by applying Theorem~\ref{t: Michel}.
%It then follows from the proof of Theorem~\ref{t: Michel} that $d_{TV}(\mathcal{L}(\Xi_A), \mathrm{PRM}(\mathbb{E}\Xi_A)) \le d_{TV}(\mathcal{L}(W_A), \mathrm{Poi}(\mathbb{E}W_A))$. Finally, an application of Theorem 1 by \cite{Barbour/Hall:1984} yields
%\[
%d_{TV}(\mathcal{L}(W_A), \mathrm{Poi}(\mathbb{E}W_A)) \le \frac{1-e^{-\mathbb{E}W_A}}{\mathbb{E}W_A} \sum_{i=1}^n P(X_i\in A)^2 \le P(X \in A).
%\]
\end{proof}
The next two Propositions are giving useful bounds on the $d_{TV}$ and $d_2$ distances respectively between point measures. For the second one, we will briefly recall some notions and introduce references to understand the results.
The first concept is that of {\it Palm density}, about which the reader will find a more thorough introduction in \citet[Chapter 10]{Kallenberg:1983}. Suppose $\Xi$ is a point process on $E$ with $\sigma$-finite mean measure $\bl$. For any $z \in E$, a point process $\Xi_z$ is called \textit{Palm process
associated with $\Xi$ at $z$} if, for any measurable function $f: E \times \overline{M}_p(E) \to \mathbb{R}_+$,
\begin{equation}\label{d: Palm_equality}
\mathbb{E} \left[ \int_E f(z, \Xi) \Xi(dz)\right] = \mathbb{E} \left[ \int_E f(z,\Xi_z) \bl(dz)\right].  
\end{equation}
Intuitively, a Palm distribution of a point process $\Xi$ at a prescribed location $z$ is the distribution of $\Xi$ conditional on the presence of a point of $\Xi$ at $z$. Note that a process $\Xi$ is a Poisson process if and only if
\[
\mathcal{L}(\Xi_z) = \mathcal{L}(\Xi + \delta_z) \textnormal{ $\bl$-a.s.}
\]
%To introduce the result we need the following pseudo-metric:
%\begin{equation}\label{eq:d_1}
%d_1'(\bl,\,\tilde \bl)=\sup_{k\in \mathcal K}\frac1{s_1(k)}\left|\int k \frac{\bl(dz)}{\bl(E)}-\int k \frac{\tilde\bl(dz)}{\tilde\bl(E)} \right|.
%\end{equation}
%We are now ready to state the two results: 
Finally, we recall the following useful results:
\begin{prop}{\citet[p.235]{Barbour_et_al.:1992}}\label{t: dTV_two_PRM}
Let $\bl$ and $\tilde{\bl}$ be two finite measures on $E$. Then
\[
d_{TV}\left(\mathrm{PRM}(\bl), \mathrm{PRM}(\tilde{\bl})\right) \le \int_E | \bl - \tilde{\bl}|(dz).
\]
\end{prop}
\begin{prop}{\citet[Theorem 1.5]{Xia:1995} (version for Poisson random measures)}\label{t: d2_two_PRM} Let $\bl,\,\tilde \bl$ be two finite measures on $E$, and let $\lambda=\bl(E) =\tilde\bl(E)$. Then it holds that
\begin{align*}
d_2\left(\mathrm{PRM}(\bl), \mathrm{PRM}(\tilde{\bl})\right)&\le(1-\mathrm{e}^{-\lambda})d_1(\bl, \tilde{\bl}).
\end{align*}
\end{prop}
\begin{remark}
We highlight that the pseudo-metric $d_1'$ used in \cite{Xia:1995} for the above result is equal to the $d_1$-metric in our setting where $\bl(E) =\tilde\bl(E)$. 
\end{remark}
%\begin{prop}{\citet[Theorem 3.6]{Barbour/Brown:1992}}\label{t: d2_two_PRM}
%Let $\bl$ and $\tilde{\bl}$ be two finite measures over $E$ such that $\bl(E)= \tilde{\bl}(E) = \lambda$. Then
%\[
%d_2\left(\mathrm{PRM}(\bl), \mathrm{PRM}(\tilde{\bl})\right) \le (1-\mathrm{e}^{-\lambda})(2-\mathrm{e}^{-\lambda}) d_1(\bl, \tilde{\bl}). 
%\]
%\end{prop}
\subsection{The bivariate Marshall-Olkin geometric distribution}\label{s: MO_Geo}
The bivariate Marshall-Olkin geometric distribution arises as a natural generalisation of the geometric distribution to two dimensions.
It was first introduced by \cite{Hawkes:1972} and later studied by \cite{Marshall/Olkin:1985} as the discrete counterpart to their bivariate exponential 
distribution, first derived by them in \cite{Marshall/Olkin:1967} using shock models. Limit distributions for maxima of i.i.d. Marshall-Olkin geometric
random pairs were established in \cite{Mitov/Nadarajah:2005} and \cite{Feidt_et_al:2010}. 

Underlying the Marshall-Olkin geometric distribution are Bernoulli trials. Suppose $S$ and $T$ are two Bernoulli random variables 
with joint probability mass function $P(S=i,T=j)=p_{ij}$, for all $i,j=0,1$, and let $(S_1,\,T_1),\,(S_2,\,T_2),\ldots$ be i.i.d. copies of $(S,\,T)$. Let $X_1$ and $X_2$ denote the numbers of $0$s before the first $1$ in the sequences $S_1, S_2, \dots$ and $T_1, T_2,\dots$, respectively. 
Obviously, $X_1$ and $X_2$ follow geometric distributions with failure probabilities
$q_1 := P(S=0) = p_{00} + p_{01}$ and $q_2:= P(T=0) = p_{00}+p_{10}$, respectively. 
Their joint probability mass function is given by
\begin{equation}\label{d: MO_pointprob}
 P(X_1=k, X_2=l)= 
 \left\{ \begin{array}{ll}
 p_{00}^k q_2^{l-k} (1-p_{00}/ q_2 - q_2 + p_{00} ) & \textrm{ for } k < l , \\
 p_{00}^k ( 1-q_1 - q_2 + p_{00})& \textrm{ for } k=l ,  \\
 p_{00}^l q_1^{k-l} (1-q_1 - p_{00}/q_1 + p_{00} ) & \textrm{ for } k>l ,
 \end{array} \right.
\end{equation}
for any $(k,l) \in \mathbb{Z}_+^2$, where $\mathbb{Z}_+=\{0,\,1,\,2,\,\ldots\}$. The distribution of $\mathbf{X}= (X_1,X_2)$ thus depends on three parameters: the two marginal failure probabilities $q_1$ and
$q_2$, as well as $p_{00}=P(S=0,T=0)$, the probability of joint failure.
%$p_{00} = P(X_1 >0, X_2>0)$, the probability of simultaneous failure of both components. 
We assume that $p_{00} \ge q_1 q_2$. 
We have
\begin{equation}\label{d: MO_survival}
P(X_1 \ge k, X_2 \ge l)
 = \left\{ \begin{array}{ll}
 p_{00}^k q_2^{l-k} & \textrm{ for } k < l , \\
 p_{00}^k & \textrm{ for } k=l , \\
 p_{00}^l q_1^{k-l} & \textrm{ for } k>l.
 \end{array} \right. 
 \end{equation}
%The survival copula $\hat{C}$ is given by a Marshall-Olkin copula $C_{\alpha, \beta}$ as defined in \cite{Nelsen:2006}, Section 3.1.1:
%$$
%\hat{C}(u,v) = C_{\alpha,\beta}(u,v),\; (u,v) \in (0,1)^2
%$$
%with parameters $\alpha, \beta \in [0,1]$ set to be
%\[
%\alpha =\frac{\log (p_{00}/q_1q_2)}{\log (1/q_1)}\,, \qquad \beta = \frac{\log (p_{00}/q_1q_2)}{\log(1/q_2)}\,
%\]
%To show this, we may proceed as in \cite[Section 3.1.1]{Nelsen:2006} for the Marshall-Olkin exponential distribution. 
%That is, rewrite (\ref{d: MO_survival})
%as
%\[
%\left( \frac{p_{00}}{q_2}\right)^k \left( \frac{p_{00}}{q_1}\right)^l \left( \frac{q_1 q_2}{p_{00}}\right)^{\max (k,l)} 
%= q_1^k q_2^l \min \left\{ \left(\frac{p_{00}}{q_1q_2}\right)^k, \left(\frac{p_{00}}{q_1q_2}\right)^l\right \}, 
%\]
%using $\max (k,l) = k+l - \min(k,l)$ and $p_{00} \ge q_1 q_2$. With $u = P(X_1 \ge k)= q_1^k$, $v= P(X_2 \ge l) = q_2^l$, and
%we have $(p_{00}/q_1q_2)^k = u^{-\alpha}$ and $(p_{00}/q_1q_2)^l = v^{-\beta}$, and  since $p_{00} \ge q_1q_2$ and $q_1,q_2 \ge p_{00}$. For $\alpha,\beta \in (0,1)$, the copulas in this family have full support, i.e. $[0,1]^2$. 
Note that if $p_{00} = q_1q_2$, the Marshall-Olkin geometric distribution corresponds to a bivariate distribution with independent geometric margins. 
\section{%Multivariate Marshall-Olkin geometric maxima approximation
Rates of convergence for maxima of geometric random variables}\label{sec:max}
%In this Section we would like to obtain a bivariate approximation of the maxima for n \iid copies of the vector $( X,\,Y)$ following the bivariate Marshall-Olkin geometric distribution. 
In this section we determine bounds on the error of the approximation, in the Kolmogorov distance, of the laws of maxima of univariate geometric random variables and bivariate Marshall-Olkin geometric random pairs by appropriate limit distributions. 
The one-dimensional case is treated for example by 
\cite{Nadarajah/Mitov:2002} where they see that, for $X_1, \ldots, X_n$ \iid geometric random variables with probability of success $p=p_n \in (0,1)$, if 
$p_n \to 0$ as $n \to \infty$
then, for all $x \in \mathbb{R}$,
\begin{equation}\label{eq:geo_gumbel}
\lim_{n\to+\infty}P\left( X_{(n)} \le \frac{\log n + x}{p_n}\right)= \mathrm{e}^{-\mathrm{e}^{-x}}.
\end{equation}
The following proposition investigates the rate of convergence of this limit result and suggests two improvements.
One way to reduce the error is to approximate by a discretised version of the Gumbel distribution, the other is to use different normalising constants.
\begin{prop}\label{t: max_geo_d_K}
For each integer $n \ge 1$, let $X_1, \ldots, X_n$ be \iid geometric random variables with success probability $p_n\in (0,1)$, 
failure probability $q_n=1-p_n$, probability mass function 
$P(X_1=k)=p_nq_n^k$ and survival function $\overline{F}(k)= q_n^{k+1}$, for any $k \in \mathbb{Z}_+=\left\{0,\,1,\,2,\ldots \right\}$. 
% Question: how to get rid of 1/q-factor?
Then:
\begin{enumerate}
\item[(a)] (Approximation by a discretised Gumbel distribution) For all $k \in \mathbb{Z}_+$ and for all $k^\star \in \mathbb{R}$ defined by $k^\star = -\log n + k \log(1/q_n)$,
% Question: how to get rid of 1/q-factor?
\begin{equation}\label{(a)Gumbel}
\left| P\left( X_{(n)} < \frac{\log n + k^\star}{\log(1/q_n)} \right) -  \mathrm{e}^{-\mathrm{e}^{-k^\star}}\right| \le \frac{\log n}{q_n n} + \frac{1}{n} =:\delta_{\mathrm{PoiAppr}}.
\end{equation}
\item[(b)] (Approximation by a Gumbel distribution) For all $x \in \mathbb{R}$,
\[
\left| P\left( X_{(n)} < \frac{\log n + x}{\log(1/q_n)} \right) -  \mathrm{e}^{-\mathrm{e}^{-x}}\right| \le \delta_{\mathrm{PoiAppr}} +\mathrm{e}^{-1}\log(1/q_n) =: 
\delta_{\mathrm{Cont}}.
\]
\item[(c)] (Using the normalising constants from \cite{Nadarajah/Mitov:2002})
\[
\left| P\left( X_{(n)} < \frac{\log n + x}{1-q_n} \right) -  \mathrm{e}^{-\mathrm{e}^{-x}}\right| \le \delta_{\mathrm{Cont}} +\frac{1-q_n}{2q_n}\left( \log^2 n +\mathrm{e}^{-1}\right). 
\]
\end{enumerate}
\end{prop}
\noindent Note that the failure probability $q_n$ need not vary with the sample size $n$ for approximation by a discretised Gumbel distribution. The error bound is sharp for any constant
$q_n \equiv q \in (0,1)$, showing clearly that it makes more sense to approximate a discrete distribution by another discrete distribution than by a continuous one, as there is
no need to add an extra error as in (b). 

The extra error in (b), $\mathrm{e}^{-1}\log(1/q_n)$, is the discretisation error that
arises when going from the Gumbel concentrated on the lattice of
points $k^\star$ to the continuous Gumbel distribution over 
$\mathbb{R}$. It dominates the overall error in (b) unless 
$q_n$ tends to $1$ fast enough as $n \to \infty$, that is, unless
$1-q_n = O(\log(n)/n)$, in which case the discretisation error is of 
the same order as the first error term from (a).

%Geometric random variables with probability of success $n^{-1}$ yield on average $n-1$ failures before the first success,
%which illustrates how much worse the approximation by a continuous than by a discretised Gumbel distribution is. 
The upper bound in Part (c) shows that the choice of normalising constants or more precisely, of the scaling by $p_n$ in \eqref{eq:geo_gumbel},
may not be the optimal one. In order for the approximation in (c) to be good we require $p_n = o(1/\log^2 n)$.
%-- an even stronger condition than the one needed for (b). 
Its being a stronger condition than the one for the asymptotic 
result from \eqref{eq:geo_gumbel} is justified by (c) also being a stronger result in the sense that it gives a uniform bound. Note that the error in (c) 
is of the same order as the error in (a) only if $1-q_n = O( 1/(n\log n))$.
\begin{proof} For ease of notation we omit the subscript $n$. 
\begin{enumerate}
\item[(a)] Let $A=[y, \infty)$ for any choice of $y \ge 0$. Then $P(X_1 \in A) = q^{\lceil y \rceil}$, and, setting $k:=\lceil y \rceil \in \mathbb{Z}_+$,
Theorem \ref{t: MyMichelMultivariate} gives
\begin{equation}\label{p: first_bound_geo}
\left| P\left( X_{(n)} < k \right) - \mathrm{e}^{-nq^k} \right| \le q^k. 
\end{equation}
With $k^\star \in \mathbb{R}$ chosen such that $k=(\log n + k^\star)/\log(1/q)$, we then have
\begin{equation}\label{p: Max_Geo_error}
\left| P\left( X_{(n)} < \frac{\log n+k^\star}{\log(1/q)}\right) - \mathrm{e}^{-\mathrm{e}^{-k^\star}} \right| \le \frac{\mathrm{e}^{-k^\star}}{n}. 
\end{equation}
In order to find a uniform bound for all $k \in \mathbb{Z}_+$, choose an auxiliary point defined as $x_0:=x_{0n}:= - \log \log n$. We then study the error term in two cases:
\begin{itemize}
\item $k$: $k^\star \ge x_0$. We have
$\exp{(-k^\star)}/n \le \exp{(-x_0)}/n = \log (n)/n$. 
\item $k$: $-\log n\le k^\star \le x_0$. We have $k^\star\le m^\star\le x_0$ where 
\[m^\star := \lfloor (\log n + x_0)/\log(1/q) \rfloor\log(1/q)-\log n.\]
%Also call $\lfloor y_0\rfloor :=m$. 
Since distribution functions are non-decreasing and using (\ref{p: Max_Geo_error}), 
\[
P\left( X_{(n)} < \frac{\log n+k^\star}{\log(1/q)}\right)\le P\left( X_{(n)} < \frac{\log n+m^\star}{\log(1/q)}\right)\le \frac{\mathrm{e}^{-m^\star}}{n}  + \mathrm{e}^{-\mathrm{e}^{-m^\star}}. 
\]
Therefore we obtain
\begin{align}\label{p: bound_xk_vs.xstar}
\left| P\left( X_{(n)} < \frac{\log n+k^\star}{\log(1/q)}\right) - \mathrm{e}^{-\mathrm{e}^{-k^\star}} \right|  &\le \frac{\mathrm{e}^{-m^\star}}{n}  + \mathrm{e}^{-\mathrm{e}^{-m^\star}}  
%= \mathrm{e}^{-m \log(1/q)} + \mathrm{e}^{-\mathrm{e}^{-m^\star}}
\nonumber \\
&\le \frac{\mathrm{e}^{-x_0}}{qn} + \mathrm{e}^{-\mathrm{e}^{-x_0}} %\le
=
\frac{\log n}{qn} + \frac{1}{n},
\end{align}
where we used $m^\star \ge x_0 -\log(1/q)$ in the last inequality. 
%See Figure \ref{f: Max_geo_dK} for a sketch.
\end{itemize}
\noindent (\ref{p: bound_xk_vs.xstar}) provides a bound for all $k^\star$,
and thus
\[
\left| P\left( X_{(n)} < \frac{\log n +k^\star}{\log(1/q)}\right) - \mathrm{e}^{-\mathrm{e}^{-k^\star}} \right| \le \frac{\log n}{qn} + \frac{1}{n}.
\]
\item[(b)] Let $x= \log(1/q)y-\log n$. 
By adding and subtracting $\exp\left\{-\mathrm{e}^{-x}\right\}$ 
into (\ref{(a)Gumbel}), and noting that,
since $y \le \lceil y \rceil = k$, we have $x \le k^\star$ and
$\exp\{-\mathrm{e}^{-k^\star}\}-\exp\{-\mathrm{e}^{-x}\} \ge 0$. We thus obtain
\[
\left| P\left( X_{(n)} < \frac{\log n +x}{\log(1/q)} \right)
- \mathrm{e}^{-\mathrm{e}^{-x}} \right| \le 
\delta_{\mathrm{PoiAppr}} + \mathrm{e}^{-\mathrm{e}^{-k^\star}} - \mathrm{e}^{-\mathrm{e}^{-x}},
\]
where 
\[
\mathrm{e}^{-\mathrm{e}^{-k^\star}} - \mathrm{e}^{-\mathrm{e}^{-x}}
\le \int_{x}^{k^\star}\mathrm{e}^{-t}\mathrm{e}^{-\mathrm{e}^{-t}}\De t \le \mathrm{e}^{-1}(k^\star - x)
= \mathrm{e}^{-1}\log(1/q)(\lceil y \rceil - y) \le \mathrm{e}^{-1}\log(1/q).
\]
\item[(c)]
From \eqref{p: Max_Geo_error} and using $x\le k^\star$ as defined in (b) one has
\begin{equation}\label{p: max_geo_cont}
\left| P\left( X_{(n)} < \frac{\log n +x}{ \log(1/q)}\right)
-\mathrm{e}^{-\mathrm{e}^{-x}}\right| \le
\frac{e^{-x}}{n} + \mathrm{e}^{-\mathrm{e}^{-k^\star}}-\mathrm{e}^{-\mathrm{e}^{-x}}\le \frac{\e^{-x}}{n}+\e^{-1}\log(1/q),  
\end{equation}
where we use the upper bound $\e^{-1}\log(1/q)  $ derived in (b).
Choose $x' \ge -\log n$ such that 
\[
y = \frac{\log n + x}{\log(1/q)} = \frac{\log n + x'}{1-q}.
\]
Then from \eqref{p: max_geo_cont} and observing that $x>x'$ (since $\log(1/q)>1-q$) 
\begin{align*}
\left| P\left( X_{(n)} < \frac{\log n + x'}{1-q}\right) - \mathrm{e}^{-\mathrm{e}^{-x}}\right| &\le \frac{\mathrm{e}^{-x'}}{n}  + \mathrm{e}^{-1}\log(1/q) 
\end{align*}
By adding and subtracting $\mathrm{e}^{-\mathrm{e}^{-x'}}$ in the above expression we get
\begin{align*}
&\left| P\left( X_{(n)} < \frac{\log n + x'}{1-q}\right) - \mathrm{e}^{-\mathrm{e}^{-x'}}\right| \le \frac{\mathrm{e}^{-x'}}{n}  + \mathrm{e}^{-1}\log(1/q) + \mathrm{e}^{-\mathrm{e}^{-x}} - \mathrm{e}^{-\mathrm{e}^{-x'}}.%\\
%&\le \frac{\mathrm{e}^{-x'}}{n}  + \mathrm{e}^{-1}\log(1/q) + %\mathrm{e}^{-\mathrm{e}^{-x}} - \mathrm{e}^{-\mathrm{e}^{-x'}}.
\end{align*}
For the latter error term we find
\begin{align*}
\mathrm{e}^{-\mathrm{e}^{-x}} - \mathrm{e}^{-\mathrm{e}^{-x'}} &= \mathrm{e}^{-\mathrm{e}^{-x}} \left[ 1-\mathrm{e}^{-(\mathrm{e}^{-x'} - \mathrm{e}^{-x})}\right] \le \mathrm{e}^{-x'}\left[ 1-\mathrm{e}^{-(x-x')}\right] \\
& \le \mathrm{e}^{-x'} (x-x') = \mathrm{e}^{-x'}\left[\log (1/q) - (1-q)\right]y, %\le \mathrm{e}^{-x'}\frac{(1-q)^2}{2q} \frac{\log n + x'}{1-q} = \mathrm{e}^{-x'} (\log n + x')\frac{1-q}{2q}
\end{align*}
where we used $\exp\{-\mathrm{e}^{-x}\} \le 1$ in the first inequality and $1-\mathrm{e}^{-z} \le z$ for $z \ge 0$ in both inequalities. Note that use of the definition of the logarithm
and the geometric series give
\begin{align*}
\log(1/q)-(1-q) &= \sum_{j=2}^\infty \frac{(1-q)^j}{j} \le \sum_{j=2}^\infty \frac{(1-q)^j}{2} 
= \frac{1}{2}\left[ \sum_{j=0}^\infty (1-q)^j - 1 - (1-q)\right]\\
&= \frac{1}{2} \left[ \frac{1}{q} -2 + q\right]= \frac{(1-q)^2}{2q}.
\end{align*}
Then, 
\[
\mathrm{e}^{-\mathrm{e}^{-x}} - \mathrm{e}^{-\mathrm{e}^{-x'}} \le  \frac{(1-q)^2y}{2q}\mathrm{e}^{-x'} = \frac{1-q}{2q} (\log n + x')\mathrm{e}^{-x'} \le \frac{1-q}{2q} (\mathrm{e}^{-x'}\log n  + \mathrm{e}^{-1}).
\]
Thus,  
\[
\left| P\left( X_{(n)} < \frac{\log n + x'}{1-q}\right)
- \mathrm{e}^{-\mathrm{e}^{-x'}}\right|
\le \frac{\mathrm{e}^{-x'}}{n}  + \mathrm{e}^{-1}\log(1/q) + 
\frac{1-q}{2q} \left(\mathrm{e}^{-x'}\log n  + \mathrm{e}^{-1}\right).
\]
For a uniform bound over all $x'$, we choose $x_0$ as before in (a) and (b), and obtain, with an analogous argument, the overall error bound
\[
\frac{\log n}{n}  + \mathrm{e}^{-1}\log(1/q) + \frac{1-q}{2q}\left(\log^2 n + \mathrm{e}^{-1}\right) + \frac{1}{n}.
\]  
\end{enumerate}
\end{proof}
\noindent
%We now would like to pass to the bivariate MO geometric distribution. We use the following normalisation of the random pair $\mathbf{Z} = (X,\,Y)$:
{
We now pass to the bivariate case where we assume that the random pair $\mathbf{Z} = (X,\,Y)$ follows the Marshall-Olkin geometric distribution as defined in Section \ref{s: MO_Geo}, taking values $(k,\ell)\in \mathbb{Z}_+^2$. We use the following normalisation:
\begin{equation}\label{d: normalisation_biv_MOgeo}
\mathbf{Z}^\star = (X^\star, Y^\star) = \left( \log\left(\frac1{p_{00}}\right) X - \log n\, ,\,  \log\left(\frac1{p_{00}}\right) Y - \log n\right),
\end{equation}
taking values $(k^\star, \ell^\star) \in [-\log n, \infty)^2$. Similarly to (a) in Proposition \ref{t: max_geo_d_K} above, the following proposition provides bounds on the error of the approximation, in the Kolmogorov distance, of the distribution of maxima of Marshall-Olkin geometric pairs by a discrete limit distribution.
\begin{prop}\label{p: Ale}
For each integer $n \ge  3$, let $\mathbf{Z}_1^\star, \ldots, \mathbf{Z}^\star_n$ be \iid copies of the random pair 
$\mathbf{Z}^\star = (X^\star,\,Y^\star)$ as %in Prop.~\ref{s: MO_Geo_dTV}. 
defined in (\ref{d: normalisation_biv_MOgeo}).
Moreover let $ X^\star_{(n)}=\max_{1\le i \le n}X^\star_{i}$ and similarly for $ Y^\star_{(n)}$. Then, for all $(k^\star,\,\ell^\star)\in [-\log n, \infty)^2$,
%\Z_+$
\begin{eqnarray*}
&&\left| P\left( X^\star_{(n)}< k^\star,\, Y^\star_{(n)}< \ell^\star\right)-H(k^\star,\,\ell^\star)\right|
{
\le \left(\frac{1}{n}\right)^{\frac{\log \min \{q_1,q_2 \}}{\log p_{00}}} \cdot \frac{\log n}{p_{00}} +\mathrm e^{-\sqrt{\log n}},}
%\\
%&&\le\max\left\{2,\, \frac{\log p_{00}}{ \log\left(p_{00}/ q_2\right)},\frac{\log p_{00}}{ \log\left(p_{00}/ q_1\right)}\right\}\frac{\sqrt{\log n}}{n^{\min\left\{\frac{\log q_2}{\log p_{00}},\,\frac{\log q_1}{\log p_{00}}\right\}}}+\mathrm{e}^{-\sqrt{\log n}}.
\end{eqnarray*}
where 
$$
H(x,\,y)=\left\{ \begin{array}{lr}
\mathrm{e}^{-\mathrm{e}^{-\frac{\log(p_{00}/q_2)}{\log p_{00}}\,x - \frac{\log q_2}{\log p_{00}}\,y}} & x<y,\\
\mathrm{e}^{-\mathrm{e}^{-x}} & x=y,\\
\mathrm{e}^{-\mathrm{e}^{-\frac{\log(p_{00}/q_1)}{\log p_{00}}\,y - \frac{\log q_1}{\log p_{00}}\,x}} & x>y,\\
                  %\mathrm{e}^{-\mathrm{e}^{-x}\mathrm{e}^{-\frac{\log q_2}{\log p_{00}}(y-x)}} & x<y,\\
                  %\mathrm{e}^{-\mathrm{e}^{-x}} & x=y,\\
                  %mathrm{e}^{-\mathrm{e}^{-y}\mathrm{e}^{-\frac{\log q_1}{\log p_{00}}(x-y)}} & x>y.
                  \end{array}
\right.
$$
\end{prop} 
We remark here that the 
%limiting function 
limit distribution $H$ is different from the ones obtained in \cite{Mitov/Nadarajah:2005}, who used different normalisations. The idea underlying the proof of Proposition~\ref{p: Ale} is to apply Theorem~\ref{t: MyMichelMultivariate} to $A=[k,\infty) \times[\ell,\infty)$ and $W_A=\sum_{i=1}^{n} I_{\{\mathbf{Z}_i \in A\}}$ as follows:
\[
 \left| P(W_A=0) - \mathrm e^{-nP(\mathbf{Z}\in A)}\right| \le P(\mathbf{Z} \in A).
\]
The limit distribution $H$ defined in Proposition~\ref{p: Ale} thus corresponds to $\mathrm e^{-nP(X^\star \ge k^\star, Y^\star \ge \ell^\star)}$.
%where the assumptions on the parameters $p_{00}$, $q_1$ and $q_2$ are differently varying with the sample size, whence their techniques were not applicable in our case. The proof follows closely that of the univariate case in Prop.~\ref{t: max_geo_d_K}.
\begin{proof}
\begin{enumerate}[1. ]
\item If $k^\star=\ell^\star$, we choose the auxiliary threshold 
%$$
$
x_0:=- \log \log n
$.
%$$
%If 
For all $k$ such that $k^\star\ge x_0$, 
%then 
Theorem~\ref{t: MyMichelMultivariate} gives 
%that
\begin{equation}\label{p: Max_Geo_error_2}
\left| P\left(  X^\star_{(n)}< k^\star,\, Y^\star_{(n)}< k^\star\right) - \mathrm{e}^{-\mathrm{e}^{-k^\star}} \right| \le 
{\frac{\mathrm{e}^{-k^\star}}{n} \le } 
\frac{\mathrm{e}^{-x_0}}{n}\le \frac{\log n}{n}. 
\end{equation}
For $k$ such that $k^\star \le x_0$, we may 
bound the error in (\ref{p: Max_Geo_error_2}) (i.e. the {absolute value in}
% the left-hand-side of
(\ref{p: Max_Geo_error_2}))
by further adding the limit distribution to the error {bound} at $m^\star$, where 
$m := \lfloor y_0 \rfloor $ and $y_0:= (\log n + x_0)/\log(1/p_{00}) $,
i.e. 
%\begin{align}
% \begin{split}
\[
\left| P\left(  X^\star_{(n)}< k^\star,\, Y^\star_{(n)}< k^\star\right) - \mathrm{e}^{-\mathrm{e}^{-k^\star}} \right|  
%&
\le \frac{\mathrm{e}^{-m^\star}}{n}  + \mathrm{e}^{-\mathrm{e}^{-m^\star}}  
%= \mathrm{e}^{-m \log(1/p_{00})} + \mathrm{e}^{-\mathrm{e}^{-m^\star}} \\
%&
\le \frac{\mathrm{e}^{-x_0}}{n\,p_{00}} + \mathrm{e}^{-\mathrm{e}^{-x_0}} \le \frac{\log n}{n\,p_{00}} + \frac{1}{n},
\]
%\end{split}
%\end{align}
where we used %$m \ge y_0 -1$ 
{$x_0 - \log(1/p_{00}) \le m^\star \le x_0$} in the second inequality. 
\item Suppose instead $k^\star<\ell^\star$ (the case $k^\star>\ell^\star$ is symmetric). 
By Theorem~\ref{t: MyMichelMultivariate} we have 
\begin{eqnarray}
&&\left|P\left(  X^\star_{(n)}< k^\star,\, Y^\star_{(n)}< \ell^\star\right) -
%\mathrm{e}^{-\mathrm{e}^{-k^\star}\mathrm{e}^{-\frac{\log q_2}{\log p_{00}}(\ell^\star-k^\star)}}
\mathrm{e}^{-\mathrm{e}^{-\frac{\log(p_{00}/q_2)}{\log p_{00}}\,k^\star - \frac{\log q_2}{\log p_{00}}\,\ell^\star}}
\right|
\le \frac{1}{n}\,
\mathrm{e}^{-\frac{\log(p_{00}/q_2)}{\log p_{00}}\,k^\star - \frac{\log q_2}{\log p_{00}}\,\ell^\star}
%\mathrm{e}^{-k^\star}\mathrm{e}^{-\frac{\log q_2}{\log p_{00}}(\ell^\star-k^\star)} 
.\qquad\qquad\label{eq:ineq}
\end{eqnarray}
We first assume
\begin{equation}\label{eq:cond_p00}
p_{00}\ge q_2^2.
\end{equation}
Choose the following auxiliary points:
$$
x_0:=-\frac{\log p_{00}}{ 2\log\left(p_{00}/ q_2\right)}\log\log n,\quad y_0:=-\frac{\log p_{00}}{2\log q_2}\log\log n.
$$
Note that, under \eqref{eq:cond_p00}, $x_0\le y_0$. Hence:
\begin{enumerate}
\item If $k^\star\ge x_0$, $\ell^\star\ge y_0$, then
$$
%\frac{1}{n}\mathrm{e}^{-k^\star}\mathrm{e}^{-\frac{\log q_2}{\log p_{00}}(\ell^\star-k^\star)} 
\frac{1}{n}\,
\mathrm{e}^{-\frac{\log(p_{00}/q_2)}{\log p_{00}}\,k^\star - \frac{\log q_2}{\log p_{00}}\,\ell^\star}
\le \frac{\log n}{n}.
$$
\item If $k^\star\le x_0$, $\ell^\star\le y_0$, then the error in (\ref{eq:ineq}) may be bounded by further adding the limit distribution to the error {bound }at 
%then we may add to the error in the RHS of \eqref{eq:ineq} the limit distribution at the point 
$(\kappa^\star,\,\lambda^\star)$, with
$$
\kappa:=\left\lfloor\frac{x_0+\log n}{\log\left(1/p_{00}\right)} \right\rfloor,\quad \lambda:=\left\lfloor\frac{y_0+\log n}{\log\left(1/p_{00}\right)}\right\rfloor,
$$
%We note that 
%\begin{equation}\label{eq:aux}
%\begin{array}{c}
%x_0-\log\frac1{p_{00}}\le \kappa^\star\le x_0,\\
%y_0-\log\frac1{p_{00}}\le \lambda^\star\le y_0.
%\end{array}
%\end{equation}
%Hence we can proceed in this way: first observe that
%\begin{eqnarray*}
%&&\mathrm{e}^{-\mathrm{e}^{-\frac{\log (p_{00}/q_2)}{\log p_{00}}k^\ast}-\mathrm{e}^{-\frac{\log q_2}{\log p_{00}}\ell^\star}}=\mathrm{e}^{-\mathrm{e}^{\log\log n}}= 1/n.
%\end{eqnarray*}
%Also, using \eqref{eq:aux} in the first inequality below,
%\marginpar{Your suggestion was correct. Changed accordignly}
% {\color{red} Here, using $x_0 -\log(1/p_{00}) \le \kappa^\star \le x_0$ and same for $\lambda^\star$, I get $(\log n)/n$ below, which I guess can't be right as it is too good - what do you get when you use this instead of going via $\kappa$ and $\lambda$?}
%\begin{eqnarray*}
%&&\frac{1}{n}\mathrm{e}^{-\frac{\log (p_{00}/q_2)}{\log p_{00}}k^\star}\mathrm{e}^{-\frac{\log q_2}{\log p_{00}}\ell^\star} \le\frac1n\mathrm{e}^{-\frac{\log (p_{00}/q_2)}{\log p_{00}}\kappa^\star}\mathrm{e}^{-\frac{\log q_2}{\log p_{00}}\lambda^\star}\\
%&&\le \frac1n\mathrm{e}^{-\frac{\log (p_{00}/q_2)}{\log p_{00}}x_0-\log (p_{00}/q_2)-\frac{\log q_2}{\log p_{00}}y_0-\log q_2    }=\frac{\log n}{n p_{00}}.
%\end{eqnarray*}
{
i.e. the error in \eqref{eq:ineq} is bounded by
\begin{eqnarray*}
&&\frac{1}{n}\,\mathrm{e}^{-\frac{\log (p_{00}/q_2)}{\log p_{00}}\kappa^\star -\frac{\log q_2}{\log p_{00}}\lambda^\star}
+ \mathrm{e}^{-\mathrm{e}^{-\frac{\log (p_{00}/q_2)}{\log p_{00}}\kappa^\star-\frac{\log q_2}{\log p_{00}}\lambda^\star}}\\
&&\qquad \le \frac{1}{n}\,\mathrm{e}^{-\frac{\log (p_{00}/q_2)}{\log p_{00}}x_0-\frac{\log q_2}{\log p_{00}}y_0 - \log p_{00}}
+ \mathrm{e}^{-\mathrm{e}^{-\frac{\log (p_{00}/q_2)}{\log p_{00}}x_0-\frac{\log q_2}{\log p_{00}}y_0}}
=\frac{\log n}{n p_{00}} + \frac{1}{n}\,,
\end{eqnarray*}
where we used $x_0 - \log(1/p_{00}) \le \kappa^\star \le x_0$ and $y_0-\log(1/p_{00}) \le \lambda^\star \le y_0$ in the inequality.
}
%Overall in this case the error in
%\eqref{eq:ineq} is bounded by
%$$ \frac{1}{p_{00}}\frac{\log n}{n}+\frac{1}{n}.$$
%Here we have used \eqref{eq:aux} for the upper bound, and then plugged in the values of $x_0$ and $y_0$.
\item If $k^\star\ge x_0$, $\ell^\star\le y_0$, then we may 
%add to the error in the RHS of 
bound the error in \eqref{eq:ineq}
by further adding
the limit distribution to the error {bound}
at the point $(\lambda^\star,\,\lambda^\star)$.
%as in Case 2. 
This yields an error bound of
\begin{eqnarray*}
&&{
\frac{1}{n}\,\mathrm{e}^{-\frac{\log (p_{00}/q_2)}{\log p_{00}}\lambda^\star -\frac{\log q_2}{\log p_{00}}\lambda^\star}
+ \mathrm{e}^{-\mathrm{e}^{-\frac{\log (p_{00}/q_2)}{\log p_{00}}\lambda^\star-\frac{\log q_2}{\log p_{00}}\lambda^\star}}
%\frac{1}{n}\mathrm{e}^{-\lambda^\star}\mathrm{e}^{-\frac{\log q_2}{\log p_{00}}(\lambda^\star-\lambda^\star)}+\mathrm{e}^{-\mathrm{e}^{-\lambda^\star\frac{\log q_2}{\log p_{00}}}-\mathrm{e}^{-\lambda^\star\frac{\log \left(p_{00}/q_2\right)}{\log p_{00}}}}\le 
 = \frac{1}{n}\,\mathrm{e}^{-\lambda^\star}+\mathrm{e}^{-\mathrm{e}^{-\lambda^\star}}
 }\\
&&\le \frac1n\, \e^{-y_0+\log (1/{p_{00}})}+\e^{-\e^{-y_0}}=\frac{ \left(  \log n\right)^{  \frac{ \log p_{00}   }{  2\log q_2  }  }   }{        np_{00}}+\e^{- \left(  \log n\right)^{  \frac{ \log p_{00}   }{  2\log q_2  } }}\le \frac{\log n}{np_{00}}+\e^{-\sqrt{\log n}},
\end{eqnarray*}
where we {used $y_0-\log(1/p_{00}) \le \lambda^\star \le y_0$, and}
%have used the fact that 
$\frac12\le\left({\log p_{00}}\right)/\left({2\log q_{2}}\right)\le 1$ since $q_2^2\le p_{00}\le q_2$.
\item 
If $k^\star\le x_0$, $\ell^\star\ge y_0$, the error bound is given by
\begin{eqnarray*}
&&\frac1n\,\e^{-\frac{\log\left(p_{00}/q_2 \right)     }{ \log p_{00}      }k^\star-\frac{\log q_2    }{ \log p_{00}      }\ell^\star}
\le \frac1n\,\e^{\frac{\log\left(p_{00}/q_2 \right)     }{ \log p_{00}      }\log n-\frac{\log q_2    }{ \log p_{00}      }y_0}=\frac{\sqrt{\log n}}{n^{\frac{\log q_2    }{ \log p_{00}      }}}\le \sqrt{\frac{\log n}{n}}\,,
\end{eqnarray*}
since $k^\star\ge -\log n$ and $\frac{\log q_2}{\log p_{00}}\ge \frac12$ by \eqref{eq:cond_p00}.
\end{enumerate}
{Therefore, if \eqref{eq:cond_p00} holds, an overall bound on the error in \eqref{eq:ineq} is given by $\frac{\log n}{np_{00}} + e^{-\sqrt{\log n}}$.}
If instead $p_{00}\le q_2^2$ (hence $x_0\ge y_0$), we consider only the cases (a), (b), (d) as above (since we must have $k^\ast < \ell^\star$). Hence
\begin{enumerate}[i. ]
\item If $k^\star\ge x_0$, $\ell^\star\ge x_0$, then the error bound is given by
\begin{eqnarray*}
&&\frac{1}{n}\,\mathrm{e}^{-\frac{\log\left(p_{00}/q_2 \right)}{ \log p_{00}     }k^\star-\frac{\log q_2}{ \log p_{00}     }\ell^\star}
\le\frac{1}{n}\,  \mathrm{e}^{-\frac{\log\left(p_{00}/q_2 \right)}{ \log p_{00}     }x_0-\frac{\log q_2}{ \log p_{00}     }x_0}
=\frac{\e^{-x_0}}{n}
%=\frac{(\log n)^{\frac{\log p_{00}}{2{\log\left(p_{00}/q_2\right)}}}}{n}
\le\frac{\log n}{n}\,.
\end{eqnarray*}
Here we have used the fact that 
$p_{00}\le q_2^2$ 
implies $\frac{\log p_{00}}{2{\log\left(p_{00}/q_2\right)}}\le 1$.
\item If $k^\star\le x_0$, $\ell^\star\le x_0$, then, %as before 
choosing $\kappa:=\lambda:=
\left\lfloor\frac{x_0+\log n}{\log\left(1/p_{00}\right)} \right\rfloor$,
{the error in \eqref{eq:ineq} may be bounded by}
%one obtains
\begin{eqnarray*}
&&\frac1n\, \e^{-\frac{\log \left(p_{00}/q_2\right)}{\log p_{00}}\kappa^\star-\frac{\log q_2}{\log p_{00}}\kappa^\star}+\e^{-\e^{-\frac{\log \left(p_{00}/q_2\right)}{\log p_{00}}\kappa^\star- \frac{\log q_2}{\log p_{00}}\kappa^\star   }}=\frac{\e^{-\kappa^\star}}{n}+\e^{-\e^{-\kappa^\star}}\\
&&\le \frac{\e^{-x_0}}{n p_{00}}+\e^{-\e^{-x_0}}=\frac{(\log n)^{\frac{\log p_{00}}{2{\log\left(p_{00}/q_2\right)}}}}{n p_{00}}+\e^{-\frac{\log p_{00}}{2{\log\left(p_{00}/q_2\right)}}}
\le \frac{\log n}{n p_{00}}+\e^{-\sqrt{\log n}},
\end{eqnarray*}
%employing \eqref{eq:aux} 
where we used $x_0 - \log(1/p_{00}) \le \kappa^\star \le x_0$ in the first inequality and $q_2^2\ge p_{00}$ in the %last inequality.
{second.}
%Overall in this case the RHS of \eqref{eq:ineq} is bounded by
%$$ \frac{q_2}{p_{00}}\frac{\log n}{n}+\frac{1}{n}\le \frac{\log n}{n}+\frac{1}{n}.$$
\item If $k^\star\le x_0$, $\ell^\star\ge x_0$, as in %the previous case 4. 
{(d)} we obtain
\begin{eqnarray*}
&&\frac{1}{n}\,\mathrm{e}^{-\frac{\log\left(p_{00}/q_2 \right)}{ \log p_{00}     }k^\star-\frac{\log q_2}{ \log p_{00}     }\ell^\star}
\le\frac{1}{n}\,  \mathrm{e}^{\frac{\log\left(p_{00}/q_2 \right)}{ \log p_{00}     }\log n-\frac{\log q_2}{ \log p_{00}     }x_0}
\le \frac{\mathrm{e}^{-\frac{\log q_2}{ \log p_{00}}x_0}}{n^{\frac{\log q_2}{ \log p_{00}   }}}
=\frac{\sqrt{\log n}}{n^{\frac{\log q_2}{ \log p_{00}   }}}\,.
\end{eqnarray*}
\end{enumerate}
{
Thus, if $p_{00} \le q_2^2$, an overall bound on the error in \eqref{eq:ineq} is given by
%in the second case the error bound of \eqref{eq:ineq} is $
$\frac{\log n}{n^{\frac{\log q_2}{ \log p_{00}}}p_{00}}+\mathrm{e}^{-\sqrt{\log n}}\,.
$}
\end{enumerate}
\end{proof}
\section{%MPPE's with bivariate Marshall-Olkin geometric marks
Rates of convergence for MPPEs with geometric marks}\label{Sec: MO_Geo}
\subsection{Univariate geometric marks
%case
}\label{s: MPPE_geo}
In Proposition \ref{t: max_geo_d_K} we demonstrated that for maxima of geometric random variables the approximation by a discretised Gumbel distribution living on
lattice points $k^\star$ gives a smaller error than the approximation by a continuous Gumbel distribution on $\mathbb{R}$. For the latter approximation to be sharp,
we need the condition that the failure probability $q_n$ depends on $n$ in such a way that $1-q_n = o(1/\log n)$ for $n \to \infty$. We encounter a similar behaviour 
when approximating an MPPE with geometric marks, defined by 
$\Xi_A$ $:=$ $\sum_{i=1}^n$ $I_{\{X_i \in A\}}\delta_{X_i}$,
by a Poisson process. The set $A \in \mathcal{B}([0, \infty)^2)$ will, in all further applications, be chosen such that points falling into $A$ can be considered ``extreme". We consider here the MPPE 
\begin{equation}\label{Ferrero Rocher}
\Xi^\star_{A^\star} := \sum_{i=1}^n I_{\{X^\star_i \in A^\star\}}\delta_{X^\star_i},
\end{equation}
the normalised version of (\ref{Ferrero Rocher}) where the marks are subject to the normalisation used in Proposition \ref{t: max_geo_d_K} (a). Proposition \ref{t: Geo_MPPE_dTV} below gives the error in total variation of the 
approximation of the law of $\Xi^\star_{A^\star}$ by a Poisson process with mean measure living on the lattice $E^\star$ of normalised points $k^\star$. On the other hand, Proposition \ref{t: Geo_MPPE_d2} 
determines the error of the approximation by a Poisson process with an easy-to-use continuous mean measure, and uses the $d_2$-metric to achieve this.
%.........................................................................................................
\begin{prop}\label{t: Geo_MPPE_dTV}
For each integer $n \ge 1$, let $X_1, \ldots, X_n$ be \iid geometric random variables 
with failure probability $q\in (0,1)$ and $P(X_1 \ge y)= q^{\lceil y \rceil}$, for any $y \ge 0$. 
Define the normalised random variables $X_i^\star = \log(1/q) X_i - \log n$, $i = 1, \ldots, n$, taking values in 
$E^\star=\log(1/q)\mathbb{Z}_+ -\log n$. 
Let $A^\star = [u^\star, \infty)$ for any choice of $u^\star \in [-\log n, \infty)$, 
and let $\Xi^\star_{A^\star}$
be defined as in (\ref{Ferrero Rocher}). Then the mean measure of $\Xi^\star_{A^\star}$ is given by 
\begin{equation}\label{t: mean_meas_normalised_MPPE}
\bpis(B^\star) = \sum_{k^\star \in A^\star \cap E^\star \cap B^\star}(1-q)\mathrm{e}^{-k^\star}, \quad   \textit{ for any } B^\star \in \mathcal{B}([-\log n, \infty)), 
\end{equation}
and
\[
d_{TV}\left(\mathcal{L}(\Xi^\star_{A^\star}), \mathrm{PRM}(\bpis)\right) \le \frac{\mathrm{e}^{- u^\star}}{n}\,. 
\] 
\end{prop}
\begin{proof}
For all $k \in \mathbb{Z}_+$, we use the normalisation $k=(k^\star + \log n)/\log(1/q)$, where 
$k^\star$ $ \in E^\star$. We then have
%\begin{equation}\label{p: point_prob_geo}
$P(X_1 =k)$ $=$ $(1-q)q^k$ $=$ $(1-q)\, \frac{\mathrm{e}^{-k^\star}}{n}$ $=$ $P(X^\star_1 = k^\star)$, 
%\end{equation}
and, for any $B^\star \in \mathcal{B}([-\log n, \infty))$, 
\begin{equation}\label{p: mean_meas_geo_discrete}
\bpis(B^\star) = nP(X^\star \in A^\star \cap B^\star) = \sum_{k^\star \in A^\star \cap E^\star \cap B^\star} nP(X_1^\star = k^\star) =
\sum_{k^\star \in A^\star \cap E^\star \cap B^\star} (1-q)\mathrm{e}^{-k^\star}.
\end{equation}
Using Theorem \ref{t: MyMichelMultivariate}, we obtain 
\[
d_{TV}(\mathcal{L}(\Xi^\star_{A^\star}), \mathrm{PRM}(\bpis)) \le P(X^\star_1 \ge u^\star) = 
P\left(X_1 \ge  \frac{u^\star + \log n}{\log(1/q)}  \right) 
= q^{\left \lceil \frac{u^\star + \log n}{\log(1/q)} \right \rceil} \le \frac{\mathrm{e}^{-u^\star}}{n}\,.
\]

\end{proof}
%.........................................................................................................
The following proposition now uses the $d_2$-metric to approximate the MPPE with geometric marks by a Poisson process with continuous intensity,
as the total variation metric is too strong to achieve this. The continuous intensity measure we aim for is the same as that of an MPPE with 
exponential marks. The result is achieved in two steps: we first estimate the error in the $d_2$-distance of the approximation 
by a Poisson process with mean measure given by (\ref{t: mean_meas_normalised_MPPE}), and then compare this Poisson process by another one with the desired continuous
mean measure, again in the $d_2$-distance, by making use of Proposition \ref{t: d2_two_PRM}.
We assume here that
$d_0$ is the Euclidean distance on $\mathbb{R}$ bounded by $1$, i.e. $d_0(z_1, z_2) = \min(|z_1-z_2|, 1)$ for any $z_1, z_2 \in \mathbb{R}$, 
and define the $d_1$- and $d_2$-distances as in (\ref{d: def_d1}) and (\ref{d: def_d2}), respectively,
in Subsection \ref{subsec:dist_meas}.

%.........................................................................................................
\begin{prop}\label{t: Geo_MPPE_d2}
%(Geometric distribution) 
For each integer $n \ge 1$, let $X_i$, $X_i^\star$, $i=1, \ldots, n$, and $E^\star$ be defined as in Proposition
\ref{t: Geo_MPPE_dTV}. Let $A^\star = [u^\star, \infty)$ for any choice of $u^\star \in E^\star$, let $\Xi^\star_{A^\star}$ be defined as in 
(\ref{Ferrero Rocher})
%(\ref{d: normalised_MPPE}),
with mean measure $\bpis$ as in (\ref{t: mean_meas_normalised_MPPE}),
and define the continuous measure
%\[
$\bls(B^\star)$ $=$ $\int_{A^\star \cap B^\star} \mathrm{e}^{-x} \De x$, %\quad   \textit{ for any } 
for any 
$B^\star \in \mathcal{B}([-\log n, \infty))$.
%\]
Then %$\bl^\star([-\log n, \infty)) = \tilde{\bl}^\star([-\log n, \infty)) = \mathrm{e}^{-u^\star}$, and 
\[
d_2\left(\mathcal{L}(\Xi^\star_{A^\star}), \mathrm{PRM}(\bls)\right) 
 %WRONG: \phantom{bla} \le  \frac{1-q}{1+q} \,\left\{ \frac{\mathrm{e}^{-2u^\star}}{n}\wedge\frac{2\mathrm{e}^{-u^\star}}{n}\left( 1+ 2 \log^+ \left( \frac{\mathrm{e}^{-u^\star}}{2}\right) \right) \right\} 
\le \frac{\mathrm{e}^{-u^\star}}{n} +  \min\left\{\log\left(1/q\right),\, 1\right\}. 
\]
\end{prop}
\begin{proof}
We have
\[
d_2\left(\mathcal{L}(\Xi^\star_{A^\star}), \mathrm{PRM}(\bls)\right) 
\le d_2\left(\mathcal{L}(\Xi^\star_{A^\star}), \mathrm{PRM}(\bpis)\right) + d_2\left(\mathrm{PRM}(\bpis), \mathrm{PRM}(\bls)\right),
\]
where, by Proposition \ref{t: Geo_MPPE_dTV},
\[
 d_2\left(\mathcal{L}(\Xi^\star_{A^\star}), \mathrm{PRM}(\bpis\right) \le  d_{TV}\left(\mathcal{L}(\Xi^\star_{A^\star}), \mathrm{PRM}(\bpis)\right)
\le \frac{\mathrm{e}^{-u^\star}}{n}\,. 
\]
It thus remains to determine an estimate of $d_2\left(\mathrm{PRM}(\bpis), \mathrm{PRM}(\bls)\right)$. Since 
%\[
$\bls (A^\star)$ $=$ $\int_{u^\star}^\infty \mathrm{e}^{-x} \De x$ $=$ $\mathrm{e}^{-u^\star}$ $=$ $\bpis(A^\star)$, 
%\]
Proposition \ref{t: d2_two_PRM} gives
\begin{equation}\label{p: d2_to_d1_geo}
d_2\left(\mathrm{PRM}(\bpis), \mathrm{PRM}(\bls)\right) 
\le \left(1-\mathrm{e}^{-\mathrm{e}^{-u^\star}}\right)d_1(\bpis, \bls) \le d_1(\bpis, \bls). 
\end{equation}
By Definition (\ref{d: def_d1}) of the $d_1$-distance,
\begin{equation}\label{p: d_1_geo}
d_1(\bpis, \bls) = \mathrm{e}^{u^\star}\, \sup_{\kappa \in \mathcal{K}}\,  \, 
\left| \int_{-\log n}^\infty \kappa(x) \bpis(\De x) - \int_{-\log n}^\infty \kappa(x) \bls(\De x)\right|. 
\end{equation}
We may write the two integrals in the above expression as a sum of integrals over the ``normalised unit intervals'' 
$[k^\star, (k+1)^\star) = [k^\star, k^\star + \log(1/q))$, for all $k^\star \in E^\star \cap [u^\star, \infty)$. The modulus then
equals
\begin{equation}\label{p: mod_geo_d1}
\left| \sum_{k^\star \ge u^\star} \left\{ \int_{k^\star}^{k^\star + \log(1/q)} \kappa(x) \bpis(\De x) 
- \int_{k^\star}^{k^\star + \log(1/q)} \kappa(x) \bls(\De x)
\right\} \right| .
\end{equation}
Since $\bpis$ is concentrated on the lattice points $k^\star \in E^\star \cap [u^\star, \infty)$, we have
\[
\int_{k^\star}^{k^\star + \log(1/q)} \kappa(x) \bpis(\De x) 
= \kappa(k^\star) \bpis(\{k^\star\}) = \kappa(k^\star) (1-q)\mathrm{e}^{-k^\star}.
\]
Note that we obtain the same result by computing
\[
 \int_{k^\star}^{k^\star + \log(1/q)} \kappa(k^\star) \bls(\De x) 
= \kappa(k^\star)  \int_{k^\star}^{k^\star + \log(1/q)} \mathrm{e}^{-x} \De x 
= \kappa(k^\star) (1-q)\mathrm{e}^{-k^\star}.
\]
We may thus express (\ref{p: mod_geo_d1}) as follows:
\[
\left| \sum_{k^\star \ge u^\star}  \int_{k^\star}^{k^\star + \log(1/q)} 
\left\{ \kappa(k^\star) - \kappa(x) \right\}\bls(\De x) 
\right|
\le  \sum_{k^\star \ge u^\star} \int_{k^\star}^{k^\star + \log(1/q)}
\left|  \kappa(k^\star) - \kappa(x) \right| \bls(\De x). 
\]
From Lipschitz continuity of $\kappa$, we know that $|\kappa(k^\star) - \kappa(x)| \le s_1(\kappa) d_0(k^\star, x)$
for any $x \in [k^\star, k^\star + \log(1/q))$, where $k^\star \in E^\star \cap [u^\star, \infty)$. The maximum Euclidean distance between $k^\star$ and any point in 
$[k^\star, k^\star + \log(1/q))$ is of course given by $\log(1/q)$. Since we bound $d_0$ by $1$, we have
\[
\left|  \kappa(k^\star) - \kappa(x) \right|  \le s_1(\kappa) \min\left\{\log(1/q), \,1 \right\}. 
\]
For the $d_1$-distance in (\ref{p: d_1_geo}) we now find, using $\bls([u^\star, \infty)) = \mathrm{e}^{-u^\star}$,
\begin{align*}
d_1(\bpis, \bls) 
\le \mathrm{e}^{u^\star} \sum_{k^\star \ge u^\star} \int_{k^\star}^{k^\star + \log(1/q)} \min\left\{ \log(1/q),\, 1\right\} \bls(\De y) 
 = \min\left\{\log(1/q),\,  1\right\},
\end{align*}
which we plug into (\ref{p: d2_to_d1_geo}) to obtain an estimate for $d_2(\mathrm{PRM}(\bpis), \mathrm{PRM}(\bls)) $.
\end{proof}
The approximation of $\mathcal{L}(\Xi^\star_{\As})$ by $\mathrm{PRM}(\bls)$, whose continuous intensity function $\mathrm{e}^{-x}$ corresponds to that of MPPEs with exponential marks, gives rise to an additional error term which depends only on the failure probability of the geometric distribution. The error will %still
become small only if we allow the failure probability $q=q_n$ to tend to $1$ as $n \to \infty$. Since $\log(1/q_n)$ is the length of the normalised unit intervals, this condition causes the lattice structure to melt into the whole real subset $[-\log n, \infty)$ as $n \to \infty$.  Note that Proposition \ref{t: Geo_MPPE_d2} does not require $q_n$ to vary at a particular rate. The reason for that is that we chose the threshold $u_n^\star$ as element of the lattice $E^\star$. If we had not done so, we would have obtained an additional error term of size $\log(1/q_n)\mathrm{e}^{-u_n^\star}$. In this case, $q_n$ would have 
needed to vary at a fast enough rate to guarantee a small error despite the factor $\mathrm{e}^{-u_n^\star}$, which roughly corresponds to the expected number of exceedances and should thus be greater than $1$. We refer to Section \ref{s: MO_Geo_d2}, where we established the error estimate in full detail for MPPEs with bivariate geometric marks.  
%.........................................................................................................
%%%%%%%%%%%%%%%%%%%%%%%%%%%%%%%%%%%%%%%%%%%%%%%%%%%%%%%%%%%%%%%%%%%%%%%%%%%%%%%%%%%%%%%%%%%%%%%%%%%%%%%%%%
\subsection{Bivariate Marshall-Olkin geometric marks
%case
}
\subsubsection{Approximation in $d_{TV}$ by a Poisson process on a lattice}\label{s: MO_Geo_dTV}
For any integer $n \ge 1$, let $\mathbf{X}_1, \ldots, \mathbf{X}_n$ be \iid copies of the random pair 
$\mathbf{X} = (X_1,X_2)$, which follows the Marshall-Olkin geometric distribution 
from Section \ref{s: MO_Geo}.
%and takes values in $\mathbb{Z}_+^2 \subset [0,\infty)^2$. Let $A \in \mathcal{B}([0,\infty)^2)$
%Let $A \in \mathcal{B}([0, \infty)^2)$. We consider the MPPE
%$
%\Xi_A = \sum_{i=1}^n I_{\{\textbf{X}_i\in A\}} \delta_{\mathbf{X}_i},  
%$
%which lives on the lattice $\mathbb{Z}_+^2$. 
We %introduce 
use the following normalisation for studying joint threshold exceedances, introduced in Section \ref{sec:max} above:
% The following normalisation is the Marshall-Olkin geometric counterpart to the normalisation used in Section \ref{s: MO_Exp_joint}
% for studying joint threshold exceedances of Marshall-Olkin exponential marks:
\begin{equation}\label{d: MO_Geo_normalisation_joint}
(k^\star, l^\star) = \left(k \log (1/p_{00}) - \log n \,,\, l \log (1/p_{00}) - \log n\right), \quad \text{for any } (k,l) \in \mathbb{Z}_+^2,  
\end{equation}
and denote by 
%On the space 
$E^\star$ 
the lattice of normalised points $(k^\star, l^\star)$.
%rescaled as above we study the process 
%\begin{equation}\label{d: normalised_MPPE}
%\Xi^\star_{A^\star} := \sum_{i=1}^n I_{\{ X_i^\star \in A^\star\}} \delta_{X^\star_i}.
%\end{equation}
%Under this normalisation, $\Xi_A$ corresponds to 
%\begin{equation}\label{d: MO_Geo_MPPE}
%\Xi^\star_{A^\star} =\sum_{i=1}^n I_{\left\{\textbf{X}^\star_i \in A^\star\right\}} \delta_{\mathbf{X}_i^\star},
%\end{equation}
%which lives on the lattice $E^\star$ of normalised points $(k^\star, l^\star)$. Note that $E^\star \subset [-\log n, \infty)^2$. Furthermore, denote by
%\begin{equation}\label{d: MO_Geo_mean}
%W^\star_{A^\star}= \sum_{i=1}^n I_{\left\{\textbf{X}^\star_i \in A^\star\right\}} 
%\end{equation}
%the random number of normalised points in $A^\star$. 
%For the particular choice $A=A_n = [u_{n},\infty)^2$ for some threshold $u_{n} \in [0,\infty)$,
%we obtain $A^\star= A^\star_n = [u_n^\star, \infty)^2$ with $u_{n}^\star = u_n \log(1/p_{00}) - \log n$, and $\Xi^\star_{A^\star}$
%captures joint threshold exceedances of the components of the normalised random pairs $\mathbf{X}^\star_1, \ldots, \mathbf{X}^\star_n$. 
The following proposition gives straightforward error estimates for the approximation of the law of $\Xi^\star_{A^\star}$ by that of a Poisson process
with mean measure $\mathbb{E}\Xi^\star_{A^\star}$, both for general sets $A^\star$, and for the particular choice $A^\star = [u_n^\star, \infty)^2$ (for which $\Xi^\star_{A^\star}$ captures joint threshold exceedances of the components of the normalised random pairs $\mathbf{X}_1^\star, \ldots, \mathbf{X}_n^\star$).
%-----------------------------------------------------------------------------------------------------------------------------------------------
\begin{prop}\label{t: MO_Geo_dTV_lattice}
Suppose $\textbf{X}=(X_1,X_2)$ follows the Marshall-Olkin geometric distribution with parameters $q_1, q_2, p_{00} \in (0,1)$.   
For each integer $n \ge 1$, let $\mathbf{X}^\star_1, \ldots, \mathbf{X}^\star_n$ be \iid copies of the normalised random pair 
$\mathbf{X}^\star = (X^\star_1, X^\star_2)$ with state space $E^\star$, where $X_j^\star = \log(1/p_{00})X_j - \log n$, for $j=1,2$. 
Let $A^\star \in \mathcal{B}([0,\infty)^2)$ and let $\Xi^\star_{A^\star}$ and $W^\star_{A^\star}$ be defined as 
%in (\ref{d: MO_Geo_MPPE}) and 
%(\ref{d: MO_Geo_mean}), respectively. 
follows:
\[
\Xi^\star_{A^\star} =\sum_{i=1}^n I_{\left\{\textbf{X}^\star_i \in A^\star\right\}} \delta_{\mathbf{X}_i^\star}, \quad \text{and} \quad
W^\star_{A^\star}= \sum_{i=1}^n I_{\left\{\textbf{X}^\star_i \in A^\star\right\}}. 
\]
Then the mean measure of $\Xi^\star_{A^\star}$ is given by
\[
\bpis(B^\star) := \bpis_{\As}(B^\star) := \mathbb{E}\Xi^\star_{A^\star} (B^\star)= \sum_{(k^\star,l^\star) \in A^\star \cap E^\star \cap B^\star} nP(X_1^\star = k^\star, X_2^\star = l^\star), 
\]
for any $B^\star \in \mathcal{B}([-\log n, \infty)^2)$, where, for any $(k^\star, l^\star) \in E^\star$,
\begin{align}\label{d: MO_Geo_mass_fct_normalised}
&P\left(X_1^\star = k^\star, X_2^\star = l^\star\right) \nonumber \\
&= \left\{
\begin{array}{ll}
\frac{1}{n}\, (1-\frac{p_{00}}{q_2} - q_2 + p_{00})\,\mathrm{e}^{-\frac{\log(p_{00}/q_2)}{\log p_{00}} \, k^\star} \mathrm{e}^{- \frac{\log q_2}{\log p_{00}}\, l^\star}  & \textrm{ for } k^\star < l^\star , \\
\frac{1}{n}\,( 1-q_1 - q_2 + p_{00})\, \mathrm{e}^{-k^\star} \phantom{\mathrm{e}^{-\frac{\log(p_{00}/q_2)}{\log p_{00}} \, k^\star}}& \textrm{ for } k^\star=l^\star ,  \\
\frac{1}{n}\,(1-q_1 - \frac{p_{00}}{q_1} + p_{00} )\, \mathrm{e}^{-\frac{\log q_1}{\log p_{00}}\,k^\star} \mathrm{e}^{-\frac{\log(p_{00}/q_1)}{\log p_{00}}\,l^\star}& \textrm{ for } k^\star>l^\star ,
\end{array}
\right.
\end{align}
and 
$d_{TV}\left( \mathcal{L}\left(\Xi^\star_{A^\star} \right), \mathrm{PRM}(\bpis)\right)  \le P(\mathbf{X}^\star \in A^\star)$.
With $A^\star = A^\star_n = [u_n^\star, \infty)^2$ for any choice of $u_n^\star \ge -\log n$, we obtain
\begin{equation}\label{t: MO_Geo_dTV_basic_Poiapprox_appl}
d_{TV}\left( \mathcal{L}\left(\Xi^\star_{A^\star} \right), \mathrm{PRM}(\bpis)\right) 
\le \frac{\mathrm{e}^{-u^\star_n}}{n}\,.
\end{equation}
\end{prop}
\begin{proof}
With (\ref{d: MO_pointprob}) and (\ref{d: MO_Geo_normalisation_joint}), we obtain (\ref{d: MO_Geo_mass_fct_normalised}).
%\[
%P\left(X_1^\star = k^\star, X_2^\star = l^\star\right) = P\left( X_1 = \frac{k^\star + \log n}{\log(1/p_{00})}, X_2 = \frac{l^\star + \log n }{\log(1/p_{00})}\right), 
%\]
%we obtain (\ref{d: MO_Geo_mass_fct_normalised}) for the joint probability mass function of $\mathbf{X}^\star$.
For any 
%set 
$B^\star \in \mathcal{B}([-\log n, \infty)^2)$, 
the mean measure of $\Xi^\star_{A^\star}$ applied to $B^\star$ is 
%then 
given by
\[
nP(\mathbf{X}^\star \in A^\star \cap B^\star) =\sum_{(k^\star,l^\star) \in A^\star \cap E^\star \cap B^\star} nP(X_1^\star = k^\star, X_2^\star = l^\star).
\]
By Theorem \ref{t: MyMichelMultivariate},
%\[
$
d_{TV}\left( \mathcal{L}\left(\Xi^\star_{A^\star} \right)\,, \mathrm{PRM}(\bpis)\right) 
 \le P(\mathbf{X}^\star \in A^\star),
 $
%\]
where, using \eqref{d: MO_survival}, we find
\begin{align}
&P\left( \mathbf{X}^\star \in A^\star\right) 
%= P\left( X_1^\star \ge u_n^\star , X_2^\star \ge u_{n}^\star\right) 
%= P \left( X_1 \ge \frac{u^\star_n + \log n}{\log(1/p_{00})}, X_2 \ge \frac{u^\star_n + \log n}{\log(1/p_{00})\, l^\star} \right) \nonumber \\
%=\, \,& 
=
P \left( X_1 \ge \left \lceil \frac{u^\star_n + \log n}{\log(1/p_{00})} \right \rceil, 
X_2 \ge \left \lceil \frac{u^\star_n + \log n}{\log(1/p_{00})} \right \rceil \right)
= p_{00}^{\left \lceil \frac{u^\star_n + \log n}{\log(1/p_{00})} \right \rceil} \le \frac{\mathrm{e}^{-u^\star_n}}{n}\, .\label{eq:diagonal_case}
\end{align}
\end{proof}
%-----------------------------------------------------------------------------------------------------------------------------------------------
%\begin{remark}
%For $A^\star = [u_{1n}^\star, \infty) \times [u_{2n}^\star, \infty)$ with $u_{1n}^\star \neq u_{2n}^\star \in [-\log n, \infty)$, we can 
%use Theorem \ref{t: MyMichelMultivariate},
%in order to determine an estimate for $d_{TV}(\mathcal{L}(\Xi^\star_{A^\star}),$ $ \mathrm{PRM}(\bpis))$. 
%\end{remark}
%The error bound in (\ref{t: MO_Geo_dTV_basic_Poiapprox_appl}) is exactly the same as the one that we found in (\ref{e: MO_Exp_joint_dTV_Michel}) 
%for analogous MPPE's with Marshall-Olkin exponential instead of geometric marks. The difference is of course that the mean measure $\bpis$ of the MPPE with Marshall-Olkin
%geometric marks lives only on points $(k^\star, l^\star) \in A^\star \cap E^\star$ instead of on the whole of $A^\star \cap [-\log n, \infty)^2$.
%-----------------------------------------------------------------------------------------------------------------------------------------------
\subsubsection{Construction of a ``continuous'' intensity function}\label{s: MO_Geo_cont_int_fct} 
Proposition \ref{t: MO_Geo_dTV_lattice} gives an error bound for the approximation of the MPPE $\Xi^\star_{A^\star}$ by a Poisson process whose
mean measure $\mathbb{E}\Xi^\star_{A^\star}$ lives on the lattice of normalised points $(k^\star, l^\star)$, i.e. on
\begin{align*}
E^\star %&= \left\{ (k^\star, l^\star):\, k^\star = k \log(1/p_{00}) - \log n, l^\star = l \log(1/p_{00}) - \log n, \text{for all }(k,l) \in \mathbb{Z}_+^2\right\} \\
&= \left(\log(1/p_{00})\mathbb{Z}_+  - \log n\right)^2 \subset [-\log n, \infty)^2.
\end{align*}
We would however prefer to approximate the law of the MPPE by that of a Poisson process with an easier-to-use and more flexible \textit{continuous} 
intensity measure $\bl^\star = \bl^\star_{A^\star}$ living on $A^\star \cap [-\log n, \infty)^2$. 

%For readers acquainted with the theory of copulas, the survival copula of the Marshall-Olkin geometric distribution is a Marshall-Olkin copula, and thereby consists 
%of both an absolutely continuous part and a singular part on the curve $u^\alpha = v^\beta$ (which corresponds to the diagonal in $[-\log n, \infty)^2$). 
As the survival copula of the Marshall-Olkin geometric distribution is a Marshall-Olkin copula, and thereby consists of both an absolutely continuous part off the diagonal in $[-\log n, \infty)^2$ and a singular part on the diagonal (refer to Section 3.1.1 in \cite{Nelsen:2006}), the
%The
``continuous'' intensity measure $\bl^\star$ will have to be of the form % how does this translate to normalised space?
\begin{equation}\label{d: MO_Geo_cont_int_meas}
\bl^\star(B^\star) = \int_{A^\star \cap B^\star} \lambda^\star(s,t)\De s\De t + \int_{A^\star \cap B^\star \cap \{(s,t):\, s=t\}} \acute{\lambda}^\star(s)\De \De ts, 
\end{equation}
for any $B^\star \in \mathcal{B}([-\log n, \infty)^2)$, for ``continuous'' intensity functions $\lambda^\star$ and $\acute{\lambda}^\star$ that, 
if integrated over the entire space, will give $n$.
%, i.e. that will ensure that
%\[
%\int_{-\log n}^\infty \int_{-\log n}^\infty \lambda^\star(s,t) ds\De t + \int_{-\log n}^\infty \acute{\lambda}^\star(s)ds = nP\left(\mathbf{X}^\star \in [-\log n, \infty)^2 \right) = n.
%\]
\begin{remark}
Note that for simplicity of language we here (and later on) somewhat abuse terminology when speaking of a ``continuous'' intensity function $\ls$ or a ``continuous'' intensity 
measure $\bls$. The bivariate intensity function $\ls$ is not continuous, but piecewise continuous, having a jump along the diagonal. 
The measure $\bls$ is continuous only in the sense that it has an intensity with respect to 
Lebesgue measure ($2$-dimensional on the off-diagonal and $1$-dimensional on the diagonal) and not with respect to a point measure. 
\end{remark}
The idea is to spread the point mass sitting on each of the off-diagonal lattice points $(k^\star, l^\star) \in E^\star$, $k^\star \neq l^\star$, uniformly
over each of their corresponding coordinate rectangles (or rather, coordinate squares)
\[
R^\star_{k^\star,l^\star} 
= \left[k^\star, k^\star + \log \left(\frac{1}{p_{00}}\right)\right) \times \left[l^\star, l^\star + \log \left(\frac{1}{p_{00}}\right)\right), \quad k^\star \neq l^\star,
\]
and to also spread the point probabilities of the diagonal points $(k^\star,k^\star)$ over the diagonal line $s=t$, where $s,t \ge - \log n$. 
We achieve this in the following three steps, illustrated in Figure \ref{f: spread_prob_1}.

\textbf{Step 1.} Consider only the off-diagonal lattice points. We of course have 
\[
P\left(\mathbf{X}^\star \in R^\star_{k^\star,l^\star}  \right)= P\left(X_1^\star = k^\star, X_2^\star = l^\star\right),
\]
which is given by (\ref{d: MO_Geo_mass_fct_normalised}), and we may express the mean $nP(\mathbf{X}^\star \in A^\star)$ as %$n$ times 
\begin{multline}\label{p: spread_prob_1}
\sum_{ (k^\star,l^\star) \in A^\star, k^\star \neq l^\star } n
\int \int_{\Rstar} \frac{P(X_1^\star=k^\star, X_2^\star=l^\star)}{\log^2 (1/p_{00})}\, \De s\De t + \sum_{(k^\star, k^\star) \in A^\star}n P\left(X_1^\star=k^\star, X_2^\star=k^\star\right),
\end{multline}
where $\log^2 (1/p_{00})$ is the surface area of $\Rstar$. 
%We have not actually changed anything yet as the integrand 
%$P(X_1^\star=k^\star, X_2^\star=l^\star)/\log^2 (1/p_{00})$ is constant
%with respect to the integrating variables $s$ and $t$, and 
%\[
% \int \int_{\Rstar} \frac{P(X_1^\star=k^\star, X_2^\star =l^\star)}{\log^2 (1/p_{00})} \,ds\De t = P(X_1^\star=k^\star, X_2^\star=l^\star).
%\]
%%%%%%%%%%%%%%%%%%%%%%%%%%%%%
\begin{figure}
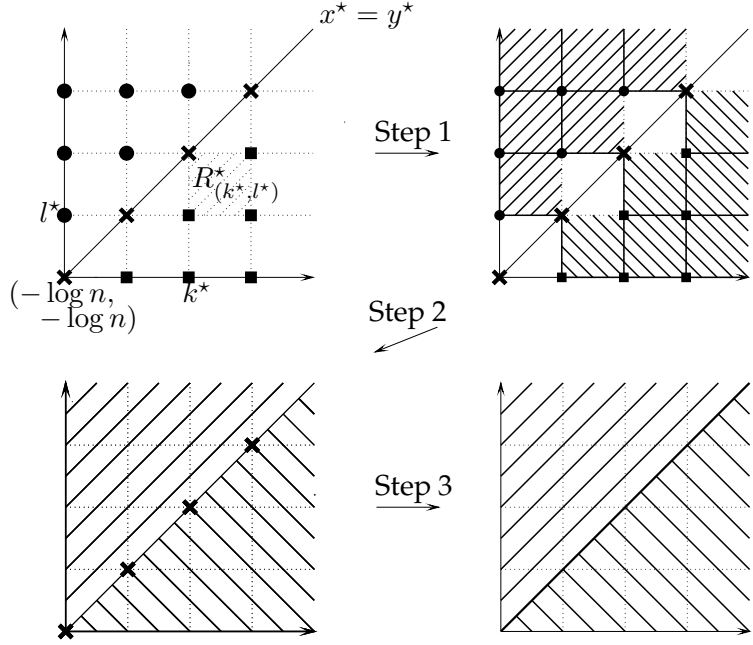

\centering
\begin{pgfpicture}{31.34mm}{21.24mm}{129.86mm}{110.63mm}
\pgfsetxvec{\pgfpoint{1.00mm}{0mm}}
\pgfsetyvec{\pgfpoint{0mm}{1.00mm}}
\color[rgb]{0,0,0}\pgfsetlinewidth{0.10mm}\pgfsetdash{}{0mm}
\pgfmoveto{\pgfxy(38.78,102.36)}\pgflineto{\pgfxy(38.78,69.31)}\pgfstroke
\pgfmoveto{\pgfxy(38.78,102.36)}\pgflineto{\pgfxy(38.48,101.16)}\pgflineto{\pgfxy(38.78,101.76)}\pgflineto{\pgfxy(39.08,101.16)}\pgflineto{\pgfxy(38.78,102.36)}\pgfclosepath\pgffill
\pgfmoveto{\pgfxy(38.78,102.36)}\pgflineto{\pgfxy(38.48,101.16)}\pgflineto{\pgfxy(38.78,101.76)}\pgflineto{\pgfxy(39.08,101.16)}\pgflineto{\pgfxy(38.78,102.36)}\pgfclosepath\pgfstroke
\pgfsetdash{{0.10mm}{0.50mm}}{0mm}\pgfmoveto{\pgfxy(47.05,102.36)}\pgflineto{\pgfxy(47.05,69.31)}\pgfstroke
\pgfmoveto{\pgfxy(63.58,102.36)}\pgflineto{\pgfxy(63.58,69.31)}\pgfstroke
\pgfmoveto{\pgfxy(38.78,85.83)}\pgflineto{\pgfxy(71.84,85.83)}\pgfstroke
\pgfmoveto{\pgfxy(38.78,94.10)}\pgflineto{\pgfxy(71.84,94.10)}\pgfstroke
\pgfsetdash{}{0mm}\pgfmoveto{\pgfxy(38.78,69.31)}\pgflineto{\pgfxy(71.84,102.36)}\pgfstroke
\pgfcircle[fill]{\pgfxy(47.05,85.83)}{0.80mm}
\pgfsetlinewidth{0.40mm}\pgfcircle[stroke]{\pgfxy(47.05,85.83)}{0.80mm}
\pgfcircle[fill]{\pgfxy(47.05,94.10)}{0.80mm}
\pgfcircle[stroke]{\pgfxy(47.05,94.10)}{0.80mm}
\pgfcircle[fill]{\pgfxy(55.31,94.10)}{0.80mm}
\pgfcircle[stroke]{\pgfxy(55.31,94.10)}{0.80mm}
\pgfcircle[fill]{\pgfxy(38.78,85.83)}{0.80mm}
\pgfcircle[stroke]{\pgfxy(38.78,85.83)}{0.80mm}
\pgfcircle[fill]{\pgfxy(38.78,77.57)}{0.80mm}
\pgfcircle[stroke]{\pgfxy(38.78,77.57)}{0.80mm}
\pgfcircle[fill]{\pgfxy(38.78,94.10)}{0.80mm}
\pgfcircle[stroke]{\pgfxy(38.78,94.10)}{0.80mm}
\pgfcircle[fill]{\pgfxy(47.05,77.57)}{0.20mm}
\pgfsetlinewidth{0.10mm}\pgfcircle[stroke]{\pgfxy(47.05,77.57)}{0.20mm}
\pgfcircle[fill]{\pgfxy(47.05,77.57)}{0.20mm}
\pgfcircle[stroke]{\pgfxy(47.05,77.57)}{0.20mm}
\pgfcircle[fill]{\pgfxy(47.05,69.31)}{0.20mm}
\pgfcircle[stroke]{\pgfxy(47.05,69.31)}{0.20mm}
\pgfcircle[fill]{\pgfxy(55.31,69.31)}{0.20mm}
\pgfcircle[stroke]{\pgfxy(55.31,69.31)}{0.20mm}
\pgfcircle[fill]{\pgfxy(63.58,69.31)}{0.20mm}
\pgfcircle[stroke]{\pgfxy(63.58,69.31)}{0.20mm}
\pgfcircle[fill]{\pgfxy(55.31,77.57)}{0.20mm}
\pgfcircle[stroke]{\pgfxy(55.31,77.57)}{0.20mm}
\pgfcircle[fill]{\pgfxy(55.31,77.57)}{0.20mm}
\pgfcircle[stroke]{\pgfxy(55.31,77.57)}{0.20mm}
\pgfmoveto{\pgfxy(55.39,77.57)}\pgflineto{\pgfxy(55.55,77.73)}\pgflineto{\pgfxy(55.47,77.81)}\pgflineto{\pgfxy(55.31,77.65)}\pgflineto{\pgfxy(55.15,77.81)}\pgflineto{\pgfxy(55.07,77.73)}\pgflineto{\pgfxy(55.23,77.57)}\pgflineto{\pgfxy(55.07,77.41)}\pgflineto{\pgfxy(55.15,77.33)}\pgflineto{\pgfxy(55.31,77.49)}\pgflineto{\pgfxy(55.47,77.33)}\pgflineto{\pgfxy(55.55,77.41)}\pgfclosepath\pgffill
\pgfmoveto{\pgfxy(55.39,77.57)}\pgflineto{\pgfxy(55.55,77.73)}\pgflineto{\pgfxy(55.47,77.81)}\pgflineto{\pgfxy(55.31,77.65)}\pgflineto{\pgfxy(55.15,77.81)}\pgflineto{\pgfxy(55.07,77.73)}\pgflineto{\pgfxy(55.23,77.57)}\pgflineto{\pgfxy(55.07,77.41)}\pgflineto{\pgfxy(55.15,77.33)}\pgflineto{\pgfxy(55.31,77.49)}\pgflineto{\pgfxy(55.47,77.33)}\pgflineto{\pgfxy(55.55,77.41)}\pgfclosepath\pgfstroke
\pgfcircle[fill]{\pgfxy(55.31,77.57)}{0.20mm}
\pgfcircle[stroke]{\pgfxy(55.31,77.57)}{0.20mm}
\pgfcircle[fill]{\pgfxy(63.58,77.57)}{0.20mm}
\pgfcircle[stroke]{\pgfxy(63.58,77.57)}{0.20mm}
\pgfcircle[fill]{\pgfxy(63.58,85.83)}{0.20mm}
\pgfcircle[stroke]{\pgfxy(63.58,85.83)}{0.20mm}
\pgfcircle[fill]{\pgfxy(47.05,69.31)}{0.20mm}
\pgfcircle[stroke]{\pgfxy(47.05,69.31)}{0.20mm}
\pgfcircle[fill]{\pgfxy(47.05,69.31)}{0.20mm}
\pgfcircle[stroke]{\pgfxy(47.05,69.31)}{0.20mm}
\pgfcircle[fill]{\pgfxy(47.05,69.31)}{0.20mm}
\pgfcircle[stroke]{\pgfxy(47.05,69.31)}{0.20mm}
\pgfcircle[fill]{\pgfxy(55.31,69.31)}{0.20mm}
\pgfcircle[stroke]{\pgfxy(55.31,69.31)}{0.20mm}
\pgfcircle[fill]{\pgfxy(63.58,69.31)}{0.20mm}
\pgfcircle[stroke]{\pgfxy(63.58,69.31)}{0.20mm}
\pgfcircle[fill]{\pgfxy(47.05,77.57)}{0.20mm}
\pgfcircle[stroke]{\pgfxy(47.05,77.57)}{0.20mm}
\pgfcircle[fill]{\pgfxy(38.78,69.31)}{0.20mm}
\pgfcircle[stroke]{\pgfxy(38.78,69.31)}{0.20mm}
\pgfcircle[fill]{\pgfxy(55.31,85.83)}{0.20mm}
\pgfcircle[stroke]{\pgfxy(55.31,85.83)}{0.20mm}
\pgfcircle[fill]{\pgfxy(63.58,94.10)}{0.20mm}
\pgfcircle[stroke]{\pgfxy(63.58,94.10)}{0.20mm}
\pgfsetlinewidth{0.30mm}\pgfmoveto{\pgfxy(47.05,110.63)}\pgflineto{\pgfxy(47.05,110.63)}\pgfstroke
\pgfsetlinewidth{0.10mm}\pgfmoveto{\pgfxy(38.78,69.31)}\pgflineto{\pgfxy(71.84,69.31)}\pgfstroke
\pgfmoveto{\pgfxy(71.84,69.31)}\pgflineto{\pgfxy(70.64,69.61)}\pgflineto{\pgfxy(71.24,69.31)}\pgflineto{\pgfxy(70.64,69.01)}\pgflineto{\pgfxy(71.84,69.31)}\pgfclosepath\pgffill
\pgfmoveto{\pgfxy(71.84,69.31)}\pgflineto{\pgfxy(70.64,69.61)}\pgflineto{\pgfxy(71.24,69.31)}\pgflineto{\pgfxy(70.64,69.01)}\pgflineto{\pgfxy(71.84,69.31)}\pgfclosepath\pgfstroke
\pgfmoveto{\pgfxy(47.37,77.57)}\pgflineto{\pgfxy(48.01,78.21)}\pgflineto{\pgfxy(47.69,78.53)}\pgflineto{\pgfxy(47.05,77.89)}\pgflineto{\pgfxy(46.41,78.53)}\pgflineto{\pgfxy(46.09,78.21)}\pgflineto{\pgfxy(46.73,77.57)}\pgflineto{\pgfxy(46.09,76.93)}\pgflineto{\pgfxy(46.41,76.61)}\pgflineto{\pgfxy(47.05,77.25)}\pgflineto{\pgfxy(47.69,76.61)}\pgflineto{\pgfxy(48.01,76.93)}\pgfclosepath\pgffill
\pgfmoveto{\pgfxy(47.37,77.57)}\pgflineto{\pgfxy(48.01,78.21)}\pgflineto{\pgfxy(47.69,78.53)}\pgflineto{\pgfxy(47.05,77.89)}\pgflineto{\pgfxy(46.41,78.53)}\pgflineto{\pgfxy(46.09,78.21)}\pgflineto{\pgfxy(46.73,77.57)}\pgflineto{\pgfxy(46.09,76.93)}\pgflineto{\pgfxy(46.41,76.61)}\pgflineto{\pgfxy(47.05,77.25)}\pgflineto{\pgfxy(47.69,76.61)}\pgflineto{\pgfxy(48.01,76.93)}\pgfclosepath\pgfstroke
\pgfmoveto{\pgfxy(55.63,85.83)}\pgflineto{\pgfxy(56.27,86.47)}\pgflineto{\pgfxy(55.95,86.79)}\pgflineto{\pgfxy(55.31,86.15)}\pgflineto{\pgfxy(54.67,86.79)}\pgflineto{\pgfxy(54.35,86.47)}\pgflineto{\pgfxy(54.99,85.83)}\pgflineto{\pgfxy(54.35,85.19)}\pgflineto{\pgfxy(54.67,84.87)}\pgflineto{\pgfxy(55.31,85.51)}\pgflineto{\pgfxy(55.95,84.87)}\pgflineto{\pgfxy(56.27,85.19)}\pgfclosepath\pgffill
\pgfmoveto{\pgfxy(55.63,85.83)}\pgflineto{\pgfxy(56.27,86.47)}\pgflineto{\pgfxy(55.95,86.79)}\pgflineto{\pgfxy(55.31,86.15)}\pgflineto{\pgfxy(54.67,86.79)}\pgflineto{\pgfxy(54.35,86.47)}\pgflineto{\pgfxy(54.99,85.83)}\pgflineto{\pgfxy(54.35,85.19)}\pgflineto{\pgfxy(54.67,84.87)}\pgflineto{\pgfxy(55.31,85.51)}\pgflineto{\pgfxy(55.95,84.87)}\pgflineto{\pgfxy(56.27,85.19)}\pgfclosepath\pgfstroke
\pgfmoveto{\pgfxy(63.90,94.10)}\pgflineto{\pgfxy(64.54,94.74)}\pgflineto{\pgfxy(64.22,95.06)}\pgflineto{\pgfxy(63.58,94.42)}\pgflineto{\pgfxy(62.94,95.06)}\pgflineto{\pgfxy(62.62,94.74)}\pgflineto{\pgfxy(63.26,94.10)}\pgflineto{\pgfxy(62.62,93.46)}\pgflineto{\pgfxy(62.94,93.14)}\pgflineto{\pgfxy(63.58,93.78)}\pgflineto{\pgfxy(64.22,93.14)}\pgflineto{\pgfxy(64.54,93.46)}\pgfclosepath\pgffill
\pgfmoveto{\pgfxy(63.90,94.10)}\pgflineto{\pgfxy(64.54,94.74)}\pgflineto{\pgfxy(64.22,95.06)}\pgflineto{\pgfxy(63.58,94.42)}\pgflineto{\pgfxy(62.94,95.06)}\pgflineto{\pgfxy(62.62,94.74)}\pgflineto{\pgfxy(63.26,94.10)}\pgflineto{\pgfxy(62.62,93.46)}\pgflineto{\pgfxy(62.94,93.14)}\pgflineto{\pgfxy(63.58,93.78)}\pgflineto{\pgfxy(64.22,93.14)}\pgflineto{\pgfxy(64.54,93.46)}\pgfclosepath\pgfstroke
\pgfcircle[fill]{\pgfxy(38.78,69.31)}{0.20mm}
\pgfcircle[stroke]{\pgfxy(38.78,69.31)}{0.20mm}
\pgfmoveto{\pgfxy(39.10,69.31)}\pgflineto{\pgfxy(39.74,69.95)}\pgflineto{\pgfxy(39.42,70.27)}\pgflineto{\pgfxy(38.78,69.63)}\pgflineto{\pgfxy(38.14,70.27)}\pgflineto{\pgfxy(37.82,69.95)}\pgflineto{\pgfxy(38.46,69.31)}\pgflineto{\pgfxy(37.82,68.67)}\pgflineto{\pgfxy(38.14,68.35)}\pgflineto{\pgfxy(38.78,68.99)}\pgflineto{\pgfxy(39.42,68.35)}\pgflineto{\pgfxy(39.74,68.67)}\pgfclosepath\pgffill
\pgfmoveto{\pgfxy(39.10,69.31)}\pgflineto{\pgfxy(39.74,69.95)}\pgflineto{\pgfxy(39.42,70.27)}\pgflineto{\pgfxy(38.78,69.63)}\pgflineto{\pgfxy(38.14,70.27)}\pgflineto{\pgfxy(37.82,69.95)}\pgflineto{\pgfxy(38.46,69.31)}\pgflineto{\pgfxy(37.82,68.67)}\pgflineto{\pgfxy(38.14,68.35)}\pgflineto{\pgfxy(38.78,68.99)}\pgflineto{\pgfxy(39.42,68.35)}\pgflineto{\pgfxy(39.74,68.67)}\pgfclosepath\pgfstroke
\pgfmoveto{\pgfxy(47.77,70.03)}\pgflineto{\pgfxy(46.33,70.03)}\pgflineto{\pgfxy(46.33,68.59)}\pgflineto{\pgfxy(47.77,68.59)}\pgfclosepath\pgffill
\pgfmoveto{\pgfxy(47.77,70.03)}\pgflineto{\pgfxy(46.33,70.03)}\pgflineto{\pgfxy(46.33,68.59)}\pgflineto{\pgfxy(47.77,68.59)}\pgfclosepath\pgfstroke
\pgfmoveto{\pgfxy(56.03,78.29)}\pgflineto{\pgfxy(54.59,78.29)}\pgflineto{\pgfxy(54.59,76.85)}\pgflineto{\pgfxy(56.03,76.85)}\pgfclosepath\pgffill
\pgfmoveto{\pgfxy(56.03,78.29)}\pgflineto{\pgfxy(54.59,78.29)}\pgflineto{\pgfxy(54.59,76.85)}\pgflineto{\pgfxy(56.03,76.85)}\pgfclosepath\pgfstroke
\pgfmoveto{\pgfxy(56.03,70.03)}\pgflineto{\pgfxy(54.59,70.03)}\pgflineto{\pgfxy(54.59,68.59)}\pgflineto{\pgfxy(56.03,68.59)}\pgfclosepath\pgffill
\pgfmoveto{\pgfxy(56.03,70.03)}\pgflineto{\pgfxy(54.59,70.03)}\pgflineto{\pgfxy(54.59,68.59)}\pgflineto{\pgfxy(56.03,68.59)}\pgfclosepath\pgfstroke
\pgfmoveto{\pgfxy(64.30,86.55)}\pgflineto{\pgfxy(62.86,86.55)}\pgflineto{\pgfxy(62.86,85.11)}\pgflineto{\pgfxy(64.30,85.11)}\pgfclosepath\pgffill
\pgfmoveto{\pgfxy(64.30,86.55)}\pgflineto{\pgfxy(62.86,86.55)}\pgflineto{\pgfxy(62.86,85.11)}\pgflineto{\pgfxy(64.30,85.11)}\pgfclosepath\pgfstroke
\pgfmoveto{\pgfxy(64.30,78.29)}\pgflineto{\pgfxy(62.86,78.29)}\pgflineto{\pgfxy(62.86,76.85)}\pgflineto{\pgfxy(64.30,76.85)}\pgfclosepath\pgffill
\pgfmoveto{\pgfxy(64.30,78.29)}\pgflineto{\pgfxy(62.86,78.29)}\pgflineto{\pgfxy(62.86,76.85)}\pgflineto{\pgfxy(64.30,76.85)}\pgfclosepath\pgfstroke
\pgfmoveto{\pgfxy(64.30,70.03)}\pgflineto{\pgfxy(62.86,70.03)}\pgflineto{\pgfxy(62.86,68.59)}\pgflineto{\pgfxy(64.30,68.59)}\pgfclosepath\pgffill
\pgfmoveto{\pgfxy(64.30,70.03)}\pgflineto{\pgfxy(62.86,70.03)}\pgflineto{\pgfxy(62.86,68.59)}\pgflineto{\pgfxy(64.30,68.59)}\pgfclosepath\pgfstroke
\pgfmoveto{\pgfxy(42.91,69.31)}\pgflineto{\pgfxy(42.91,69.31)}\pgfstroke
\pgfputat{\pgfxy(31.34,66.00)}{\pgfbox[bottom,left]{$(-\log n,$}}
\pgfputat{\pgfxy(72.67,103.19)}{\pgfbox[bottom,left]{$x^\star=y^\star$}}
\pgfputat{\pgfxy(35.48,62.69)}{\pgfbox[bottom,left]{$-\log n)$}}
\pgfmoveto{\pgfxy(75.97,89.97)}\pgflineto{\pgfxy(75.97,89.97)}\pgfstroke
\pgfmoveto{\pgfxy(76.80,85.83)}\pgfstroke
\pgfmoveto{\pgfxy(80.10,85.83)}\pgflineto{\pgfxy(88.37,85.83)}\pgfstroke
\pgfmoveto{\pgfxy(88.37,85.83)}\pgflineto{\pgfxy(87.17,86.13)}\pgflineto{\pgfxy(87.77,85.83)}\pgflineto{\pgfxy(87.17,85.53)}\pgflineto{\pgfxy(88.37,85.83)}\pgfclosepath\pgffill
\pgfmoveto{\pgfxy(88.37,85.83)}\pgflineto{\pgfxy(87.17,86.13)}\pgflineto{\pgfxy(87.77,85.83)}\pgflineto{\pgfxy(87.17,85.53)}\pgflineto{\pgfxy(88.37,85.83)}\pgfclosepath\pgfstroke
\pgfmoveto{\pgfxy(96.63,102.36)}\pgflineto{\pgfxy(96.63,69.31)}\pgfstroke
\pgfmoveto{\pgfxy(96.63,102.36)}\pgflineto{\pgfxy(96.33,101.16)}\pgflineto{\pgfxy(96.63,101.76)}\pgflineto{\pgfxy(96.93,101.16)}\pgflineto{\pgfxy(96.63,102.36)}\pgfclosepath\pgffill
\pgfmoveto{\pgfxy(96.63,102.36)}\pgflineto{\pgfxy(96.33,101.16)}\pgflineto{\pgfxy(96.63,101.76)}\pgflineto{\pgfxy(96.93,101.16)}\pgflineto{\pgfxy(96.63,102.36)}\pgfclosepath\pgfstroke
\pgfmoveto{\pgfxy(96.63,69.31)}\pgflineto{\pgfxy(129.69,69.31)}\pgfstroke
\pgfmoveto{\pgfxy(129.69,69.31)}\pgflineto{\pgfxy(128.49,69.61)}\pgflineto{\pgfxy(129.09,69.31)}\pgflineto{\pgfxy(128.49,69.01)}\pgflineto{\pgfxy(129.69,69.31)}\pgfclosepath\pgffill
\pgfmoveto{\pgfxy(129.69,69.31)}\pgflineto{\pgfxy(128.49,69.61)}\pgflineto{\pgfxy(129.09,69.31)}\pgflineto{\pgfxy(128.49,69.01)}\pgflineto{\pgfxy(129.69,69.31)}\pgfclosepath\pgfstroke
\pgfsetdash{{0.10mm}{0.50mm}}{0mm}\pgfmoveto{\pgfxy(104.90,102.36)}\pgflineto{\pgfxy(104.90,69.31)}\pgfstroke
\pgfmoveto{\pgfxy(113.16,102.36)}\pgflineto{\pgfxy(113.16,69.31)}\pgfstroke
\pgfmoveto{\pgfxy(121.43,102.36)}\pgflineto{\pgfxy(121.43,69.31)}\pgfstroke
\pgfmoveto{\pgfxy(96.63,77.57)}\pgflineto{\pgfxy(129.69,77.57)}\pgfstroke
\pgfmoveto{\pgfxy(96.63,85.83)}\pgflineto{\pgfxy(129.69,85.83)}\pgfstroke
\pgfmoveto{\pgfxy(96.63,94.10)}\pgflineto{\pgfxy(129.69,94.10)}\pgfstroke
\pgfsetdash{}{0mm}\pgfmoveto{\pgfxy(96.63,69.31)}\pgflineto{\pgfxy(129.69,102.36)}\pgfstroke
\pgfcircle[fill]{\pgfxy(96.63,69.31)}{0.20mm}
\pgfcircle[stroke]{\pgfxy(96.63,69.31)}{0.20mm}
\pgfmoveto{\pgfxy(96.95,69.31)}\pgflineto{\pgfxy(97.59,69.95)}\pgflineto{\pgfxy(97.27,70.27)}\pgflineto{\pgfxy(96.63,69.63)}\pgflineto{\pgfxy(95.99,70.27)}\pgflineto{\pgfxy(95.67,69.95)}\pgflineto{\pgfxy(96.31,69.31)}\pgflineto{\pgfxy(95.67,68.67)}\pgflineto{\pgfxy(95.99,68.35)}\pgflineto{\pgfxy(96.63,68.99)}\pgflineto{\pgfxy(97.27,68.35)}\pgflineto{\pgfxy(97.59,68.67)}\pgfclosepath\pgffill
\pgfmoveto{\pgfxy(96.95,69.31)}\pgflineto{\pgfxy(97.59,69.95)}\pgflineto{\pgfxy(97.27,70.27)}\pgflineto{\pgfxy(96.63,69.63)}\pgflineto{\pgfxy(95.99,70.27)}\pgflineto{\pgfxy(95.67,69.95)}\pgflineto{\pgfxy(96.31,69.31)}\pgflineto{\pgfxy(95.67,68.67)}\pgflineto{\pgfxy(95.99,68.35)}\pgflineto{\pgfxy(96.63,68.99)}\pgflineto{\pgfxy(97.27,68.35)}\pgflineto{\pgfxy(97.59,68.67)}\pgfclosepath\pgfstroke
\pgfmoveto{\pgfxy(105.22,77.57)}\pgflineto{\pgfxy(105.86,78.21)}\pgflineto{\pgfxy(105.54,78.53)}\pgflineto{\pgfxy(104.90,77.89)}\pgflineto{\pgfxy(104.26,78.53)}\pgflineto{\pgfxy(103.94,78.21)}\pgflineto{\pgfxy(104.58,77.57)}\pgflineto{\pgfxy(103.94,76.93)}\pgflineto{\pgfxy(104.26,76.61)}\pgflineto{\pgfxy(104.90,77.25)}\pgflineto{\pgfxy(105.54,76.61)}\pgflineto{\pgfxy(105.86,76.93)}\pgfclosepath\pgffill
\pgfmoveto{\pgfxy(105.22,77.57)}\pgflineto{\pgfxy(105.86,78.21)}\pgflineto{\pgfxy(105.54,78.53)}\pgflineto{\pgfxy(104.90,77.89)}\pgflineto{\pgfxy(104.26,78.53)}\pgflineto{\pgfxy(103.94,78.21)}\pgflineto{\pgfxy(104.58,77.57)}\pgflineto{\pgfxy(103.94,76.93)}\pgflineto{\pgfxy(104.26,76.61)}\pgflineto{\pgfxy(104.90,77.25)}\pgflineto{\pgfxy(105.54,76.61)}\pgflineto{\pgfxy(105.86,76.93)}\pgfclosepath\pgfstroke
\pgfmoveto{\pgfxy(113.48,85.83)}\pgflineto{\pgfxy(114.12,86.47)}\pgflineto{\pgfxy(113.80,86.79)}\pgflineto{\pgfxy(113.16,86.15)}\pgflineto{\pgfxy(112.52,86.79)}\pgflineto{\pgfxy(112.20,86.47)}\pgflineto{\pgfxy(112.84,85.83)}\pgflineto{\pgfxy(112.20,85.19)}\pgflineto{\pgfxy(112.52,84.87)}\pgflineto{\pgfxy(113.16,85.51)}\pgflineto{\pgfxy(113.80,84.87)}\pgflineto{\pgfxy(114.12,85.19)}\pgfclosepath\pgffill
\pgfmoveto{\pgfxy(113.48,85.83)}\pgflineto{\pgfxy(114.12,86.47)}\pgflineto{\pgfxy(113.80,86.79)}\pgflineto{\pgfxy(113.16,86.15)}\pgflineto{\pgfxy(112.52,86.79)}\pgflineto{\pgfxy(112.20,86.47)}\pgflineto{\pgfxy(112.84,85.83)}\pgflineto{\pgfxy(112.20,85.19)}\pgflineto{\pgfxy(112.52,84.87)}\pgflineto{\pgfxy(113.16,85.51)}\pgflineto{\pgfxy(113.80,84.87)}\pgflineto{\pgfxy(114.12,85.19)}\pgfclosepath\pgfstroke
\pgfmoveto{\pgfxy(121.75,94.10)}\pgflineto{\pgfxy(122.39,94.74)}\pgflineto{\pgfxy(122.07,95.06)}\pgflineto{\pgfxy(121.43,94.42)}\pgflineto{\pgfxy(120.79,95.06)}\pgflineto{\pgfxy(120.47,94.74)}\pgflineto{\pgfxy(121.11,94.10)}\pgflineto{\pgfxy(120.47,93.46)}\pgflineto{\pgfxy(120.79,93.14)}\pgflineto{\pgfxy(121.43,93.78)}\pgflineto{\pgfxy(122.07,93.14)}\pgflineto{\pgfxy(122.39,93.46)}\pgfclosepath\pgffill
\pgfmoveto{\pgfxy(121.75,94.10)}\pgflineto{\pgfxy(122.39,94.74)}\pgflineto{\pgfxy(122.07,95.06)}\pgflineto{\pgfxy(121.43,94.42)}\pgflineto{\pgfxy(120.79,95.06)}\pgflineto{\pgfxy(120.47,94.74)}\pgflineto{\pgfxy(121.11,94.10)}\pgflineto{\pgfxy(120.47,93.46)}\pgflineto{\pgfxy(120.79,93.14)}\pgflineto{\pgfxy(121.43,93.78)}\pgflineto{\pgfxy(122.07,93.14)}\pgflineto{\pgfxy(122.39,93.46)}\pgfclosepath\pgfstroke
\pgfsetlinewidth{0.20mm}\pgfmoveto{\pgfxy(96.63,94.10)}\pgflineto{\pgfxy(96.63,102.36)}\pgfstroke
\pgfmoveto{\pgfxy(96.63,94.10)}\pgflineto{\pgfxy(104.90,94.10)}\pgfstroke
\pgfmoveto{\pgfxy(96.63,94.10)}\pgflineto{\pgfxy(104.90,102.36)}\pgfstroke
\pgfmoveto{\pgfxy(96.63,98.23)}\pgflineto{\pgfxy(100.77,102.36)}\pgfstroke
\pgfmoveto{\pgfxy(100.77,94.10)}\pgflineto{\pgfxy(104.90,98.23)}\pgfstroke
\pgfmoveto{\pgfxy(96.63,100.30)}\pgflineto{\pgfxy(98.70,102.36)}\pgfstroke
\pgfmoveto{\pgfxy(96.63,96.16)}\pgflineto{\pgfxy(102.83,102.36)}\pgfstroke
\pgfmoveto{\pgfxy(98.70,94.10)}\pgflineto{\pgfxy(104.90,100.30)}\pgfstroke
\pgfmoveto{\pgfxy(102.83,94.10)}\pgflineto{\pgfxy(104.90,96.16)}\pgfstroke
\pgfmoveto{\pgfxy(96.63,85.83)}\pgflineto{\pgfxy(96.63,94.10)}\pgfstroke
\pgfmoveto{\pgfxy(96.63,85.83)}\pgflineto{\pgfxy(104.90,85.83)}\pgfstroke
\pgfmoveto{\pgfxy(96.63,85.83)}\pgflineto{\pgfxy(104.90,94.10)}\pgfstroke
\pgfmoveto{\pgfxy(96.63,87.90)}\pgflineto{\pgfxy(102.83,94.10)}\pgfstroke
\pgfmoveto{\pgfxy(100.77,94.10)}\pgflineto{\pgfxy(96.63,89.97)}\pgfstroke
\pgfmoveto{\pgfxy(98.70,94.10)}\pgflineto{\pgfxy(96.63,92.03)}\pgfstroke
\pgfmoveto{\pgfxy(98.70,85.83)}\pgflineto{\pgfxy(104.90,92.03)}\pgfstroke
\pgfmoveto{\pgfxy(100.77,85.83)}\pgflineto{\pgfxy(104.90,89.97)}\pgfstroke
\pgfmoveto{\pgfxy(102.83,85.83)}\pgflineto{\pgfxy(104.90,87.90)}\pgfstroke
\pgfmoveto{\pgfxy(96.63,77.57)}\pgflineto{\pgfxy(96.63,85.83)}\pgfstroke
\pgfmoveto{\pgfxy(96.63,77.57)}\pgflineto{\pgfxy(104.90,77.57)}\pgfstroke
\pgfmoveto{\pgfxy(98.70,85.83)}\pgflineto{\pgfxy(96.63,83.77)}\pgfstroke
\pgfmoveto{\pgfxy(100.77,85.83)}\pgflineto{\pgfxy(96.63,81.70)}\pgfstroke
\pgfmoveto{\pgfxy(102.83,85.83)}\pgflineto{\pgfxy(96.63,79.64)}\pgfstroke
\pgfmoveto{\pgfxy(96.63,77.57)}\pgflineto{\pgfxy(104.90,85.83)}\pgfstroke
\pgfmoveto{\pgfxy(98.70,77.57)}\pgflineto{\pgfxy(104.90,83.77)}\pgfstroke
\pgfmoveto{\pgfxy(100.77,77.57)}\pgflineto{\pgfxy(104.90,81.70)}\pgfstroke
\pgfmoveto{\pgfxy(102.83,77.57)}\pgflineto{\pgfxy(104.90,79.64)}\pgfstroke
\pgfmoveto{\pgfxy(104.90,94.10)}\pgflineto{\pgfxy(104.90,85.83)}\pgfstroke
\pgfmoveto{\pgfxy(104.90,85.83)}\pgflineto{\pgfxy(113.16,85.83)}\pgfstroke
\pgfmoveto{\pgfxy(104.90,102.36)}\pgflineto{\pgfxy(104.90,94.10)}\pgfstroke
\pgfmoveto{\pgfxy(104.90,94.10)}\pgflineto{\pgfxy(113.16,94.10)}\pgfstroke
\pgfmoveto{\pgfxy(113.16,102.36)}\pgflineto{\pgfxy(113.16,94.10)}\pgfstroke
\pgfmoveto{\pgfxy(113.16,94.10)}\pgflineto{\pgfxy(121.43,94.10)}\pgfstroke
\pgfmoveto{\pgfxy(104.90,100.30)}\pgflineto{\pgfxy(106.96,102.36)}\pgfstroke
\pgfmoveto{\pgfxy(104.90,98.23)}\pgflineto{\pgfxy(109.03,102.36)}\pgfstroke
\pgfmoveto{\pgfxy(104.90,96.16)}\pgflineto{\pgfxy(111.10,102.36)}\pgfstroke
\pgfmoveto{\pgfxy(104.90,94.10)}\pgflineto{\pgfxy(113.16,102.36)}\pgfstroke
\pgfmoveto{\pgfxy(104.90,92.03)}\pgflineto{\pgfxy(113.16,100.30)}\pgfstroke
\pgfmoveto{\pgfxy(104.90,89.97)}\pgflineto{\pgfxy(113.16,98.23)}\pgfstroke
\pgfmoveto{\pgfxy(104.90,87.90)}\pgflineto{\pgfxy(113.16,96.16)}\pgfstroke
\pgfmoveto{\pgfxy(104.90,85.83)}\pgflineto{\pgfxy(113.16,94.10)}\pgfstroke
\pgfmoveto{\pgfxy(106.96,85.83)}\pgflineto{\pgfxy(113.16,92.03)}\pgfstroke
\pgfmoveto{\pgfxy(109.03,85.83)}\pgflineto{\pgfxy(113.16,89.97)}\pgfstroke
\pgfmoveto{\pgfxy(111.10,85.83)}\pgflineto{\pgfxy(113.16,87.90)}\pgfstroke
\pgfmoveto{\pgfxy(113.16,100.30)}\pgflineto{\pgfxy(115.23,102.36)}\pgfstroke
\pgfmoveto{\pgfxy(113.16,98.23)}\pgflineto{\pgfxy(117.29,102.36)}\pgfstroke
\pgfmoveto{\pgfxy(113.16,96.16)}\pgflineto{\pgfxy(119.36,102.36)}\pgfstroke
\pgfmoveto{\pgfxy(113.16,94.10)}\pgflineto{\pgfxy(121.43,102.36)}\pgfstroke
\pgfmoveto{\pgfxy(115.23,94.10)}\pgflineto{\pgfxy(121.43,100.30)}\pgfstroke
\pgfmoveto{\pgfxy(117.29,94.10)}\pgflineto{\pgfxy(121.43,98.23)}\pgfstroke
\pgfmoveto{\pgfxy(119.36,94.10)}\pgflineto{\pgfxy(121.43,96.16)}\pgfstroke
\pgfcircle[fill]{\pgfxy(96.63,77.57)}{0.60mm}
\pgfcircle[stroke]{\pgfxy(96.63,77.57)}{0.60mm}
\pgfcircle[fill]{\pgfxy(96.63,85.83)}{0.60mm}
\pgfcircle[stroke]{\pgfxy(96.63,85.83)}{0.60mm}
\pgfcircle[fill]{\pgfxy(96.63,94.10)}{0.60mm}
\pgfcircle[stroke]{\pgfxy(96.63,94.10)}{0.60mm}
\pgfcircle[fill]{\pgfxy(104.90,85.83)}{0.60mm}
\pgfcircle[stroke]{\pgfxy(104.90,85.83)}{0.60mm}
\pgfcircle[fill]{\pgfxy(104.90,94.10)}{0.60mm}
\pgfcircle[stroke]{\pgfxy(104.90,94.10)}{0.60mm}
\pgfcircle[fill]{\pgfxy(113.16,94.10)}{0.60mm}
\pgfcircle[stroke]{\pgfxy(113.16,94.10)}{0.60mm}
\pgfmoveto{\pgfxy(104.90,77.57)}\pgflineto{\pgfxy(104.90,69.31)}\pgfstroke
\pgfmoveto{\pgfxy(104.90,69.31)}\pgflineto{\pgfxy(113.16,69.31)}\pgfstroke
\pgfmoveto{\pgfxy(113.16,69.31)}\pgflineto{\pgfxy(121.43,69.31)}\pgfstroke
\pgfmoveto{\pgfxy(121.43,69.31)}\pgflineto{\pgfxy(128.86,69.31)}\pgfstroke
\pgfmoveto{\pgfxy(113.16,77.57)}\pgflineto{\pgfxy(121.43,77.57)}\pgfstroke
\pgfmoveto{\pgfxy(113.16,77.57)}\pgflineto{\pgfxy(113.16,69.31)}\pgfstroke
\pgfmoveto{\pgfxy(113.16,85.83)}\pgflineto{\pgfxy(113.16,77.57)}\pgfstroke
\pgfmoveto{\pgfxy(121.43,85.83)}\pgflineto{\pgfxy(121.43,77.57)}\pgfstroke
\pgfmoveto{\pgfxy(121.43,77.57)}\pgflineto{\pgfxy(121.43,69.31)}\pgfstroke
\pgfmoveto{\pgfxy(121.43,77.57)}\pgflineto{\pgfxy(129.69,77.57)}\pgfstroke
\pgfmoveto{\pgfxy(121.43,94.10)}\pgflineto{\pgfxy(121.43,85.83)}\pgfstroke
\pgfmoveto{\pgfxy(121.43,85.83)}\pgflineto{\pgfxy(129.69,85.83)}\pgfstroke
\pgfmoveto{\pgfxy(127.21,69.31)}\pgflineto{\pgfxy(127.21,69.31)}\pgfstroke
\pgfmoveto{\pgfxy(105.44,69.85)}\pgflineto{\pgfxy(104.36,69.85)}\pgflineto{\pgfxy(104.36,68.77)}\pgflineto{\pgfxy(105.44,68.77)}\pgfclosepath\pgffill
\pgfmoveto{\pgfxy(105.44,69.85)}\pgflineto{\pgfxy(104.36,69.85)}\pgflineto{\pgfxy(104.36,68.77)}\pgflineto{\pgfxy(105.44,68.77)}\pgfclosepath\pgfstroke
\pgfmoveto{\pgfxy(113.70,69.85)}\pgflineto{\pgfxy(112.62,69.85)}\pgflineto{\pgfxy(112.62,68.77)}\pgflineto{\pgfxy(113.70,68.77)}\pgfclosepath\pgffill
\pgfmoveto{\pgfxy(113.70,69.85)}\pgflineto{\pgfxy(112.62,69.85)}\pgflineto{\pgfxy(112.62,68.77)}\pgflineto{\pgfxy(113.70,68.77)}\pgfclosepath\pgfstroke
\pgfmoveto{\pgfxy(113.70,78.11)}\pgflineto{\pgfxy(112.62,78.11)}\pgflineto{\pgfxy(112.62,77.03)}\pgflineto{\pgfxy(113.70,77.03)}\pgfclosepath\pgffill
\pgfmoveto{\pgfxy(113.70,78.11)}\pgflineto{\pgfxy(112.62,78.11)}\pgflineto{\pgfxy(112.62,77.03)}\pgflineto{\pgfxy(113.70,77.03)}\pgfclosepath\pgfstroke
\pgfmoveto{\pgfxy(121.97,78.11)}\pgflineto{\pgfxy(120.89,78.11)}\pgflineto{\pgfxy(120.89,77.03)}\pgflineto{\pgfxy(121.97,77.03)}\pgfclosepath\pgffill
\pgfmoveto{\pgfxy(121.97,78.11)}\pgflineto{\pgfxy(120.89,78.11)}\pgflineto{\pgfxy(120.89,77.03)}\pgflineto{\pgfxy(121.97,77.03)}\pgfclosepath\pgfstroke
\pgfmoveto{\pgfxy(121.97,69.85)}\pgflineto{\pgfxy(120.89,69.85)}\pgflineto{\pgfxy(120.89,68.77)}\pgflineto{\pgfxy(121.97,68.77)}\pgfclosepath\pgffill
\pgfmoveto{\pgfxy(121.97,69.85)}\pgflineto{\pgfxy(120.89,69.85)}\pgflineto{\pgfxy(120.89,68.77)}\pgflineto{\pgfxy(121.97,68.77)}\pgfclosepath\pgfstroke
\pgfmoveto{\pgfxy(121.97,86.37)}\pgflineto{\pgfxy(120.89,86.37)}\pgflineto{\pgfxy(120.89,85.29)}\pgflineto{\pgfxy(121.97,85.29)}\pgfclosepath\pgffill
\pgfmoveto{\pgfxy(121.97,86.37)}\pgflineto{\pgfxy(120.89,86.37)}\pgflineto{\pgfxy(120.89,85.29)}\pgflineto{\pgfxy(121.97,85.29)}\pgfclosepath\pgfstroke
\pgfsetdash{{0.10mm}{0.50mm}}{0mm}\pgfsetlinewidth{0.10mm}\pgfmoveto{\pgfxy(55.31,77.57)}\pgflineto{\pgfxy(63.58,85.83)}\pgfstroke
\pgfmoveto{\pgfxy(55.31,81.70)}\pgflineto{\pgfxy(59.44,85.83)}\pgfstroke
\pgfmoveto{\pgfxy(59.44,77.57)}\pgflineto{\pgfxy(63.58,81.70)}\pgfstroke
\pgfmoveto{\pgfxy(55.31,79.64)}\pgflineto{\pgfxy(61.51,85.83)}\pgfstroke
\pgfmoveto{\pgfxy(55.31,83.77)}\pgflineto{\pgfxy(57.38,85.83)}\pgfstroke
\pgfmoveto{\pgfxy(57.38,77.57)}\pgflineto{\pgfxy(63.58,83.77)}\pgfstroke
\pgfmoveto{\pgfxy(61.51,77.57)}\pgflineto{\pgfxy(63.58,79.64)}\pgfstroke
\pgfputat{\pgfxy(55.72,81.29)}{\pgfbox[bottom,left]{$R^\star_{(k^\star,l^\star)}$}}
\pgfputat{\pgfxy(54.48,66.00)}{\pgfbox[bottom,left]{$k^\star$}}
\pgfputat{\pgfxy(35.48,76.74)}{\pgfbox[bottom,left]{$l^\star$}}
\pgfsetlinewidth{0.20mm}\pgfmoveto{\pgfxy(55.31,97.82)}\pgflineto{\pgfxy(55.31,97.82)}\pgfstroke
\pgfmoveto{\pgfxy(56.55,85.01)}\pgflineto{\pgfxy(56.55,85.01)}\pgfstroke
\pgfsetlinewidth{0.10mm}\pgfmoveto{\pgfxy(55.31,102.36)}\pgflineto{\pgfxy(55.31,94.10)}\pgfstroke
\pgfmoveto{\pgfxy(55.31,94.10)}\pgflineto{\pgfxy(55.31,85.83)}\pgfstroke
\pgfmoveto{\pgfxy(55.31,77.57)}\pgflineto{\pgfxy(55.31,69.31)}\pgfstroke
\pgfmoveto{\pgfxy(38.78,77.57)}\pgflineto{\pgfxy(55.31,77.57)}\pgfstroke
\pgfmoveto{\pgfxy(63.58,77.57)}\pgflineto{\pgfxy(71.84,77.57)}\pgfstroke
\pgfmoveto{\pgfxy(55.31,77.57)}\pgflineto{\pgfxy(63.58,77.57)}\pgfstroke
\pgfsetlinewidth{0.20mm}\pgfmoveto{\pgfxy(55.31,82.53)}\pgflineto{\pgfxy(55.31,82.53)}\pgfstroke
\pgfsetlinewidth{0.10mm}\pgfmoveto{\pgfxy(55.31,85.83)}\pgflineto{\pgfxy(55.31,77.57)}\pgfstroke
\pgfsetdash{}{0mm}\pgfsetlinewidth{0.20mm}\pgfmoveto{\pgfxy(104.90,77.57)}\pgflineto{\pgfxy(113.16,69.31)}\pgfstroke
\pgfmoveto{\pgfxy(113.16,77.57)}\pgflineto{\pgfxy(121.43,69.31)}\pgfstroke
\pgfmoveto{\pgfxy(113.16,85.83)}\pgflineto{\pgfxy(129.69,69.31)}\pgfstroke
\pgfmoveto{\pgfxy(121.43,85.83)}\pgflineto{\pgfxy(129.69,77.57)}\pgfstroke
\pgfmoveto{\pgfxy(121.43,94.10)}\pgflineto{\pgfxy(121.43,94.10)}\pgfstroke
\pgfmoveto{\pgfxy(121.43,94.10)}\pgflineto{\pgfxy(129.69,85.83)}\pgfstroke
\pgfmoveto{\pgfxy(109.03,77.57)}\pgflineto{\pgfxy(113.16,73.44)}\pgfstroke
\pgfmoveto{\pgfxy(104.90,73.44)}\pgflineto{\pgfxy(109.03,69.31)}\pgfstroke
\pgfmoveto{\pgfxy(113.16,73.44)}\pgflineto{\pgfxy(117.29,69.31)}\pgfstroke
\pgfmoveto{\pgfxy(113.16,81.70)}\pgflineto{\pgfxy(125.56,69.31)}\pgfstroke
\pgfmoveto{\pgfxy(117.29,85.83)}\pgflineto{\pgfxy(129.69,73.44)}\pgfstroke
\pgfmoveto{\pgfxy(121.43,89.97)}\pgflineto{\pgfxy(129.69,81.70)}\pgfstroke
\pgfmoveto{\pgfxy(125.56,94.10)}\pgflineto{\pgfxy(129.69,89.97)}\pgfstroke
\pgfmoveto{\pgfxy(127.62,94.10)}\pgflineto{\pgfxy(129.69,92.03)}\pgfstroke
\pgfmoveto{\pgfxy(123.49,94.10)}\pgflineto{\pgfxy(129.69,87.90)}\pgfstroke
\pgfmoveto{\pgfxy(121.43,91.62)}\pgflineto{\pgfxy(129.69,83.77)}\pgfstroke
\pgfmoveto{\pgfxy(121.43,87.90)}\pgflineto{\pgfxy(129.69,79.64)}\pgfstroke
\pgfmoveto{\pgfxy(119.36,85.83)}\pgflineto{\pgfxy(127.62,77.57)}\pgfstroke
\pgfmoveto{\pgfxy(115.23,85.83)}\pgflineto{\pgfxy(129.69,71.37)}\pgfstroke
\pgfmoveto{\pgfxy(127.62,77.57)}\pgflineto{\pgfxy(129.69,75.50)}\pgfstroke
\pgfmoveto{\pgfxy(113.16,83.77)}\pgflineto{\pgfxy(127.62,69.31)}\pgfstroke
\pgfmoveto{\pgfxy(113.16,79.64)}\pgflineto{\pgfxy(123.49,69.31)}\pgfstroke
\pgfmoveto{\pgfxy(113.16,75.50)}\pgflineto{\pgfxy(119.36,69.31)}\pgfstroke
\pgfmoveto{\pgfxy(106.96,77.57)}\pgflineto{\pgfxy(115.23,69.31)}\pgfstroke
\pgfmoveto{\pgfxy(104.90,75.50)}\pgflineto{\pgfxy(111.10,69.31)}\pgfstroke
\pgfmoveto{\pgfxy(111.10,77.57)}\pgflineto{\pgfxy(113.16,75.50)}\pgfstroke
\pgfmoveto{\pgfxy(104.90,71.37)}\pgflineto{\pgfxy(106.96,69.31)}\pgfstroke
\pgfsetlinewidth{0.10mm}\pgfmoveto{\pgfxy(39.77,52.78)}\pgflineto{\pgfxy(39.77,52.78)}\pgfstroke
\pgfmoveto{\pgfxy(80.27,38.73)}\pgflineto{\pgfxy(88.53,38.73)}\pgfstroke
\pgfmoveto{\pgfxy(88.53,38.73)}\pgflineto{\pgfxy(87.33,39.03)}\pgflineto{\pgfxy(87.93,38.73)}\pgflineto{\pgfxy(87.33,38.43)}\pgflineto{\pgfxy(88.53,38.73)}\pgfclosepath\pgffill
\pgfmoveto{\pgfxy(88.53,38.73)}\pgflineto{\pgfxy(87.33,39.03)}\pgflineto{\pgfxy(87.93,38.73)}\pgflineto{\pgfxy(87.33,38.43)}\pgflineto{\pgfxy(88.53,38.73)}\pgfclosepath\pgfstroke
\pgfsetlinewidth{0.20mm}\pgfmoveto{\pgfxy(38.95,55.26)}\pgflineto{\pgfxy(38.95,22.20)}\pgfstroke
\pgfmoveto{\pgfxy(38.95,55.26)}\pgflineto{\pgfxy(38.65,54.06)}\pgflineto{\pgfxy(38.95,54.66)}\pgflineto{\pgfxy(39.25,54.06)}\pgflineto{\pgfxy(38.95,55.26)}\pgfclosepath\pgffill
\pgfmoveto{\pgfxy(38.95,55.26)}\pgflineto{\pgfxy(38.65,54.06)}\pgflineto{\pgfxy(38.95,54.66)}\pgflineto{\pgfxy(39.25,54.06)}\pgflineto{\pgfxy(38.95,55.26)}\pgfclosepath\pgfstroke
\pgfmoveto{\pgfxy(38.95,22.20)}\pgflineto{\pgfxy(72.00,22.20)}\pgfstroke
\pgfmoveto{\pgfxy(72.00,22.20)}\pgflineto{\pgfxy(70.80,22.50)}\pgflineto{\pgfxy(71.40,22.20)}\pgflineto{\pgfxy(70.80,21.90)}\pgflineto{\pgfxy(72.00,22.20)}\pgfclosepath\pgffill
\pgfmoveto{\pgfxy(72.00,22.20)}\pgflineto{\pgfxy(70.80,22.50)}\pgflineto{\pgfxy(71.40,22.20)}\pgflineto{\pgfxy(70.80,21.90)}\pgflineto{\pgfxy(72.00,22.20)}\pgfclosepath\pgfstroke
\pgfsetdash{{0.10mm}{0.50mm}}{0mm}\pgfsetlinewidth{0.10mm}\pgfmoveto{\pgfxy(47.21,55.26)}\pgflineto{\pgfxy(47.21,22.20)}\pgfstroke
\pgfmoveto{\pgfxy(55.48,55.26)}\pgflineto{\pgfxy(55.48,22.20)}\pgfstroke
\pgfmoveto{\pgfxy(63.74,55.26)}\pgflineto{\pgfxy(63.74,22.20)}\pgfstroke
\pgfmoveto{\pgfxy(38.95,30.46)}\pgflineto{\pgfxy(72.00,30.46)}\pgfstroke
\pgfmoveto{\pgfxy(38.95,38.73)}\pgflineto{\pgfxy(72.00,38.73)}\pgfstroke
\pgfmoveto{\pgfxy(38.95,46.99)}\pgflineto{\pgfxy(72.00,46.99)}\pgfstroke
\pgfsetdash{}{0mm}\pgfmoveto{\pgfxy(38.95,22.20)}\pgflineto{\pgfxy(72.00,55.26)}\pgfstroke
\pgfmoveto{\pgfxy(39.27,22.20)}\pgflineto{\pgfxy(39.91,22.84)}\pgflineto{\pgfxy(39.59,23.16)}\pgflineto{\pgfxy(38.95,22.52)}\pgflineto{\pgfxy(38.31,23.16)}\pgflineto{\pgfxy(37.99,22.84)}\pgflineto{\pgfxy(38.63,22.20)}\pgflineto{\pgfxy(37.99,21.56)}\pgflineto{\pgfxy(38.31,21.24)}\pgflineto{\pgfxy(38.95,21.88)}\pgflineto{\pgfxy(39.59,21.24)}\pgflineto{\pgfxy(39.91,21.56)}\pgfclosepath\pgffill
\pgfmoveto{\pgfxy(39.27,22.20)}\pgflineto{\pgfxy(39.91,22.84)}\pgflineto{\pgfxy(39.59,23.16)}\pgflineto{\pgfxy(38.95,22.52)}\pgflineto{\pgfxy(38.31,23.16)}\pgflineto{\pgfxy(37.99,22.84)}\pgflineto{\pgfxy(38.63,22.20)}\pgflineto{\pgfxy(37.99,21.56)}\pgflineto{\pgfxy(38.31,21.24)}\pgflineto{\pgfxy(38.95,21.88)}\pgflineto{\pgfxy(39.59,21.24)}\pgflineto{\pgfxy(39.91,21.56)}\pgfclosepath\pgfstroke
\pgfmoveto{\pgfxy(47.53,30.46)}\pgflineto{\pgfxy(48.17,31.10)}\pgflineto{\pgfxy(47.85,31.42)}\pgflineto{\pgfxy(47.21,30.78)}\pgflineto{\pgfxy(46.57,31.42)}\pgflineto{\pgfxy(46.25,31.10)}\pgflineto{\pgfxy(46.89,30.46)}\pgflineto{\pgfxy(46.25,29.82)}\pgflineto{\pgfxy(46.57,29.50)}\pgflineto{\pgfxy(47.21,30.14)}\pgflineto{\pgfxy(47.85,29.50)}\pgflineto{\pgfxy(48.17,29.82)}\pgfclosepath\pgffill
\pgfmoveto{\pgfxy(47.53,30.46)}\pgflineto{\pgfxy(48.17,31.10)}\pgflineto{\pgfxy(47.85,31.42)}\pgflineto{\pgfxy(47.21,30.78)}\pgflineto{\pgfxy(46.57,31.42)}\pgflineto{\pgfxy(46.25,31.10)}\pgflineto{\pgfxy(46.89,30.46)}\pgflineto{\pgfxy(46.25,29.82)}\pgflineto{\pgfxy(46.57,29.50)}\pgflineto{\pgfxy(47.21,30.14)}\pgflineto{\pgfxy(47.85,29.50)}\pgflineto{\pgfxy(48.17,29.82)}\pgfclosepath\pgfstroke
\pgfmoveto{\pgfxy(55.80,38.73)}\pgflineto{\pgfxy(56.44,39.37)}\pgflineto{\pgfxy(56.12,39.69)}\pgflineto{\pgfxy(55.48,39.05)}\pgflineto{\pgfxy(54.84,39.69)}\pgflineto{\pgfxy(54.52,39.37)}\pgflineto{\pgfxy(55.16,38.73)}\pgflineto{\pgfxy(54.52,38.09)}\pgflineto{\pgfxy(54.84,37.77)}\pgflineto{\pgfxy(55.48,38.41)}\pgflineto{\pgfxy(56.12,37.77)}\pgflineto{\pgfxy(56.44,38.09)}\pgfclosepath\pgffill
\pgfmoveto{\pgfxy(55.80,38.73)}\pgflineto{\pgfxy(56.44,39.37)}\pgflineto{\pgfxy(56.12,39.69)}\pgflineto{\pgfxy(55.48,39.05)}\pgflineto{\pgfxy(54.84,39.69)}\pgflineto{\pgfxy(54.52,39.37)}\pgflineto{\pgfxy(55.16,38.73)}\pgflineto{\pgfxy(54.52,38.09)}\pgflineto{\pgfxy(54.84,37.77)}\pgflineto{\pgfxy(55.48,38.41)}\pgflineto{\pgfxy(56.12,37.77)}\pgflineto{\pgfxy(56.44,38.09)}\pgfclosepath\pgfstroke
\pgfmoveto{\pgfxy(64.06,46.99)}\pgflineto{\pgfxy(64.70,47.63)}\pgflineto{\pgfxy(64.38,47.95)}\pgflineto{\pgfxy(63.74,47.31)}\pgflineto{\pgfxy(63.10,47.95)}\pgflineto{\pgfxy(62.78,47.63)}\pgflineto{\pgfxy(63.42,46.99)}\pgflineto{\pgfxy(62.78,46.35)}\pgflineto{\pgfxy(63.10,46.03)}\pgflineto{\pgfxy(63.74,46.67)}\pgflineto{\pgfxy(64.38,46.03)}\pgflineto{\pgfxy(64.70,46.35)}\pgfclosepath\pgffill
\pgfmoveto{\pgfxy(64.06,46.99)}\pgflineto{\pgfxy(64.70,47.63)}\pgflineto{\pgfxy(64.38,47.95)}\pgflineto{\pgfxy(63.74,47.31)}\pgflineto{\pgfxy(63.10,47.95)}\pgflineto{\pgfxy(62.78,47.63)}\pgflineto{\pgfxy(63.42,46.99)}\pgflineto{\pgfxy(62.78,46.35)}\pgflineto{\pgfxy(63.10,46.03)}\pgflineto{\pgfxy(63.74,46.67)}\pgflineto{\pgfxy(64.38,46.03)}\pgflineto{\pgfxy(64.70,46.35)}\pgfclosepath\pgfstroke
\pgfsetlinewidth{0.20mm}\pgfmoveto{\pgfxy(38.95,26.33)}\pgflineto{\pgfxy(67.87,55.26)}\pgfstroke
\pgfmoveto{\pgfxy(38.95,30.46)}\pgflineto{\pgfxy(63.74,55.26)}\pgfstroke
\pgfmoveto{\pgfxy(38.95,34.59)}\pgflineto{\pgfxy(59.61,55.26)}\pgfstroke
\pgfmoveto{\pgfxy(38.95,38.73)}\pgflineto{\pgfxy(55.48,55.26)}\pgfstroke
\pgfmoveto{\pgfxy(38.95,42.86)}\pgflineto{\pgfxy(51.34,55.26)}\pgfstroke
\pgfmoveto{\pgfxy(38.95,46.99)}\pgflineto{\pgfxy(47.21,55.26)}\pgfstroke
\pgfmoveto{\pgfxy(38.95,51.12)}\pgflineto{\pgfxy(43.08,55.26)}\pgfstroke
\pgfmoveto{\pgfxy(43.08,26.33)}\pgflineto{\pgfxy(47.21,22.20)}\pgfstroke
\pgfmoveto{\pgfxy(47.21,30.46)}\pgflineto{\pgfxy(55.48,22.20)}\pgfstroke
\pgfmoveto{\pgfxy(51.34,34.59)}\pgflineto{\pgfxy(63.74,22.20)}\pgfstroke
\pgfmoveto{\pgfxy(55.48,38.73)}\pgflineto{\pgfxy(72.00,22.20)}\pgfstroke
\pgfmoveto{\pgfxy(59.61,42.86)}\pgflineto{\pgfxy(72.00,30.46)}\pgfstroke
\pgfmoveto{\pgfxy(67.87,51.12)}\pgflineto{\pgfxy(72.00,46.99)}\pgfstroke
\pgfmoveto{\pgfxy(69.94,53.19)}\pgflineto{\pgfxy(72.00,51.12)}\pgfstroke
\pgfmoveto{\pgfxy(72.00,40.38)}\pgflineto{\pgfxy(72.00,40.38)}\pgfstroke
\pgfmoveto{\pgfxy(72.00,42.86)}\pgflineto{\pgfxy(65.81,49.06)}\pgfstroke
\pgfmoveto{\pgfxy(63.74,46.99)}\pgflineto{\pgfxy(72.00,38.73)}\pgfstroke
\pgfmoveto{\pgfxy(72.00,34.59)}\pgflineto{\pgfxy(61.67,44.93)}\pgfstroke
\pgfmoveto{\pgfxy(72.00,26.33)}\pgflineto{\pgfxy(57.54,40.79)}\pgfstroke
\pgfmoveto{\pgfxy(67.87,22.20)}\pgflineto{\pgfxy(53.41,36.66)}\pgfstroke
\pgfmoveto{\pgfxy(59.61,22.20)}\pgflineto{\pgfxy(49.28,32.53)}\pgfstroke
\pgfmoveto{\pgfxy(51.34,22.20)}\pgflineto{\pgfxy(45.15,28.40)}\pgfstroke
\pgfmoveto{\pgfxy(43.08,22.20)}\pgflineto{\pgfxy(41.01,24.26)}\pgfstroke
\pgfmoveto{\pgfxy(38.95,55.26)}\pgflineto{\pgfxy(38.95,22.20)}\pgfstroke
\pgfmoveto{\pgfxy(38.95,55.26)}\pgflineto{\pgfxy(38.65,54.06)}\pgflineto{\pgfxy(38.95,54.66)}\pgflineto{\pgfxy(39.25,54.06)}\pgflineto{\pgfxy(38.95,55.26)}\pgfclosepath\pgffill
\pgfmoveto{\pgfxy(38.95,55.26)}\pgflineto{\pgfxy(38.65,54.06)}\pgflineto{\pgfxy(38.95,54.66)}\pgflineto{\pgfxy(39.25,54.06)}\pgflineto{\pgfxy(38.95,55.26)}\pgfclosepath\pgfstroke
\pgfmoveto{\pgfxy(38.95,22.20)}\pgflineto{\pgfxy(72.00,22.20)}\pgfstroke
\pgfmoveto{\pgfxy(72.00,22.20)}\pgflineto{\pgfxy(70.80,22.50)}\pgflineto{\pgfxy(71.40,22.20)}\pgflineto{\pgfxy(70.80,21.90)}\pgflineto{\pgfxy(72.00,22.20)}\pgfclosepath\pgffill
\pgfmoveto{\pgfxy(72.00,22.20)}\pgflineto{\pgfxy(70.80,22.50)}\pgflineto{\pgfxy(71.40,22.20)}\pgflineto{\pgfxy(70.80,21.90)}\pgflineto{\pgfxy(72.00,22.20)}\pgfclosepath\pgfstroke
\pgfsetdash{{0.10mm}{0.50mm}}{0mm}\pgfsetlinewidth{0.10mm}\pgfmoveto{\pgfxy(47.21,55.26)}\pgflineto{\pgfxy(47.21,22.20)}\pgfstroke
\pgfmoveto{\pgfxy(55.48,55.26)}\pgflineto{\pgfxy(55.48,22.20)}\pgfstroke
\pgfmoveto{\pgfxy(63.74,55.26)}\pgflineto{\pgfxy(63.74,22.20)}\pgfstroke
\pgfmoveto{\pgfxy(38.95,30.46)}\pgflineto{\pgfxy(72.00,30.46)}\pgfstroke
\pgfmoveto{\pgfxy(38.95,38.73)}\pgflineto{\pgfxy(72.00,38.73)}\pgfstroke
\pgfmoveto{\pgfxy(38.95,46.99)}\pgflineto{\pgfxy(72.00,46.99)}\pgfstroke
\pgfsetdash{}{0mm}\pgfmoveto{\pgfxy(38.95,22.20)}\pgflineto{\pgfxy(72.00,55.26)}\pgfstroke
\pgfmoveto{\pgfxy(39.27,22.20)}\pgflineto{\pgfxy(39.91,22.84)}\pgflineto{\pgfxy(39.59,23.16)}\pgflineto{\pgfxy(38.95,22.52)}\pgflineto{\pgfxy(38.31,23.16)}\pgflineto{\pgfxy(37.99,22.84)}\pgflineto{\pgfxy(38.63,22.20)}\pgflineto{\pgfxy(37.99,21.56)}\pgflineto{\pgfxy(38.31,21.24)}\pgflineto{\pgfxy(38.95,21.88)}\pgflineto{\pgfxy(39.59,21.24)}\pgflineto{\pgfxy(39.91,21.56)}\pgfclosepath\pgffill
\pgfmoveto{\pgfxy(39.27,22.20)}\pgflineto{\pgfxy(39.91,22.84)}\pgflineto{\pgfxy(39.59,23.16)}\pgflineto{\pgfxy(38.95,22.52)}\pgflineto{\pgfxy(38.31,23.16)}\pgflineto{\pgfxy(37.99,22.84)}\pgflineto{\pgfxy(38.63,22.20)}\pgflineto{\pgfxy(37.99,21.56)}\pgflineto{\pgfxy(38.31,21.24)}\pgflineto{\pgfxy(38.95,21.88)}\pgflineto{\pgfxy(39.59,21.24)}\pgflineto{\pgfxy(39.91,21.56)}\pgfclosepath\pgfstroke
\pgfmoveto{\pgfxy(47.53,30.46)}\pgflineto{\pgfxy(48.17,31.10)}\pgflineto{\pgfxy(47.85,31.42)}\pgflineto{\pgfxy(47.21,30.78)}\pgflineto{\pgfxy(46.57,31.42)}\pgflineto{\pgfxy(46.25,31.10)}\pgflineto{\pgfxy(46.89,30.46)}\pgflineto{\pgfxy(46.25,29.82)}\pgflineto{\pgfxy(46.57,29.50)}\pgflineto{\pgfxy(47.21,30.14)}\pgflineto{\pgfxy(47.85,29.50)}\pgflineto{\pgfxy(48.17,29.82)}\pgfclosepath\pgffill
\pgfmoveto{\pgfxy(47.53,30.46)}\pgflineto{\pgfxy(48.17,31.10)}\pgflineto{\pgfxy(47.85,31.42)}\pgflineto{\pgfxy(47.21,30.78)}\pgflineto{\pgfxy(46.57,31.42)}\pgflineto{\pgfxy(46.25,31.10)}\pgflineto{\pgfxy(46.89,30.46)}\pgflineto{\pgfxy(46.25,29.82)}\pgflineto{\pgfxy(46.57,29.50)}\pgflineto{\pgfxy(47.21,30.14)}\pgflineto{\pgfxy(47.85,29.50)}\pgflineto{\pgfxy(48.17,29.82)}\pgfclosepath\pgfstroke
\pgfmoveto{\pgfxy(55.80,38.73)}\pgflineto{\pgfxy(56.44,39.37)}\pgflineto{\pgfxy(56.12,39.69)}\pgflineto{\pgfxy(55.48,39.05)}\pgflineto{\pgfxy(54.84,39.69)}\pgflineto{\pgfxy(54.52,39.37)}\pgflineto{\pgfxy(55.16,38.73)}\pgflineto{\pgfxy(54.52,38.09)}\pgflineto{\pgfxy(54.84,37.77)}\pgflineto{\pgfxy(55.48,38.41)}\pgflineto{\pgfxy(56.12,37.77)}\pgflineto{\pgfxy(56.44,38.09)}\pgfclosepath\pgffill
\pgfmoveto{\pgfxy(55.80,38.73)}\pgflineto{\pgfxy(56.44,39.37)}\pgflineto{\pgfxy(56.12,39.69)}\pgflineto{\pgfxy(55.48,39.05)}\pgflineto{\pgfxy(54.84,39.69)}\pgflineto{\pgfxy(54.52,39.37)}\pgflineto{\pgfxy(55.16,38.73)}\pgflineto{\pgfxy(54.52,38.09)}\pgflineto{\pgfxy(54.84,37.77)}\pgflineto{\pgfxy(55.48,38.41)}\pgflineto{\pgfxy(56.12,37.77)}\pgflineto{\pgfxy(56.44,38.09)}\pgfclosepath\pgfstroke
\pgfmoveto{\pgfxy(64.06,46.99)}\pgflineto{\pgfxy(64.70,47.63)}\pgflineto{\pgfxy(64.38,47.95)}\pgflineto{\pgfxy(63.74,47.31)}\pgflineto{\pgfxy(63.10,47.95)}\pgflineto{\pgfxy(62.78,47.63)}\pgflineto{\pgfxy(63.42,46.99)}\pgflineto{\pgfxy(62.78,46.35)}\pgflineto{\pgfxy(63.10,46.03)}\pgflineto{\pgfxy(63.74,46.67)}\pgflineto{\pgfxy(64.38,46.03)}\pgflineto{\pgfxy(64.70,46.35)}\pgfclosepath\pgffill
\pgfmoveto{\pgfxy(64.06,46.99)}\pgflineto{\pgfxy(64.70,47.63)}\pgflineto{\pgfxy(64.38,47.95)}\pgflineto{\pgfxy(63.74,47.31)}\pgflineto{\pgfxy(63.10,47.95)}\pgflineto{\pgfxy(62.78,47.63)}\pgflineto{\pgfxy(63.42,46.99)}\pgflineto{\pgfxy(62.78,46.35)}\pgflineto{\pgfxy(63.10,46.03)}\pgflineto{\pgfxy(63.74,46.67)}\pgflineto{\pgfxy(64.38,46.03)}\pgflineto{\pgfxy(64.70,46.35)}\pgfclosepath\pgfstroke
\pgfsetlinewidth{0.20mm}\pgfmoveto{\pgfxy(38.95,26.33)}\pgflineto{\pgfxy(67.87,55.26)}\pgfstroke
\pgfmoveto{\pgfxy(38.95,30.46)}\pgflineto{\pgfxy(63.74,55.26)}\pgfstroke
\pgfmoveto{\pgfxy(38.95,34.59)}\pgflineto{\pgfxy(59.61,55.26)}\pgfstroke
\pgfmoveto{\pgfxy(38.95,38.73)}\pgflineto{\pgfxy(55.48,55.26)}\pgfstroke
\pgfmoveto{\pgfxy(38.95,42.86)}\pgflineto{\pgfxy(51.34,55.26)}\pgfstroke
\pgfmoveto{\pgfxy(38.95,46.99)}\pgflineto{\pgfxy(47.21,55.26)}\pgfstroke
\pgfmoveto{\pgfxy(38.95,51.12)}\pgflineto{\pgfxy(43.08,55.26)}\pgfstroke
\pgfmoveto{\pgfxy(43.08,26.33)}\pgflineto{\pgfxy(47.21,22.20)}\pgfstroke
\pgfmoveto{\pgfxy(47.21,30.46)}\pgflineto{\pgfxy(55.48,22.20)}\pgfstroke
\pgfmoveto{\pgfxy(51.34,34.59)}\pgflineto{\pgfxy(63.74,22.20)}\pgfstroke
\pgfmoveto{\pgfxy(55.48,38.73)}\pgflineto{\pgfxy(72.00,22.20)}\pgfstroke
\pgfmoveto{\pgfxy(59.61,42.86)}\pgflineto{\pgfxy(72.00,30.46)}\pgfstroke
\pgfmoveto{\pgfxy(125.72,51.12)}\pgflineto{\pgfxy(129.86,46.99)}\pgfstroke
\pgfmoveto{\pgfxy(69.94,53.19)}\pgflineto{\pgfxy(72.00,51.12)}\pgfstroke
\pgfmoveto{\pgfxy(72.00,40.38)}\pgflineto{\pgfxy(72.00,40.38)}\pgfstroke
\pgfmoveto{\pgfxy(72.00,42.86)}\pgflineto{\pgfxy(65.81,49.06)}\pgfstroke
\pgfmoveto{\pgfxy(63.74,46.99)}\pgflineto{\pgfxy(72.00,38.73)}\pgfstroke
\pgfmoveto{\pgfxy(72.00,34.59)}\pgflineto{\pgfxy(61.67,44.93)}\pgfstroke
\pgfmoveto{\pgfxy(72.00,26.33)}\pgflineto{\pgfxy(57.54,40.79)}\pgfstroke
\pgfmoveto{\pgfxy(67.87,22.20)}\pgflineto{\pgfxy(53.41,36.66)}\pgfstroke
\pgfmoveto{\pgfxy(59.61,22.20)}\pgflineto{\pgfxy(49.28,32.53)}\pgfstroke
\pgfmoveto{\pgfxy(51.34,22.20)}\pgflineto{\pgfxy(45.15,28.40)}\pgfstroke
\pgfmoveto{\pgfxy(43.08,22.20)}\pgflineto{\pgfxy(41.01,24.26)}\pgfstroke
\pgfsetlinewidth{0.10mm}\pgfmoveto{\pgfxy(96.80,55.26)}\pgflineto{\pgfxy(96.80,22.20)}\pgfstroke
\pgfmoveto{\pgfxy(96.80,55.26)}\pgflineto{\pgfxy(96.50,54.06)}\pgflineto{\pgfxy(96.80,54.66)}\pgflineto{\pgfxy(97.10,54.06)}\pgflineto{\pgfxy(96.80,55.26)}\pgfclosepath\pgffill
\pgfmoveto{\pgfxy(96.80,55.26)}\pgflineto{\pgfxy(96.50,54.06)}\pgflineto{\pgfxy(96.80,54.66)}\pgflineto{\pgfxy(97.10,54.06)}\pgflineto{\pgfxy(96.80,55.26)}\pgfclosepath\pgfstroke
\pgfmoveto{\pgfxy(96.80,22.20)}\pgflineto{\pgfxy(129.86,22.20)}\pgfstroke
\pgfmoveto{\pgfxy(129.86,22.20)}\pgflineto{\pgfxy(128.66,22.50)}\pgflineto{\pgfxy(129.26,22.20)}\pgflineto{\pgfxy(128.66,21.90)}\pgflineto{\pgfxy(129.86,22.20)}\pgfclosepath\pgffill
\pgfmoveto{\pgfxy(129.86,22.20)}\pgflineto{\pgfxy(128.66,22.50)}\pgflineto{\pgfxy(129.26,22.20)}\pgflineto{\pgfxy(128.66,21.90)}\pgflineto{\pgfxy(129.86,22.20)}\pgfclosepath\pgfstroke
\pgfsetdash{{0.10mm}{0.50mm}}{0mm}\pgfmoveto{\pgfxy(105.06,55.26)}\pgflineto{\pgfxy(105.06,22.20)}\pgfstroke
\pgfmoveto{\pgfxy(113.33,55.26)}\pgflineto{\pgfxy(113.33,22.20)}\pgfstroke
\pgfmoveto{\pgfxy(121.59,55.26)}\pgflineto{\pgfxy(121.59,22.20)}\pgfstroke
\pgfmoveto{\pgfxy(96.80,46.99)}\pgflineto{\pgfxy(129.86,46.99)}\pgfstroke
\pgfmoveto{\pgfxy(96.80,38.73)}\pgflineto{\pgfxy(129.86,38.73)}\pgfstroke
\pgfmoveto{\pgfxy(96.80,30.46)}\pgflineto{\pgfxy(129.86,30.46)}\pgfstroke
\pgfsetdash{}{0mm}\pgfsetlinewidth{0.20mm}\pgfmoveto{\pgfxy(69.94,53.19)}\pgflineto{\pgfxy(72.00,51.12)}\pgfstroke
\pgfsetlinewidth{0.30mm}\pgfmoveto{\pgfxy(96.80,22.20)}\pgflineto{\pgfxy(129.86,55.26)}\pgfstroke
\pgfsetlinewidth{0.20mm}\pgfmoveto{\pgfxy(96.80,26.33)}\pgflineto{\pgfxy(125.72,55.26)}\pgfstroke
\pgfmoveto{\pgfxy(96.80,30.46)}\pgflineto{\pgfxy(121.59,55.26)}\pgfstroke
\pgfmoveto{\pgfxy(96.80,34.59)}\pgflineto{\pgfxy(117.46,55.26)}\pgfstroke
\pgfmoveto{\pgfxy(96.80,38.73)}\pgflineto{\pgfxy(113.33,55.26)}\pgfstroke
\pgfmoveto{\pgfxy(96.80,42.86)}\pgflineto{\pgfxy(109.20,55.26)}\pgfstroke
\pgfmoveto{\pgfxy(96.80,46.99)}\pgflineto{\pgfxy(105.06,55.26)}\pgfstroke
\pgfmoveto{\pgfxy(96.80,51.12)}\pgflineto{\pgfxy(100.93,55.26)}\pgfstroke
\pgfmoveto{\pgfxy(100.93,22.20)}\pgflineto{\pgfxy(98.86,24.26)}\pgfstroke
\pgfmoveto{\pgfxy(100.93,26.33)}\pgflineto{\pgfxy(105.06,22.20)}\pgfstroke
\pgfmoveto{\pgfxy(109.20,22.20)}\pgflineto{\pgfxy(103.00,28.40)}\pgfstroke
\pgfmoveto{\pgfxy(105.06,30.46)}\pgflineto{\pgfxy(113.33,22.20)}\pgfstroke
\pgfmoveto{\pgfxy(117.46,22.20)}\pgflineto{\pgfxy(107.13,32.53)}\pgfstroke
\pgfmoveto{\pgfxy(109.20,34.59)}\pgflineto{\pgfxy(121.59,22.20)}\pgfstroke
\pgfmoveto{\pgfxy(125.72,22.20)}\pgflineto{\pgfxy(111.26,36.66)}\pgfstroke
\pgfmoveto{\pgfxy(113.33,38.73)}\pgflineto{\pgfxy(129.86,22.20)}\pgfstroke
\pgfmoveto{\pgfxy(129.86,26.33)}\pgflineto{\pgfxy(115.39,40.79)}\pgfstroke
\pgfmoveto{\pgfxy(117.46,42.86)}\pgflineto{\pgfxy(129.86,30.46)}\pgfstroke
\pgfmoveto{\pgfxy(129.86,34.59)}\pgflineto{\pgfxy(119.53,44.93)}\pgfstroke
\pgfmoveto{\pgfxy(121.59,46.99)}\pgflineto{\pgfxy(129.86,38.73)}\pgfstroke
\pgfmoveto{\pgfxy(129.86,42.86)}\pgflineto{\pgfxy(123.66,49.06)}\pgfstroke
\pgfmoveto{\pgfxy(127.79,53.19)}\pgflineto{\pgfxy(129.86,51.12)}\pgfstroke
\pgfputat{\pgfxy(80.10,87.49)}{\pgfbox[bottom,left]{$\textnormal{Step 1}$}}
\pgfputat{\pgfxy(80.10,40.38)}{\pgfbox[bottom,left]{$\textnormal{Step 3}$}}
\pgfsetlinewidth{0.10mm}\pgfmoveto{\pgfxy(88.37,62.69)}\pgflineto{\pgfxy(80.10,59.39)}\pgfstroke
\pgfmoveto{\pgfxy(80.10,59.39)}\pgflineto{\pgfxy(81.33,59.56)}\pgflineto{\pgfxy(80.66,59.61)}\pgflineto{\pgfxy(81.11,60.11)}\pgflineto{\pgfxy(80.10,59.39)}\pgfclosepath\pgffill
\pgfmoveto{\pgfxy(80.10,59.39)}\pgflineto{\pgfxy(81.33,59.56)}\pgflineto{\pgfxy(80.66,59.61)}\pgflineto{\pgfxy(81.11,60.11)}\pgflineto{\pgfxy(80.10,59.39)}\pgfclosepath\pgfstroke
\pgfputat{\pgfxy(79.28,63.52)}{\pgfbox[bottom,left]{$\textnormal{Step 2}$}}
\end{pgfpicture}%
%\graphic{spread_prob_1}
%\caption{Spread the point masses on the off-diagonal points over the corresponding coordinate rectangles.}
\caption{Three steps to construct a ``continuous'' intensity function.}
\label{f: spread_prob_1}
\end{figure}
%%%%%%%%%%%%%%%%%%%%%%%%%%%%%
As we aim to find a continuous intensity function over the entire space $[-\log n, \infty)^2$,
we exchange $k^\star$ and $l^\star$ in the expression of the point probability $P(X^\star_1=k^\star, X_2^\star=l^\star)$ from (\ref{d: MO_Geo_mass_fct_normalised})
by $s$ and $t$, respectively. 
E.g., suppose that $k^\star < l^\star$. Then we replace the integral in (\ref{p: spread_prob_1}) by
%\begin{equation}\label{p: spread_prob_2}
\[
\int \int_{\Rstar} \frac{1-p_{00} /q_2 - q_2 + p_{00}}{\log^2 (1/p_{00})} \,
\mathrm{e}^{-\frac{\log (p_{00} /q_2)}{\log p_{00}}\,s}
\mathrm{e}^{-\frac{\log q_2}{\log p_{00}}\,t} \De s \De t
%\end{equation}
%Evaluation of this new integral gives
%\begin{equation}\label{p: spread_prob_3}
=
\frac{ 1- p_{00}/q_2 - q_2 + p_{00}}{\log (p_{00}/q_2) \log q_2}\,P(X^\star_1=k^\star, X_2^\star=l^\star).
%\end{equation}
\]
The switch to variable $s$ and $t$ thus results only in the multiplication of the original point probability by a factor. 
The goal, however, is to integrate a function in $s$ and $t$ over $\Rstar$ and obtain the original point probability. This may
be achieved by simply dividing the integrand 
%in (\ref{p: spread_prob_2}) 
by the multiplying factor 
%found in 
. Hence, we rewrite the mean as follows
\[
\sum_{ (k^\star,l^\star) \in A^\star, k^\star \neq l^\star } 
\int \int_{\Rstar} \lambda^\star(s,t)\De s\De t + n\sum_{(k^\star, k^\star) \in A^\star} P\left(X_1^\star=k^\star, X_2^\star=k^\star\right),
\]
where
\begin{equation}\label{p: int_fct_above}
\lambda^\star(s,t) = 
\frac{\log (p_{00}/q_2) \log q_2}{\log^2(1/p_{00})}\, \mathrm{e}^{-\frac{\log (p_{00}/q_2)}{\log p_{00}}\,s} \mathrm{e}^{-\frac{\log q_2}{\log p_{00}}\,t}, \,
\forall (s,t) \in \Rstar\textnormal{ with }k^\star < l^\star.
\end{equation}
Analogously, we find
\begin{equation}\label{p: int_fct_below}
\lambda^\star(s,t) = 
\frac{\log (p_{00}/q_1) \log q_1}{\log^2(1/p_{00})}\, \mathrm{e}^{-\frac{\log q_1}{\log p_{00}}\,s} \mathrm{e}^{-\frac{\log (p_{00}/q_1)}{\log p_{00}}\,t}
, \, \forall (s,t) \in \Rstar\textnormal{ with }k^\star > l^\star.
\end{equation}
(\ref{p: int_fct_above}) and (\ref{p: int_fct_below}) supply suitable choices for the intensity function on coordinate 
rectangles lying above and below the diagonal, respectively. 
%Figure \ref{f: spread_prob_1} illustrates Step 1. 

\textbf{Step 2.} We expand $\lambda^\star(s,t)$ from (\ref{p: int_fct_above}) and (\ref{p: int_fct_below}) to the entire space (without the diagonal), i.e. we define
\[
\lambda^\star(s,t) := \left\{ \begin{array}{ll} 
\frac{\log (p_{00}/q_2) \log q_2}{\log^2(1/p_{00})}\, \mathrm{e}^{-\frac{\log (p_{00}/q_2)}{\log p_{00}}\,s} \mathrm{e}^{-\frac{\log q_2}{\log p_{00}}\,t} & \textnormal{ for } s<t,\\
\frac{\log (p_{00}/q_1) \log q_1}{\log^2(1/p_{00})} \, \mathrm{e}^{-\frac{\log q_1}{\log p_{00}}\,s} \mathrm{e}^{-\frac{\log (p_{00}/q_1)}{\log p_{00}}\,t} & \textnormal{ for } s>t,
\end{array}\right.  
\]
for all $(s,t) \in [-\log n, \infty)^2$
. %; see Figure \ref{f: spread_prob_2}. 
However, this adds surplus mass on the diagonal rectangles $R^\star_{k^\star, k^\star}$.
%%%%%%%%%%%%%%%%%%%%%%%%%%%%%
%\begin{figure}
%\graphic{spread_prob_2}
%\caption{Define the intensity functions determined for off-diagonal rectangles on the %entire space.}
%\label{f: spread_prob_2}
%\end{figure}
%%%%%%%%%%%%%%%%%%%%%%%%%%%%%

\textbf{Step 3.} We adjust for the surplus mass on the diagonal rectangles by subtracting it from the point probabilities of the diagonal lattice points $(k^\star,k^\star)$,
and accordingly rewrite the mean as follows:
\begin{equation}\label{p: spread_prob_5}
\int \int_{A^\star} \lambda_n^\star(s,t)\De s\De t 
+ n\sum_{(k^\star, k^\star) \in A^\star} \left\{ P(X_1^\star=k^\star, X_2^\star=k^\star) - \frac{1}{n}\int_{R^\star_{k^\star, k^\star}} \lambda_n^\star(s,t)\De s\De t\right\}.
\end{equation}
Computation of the term in curly brackets shows that the new mass that we put on the diagonal segments of each diagonal rectangle $R^\star_{k^\star,k^\star}$ is given by
\begin{equation}\label{p: spread_prob_6}
\frac{ \mathrm{e}^{-k^\star}}{n}\,(1-p_{00})\left[ \frac{\log (1/q_1 q_2)}{\log (1/p_{00})}-1\right]. 
\end{equation}
Note that this equals 
\[
\int_{k^\star}^{k^\star + \log(1/p_{00})} \frac{\mathrm{e}^{-s}}{n}\, \left[\frac{\log (1/q_1q_2)}{\log (1/p_{00})}-1 \right]\De s,           
\]
for each $k^\star \in E^\star$, where we have parameterised the intensity function on the diagonal as projection along the $s$-axis.
We thus define:
\begin{align}\label{d: ls}
\begin{split}
\lambda^\star(s,t)&= \left \{
\begin{array}{lll} 
\frac{\log (p_{00}/q_2) \log q_2}{\log^2(1/p_{00})}\, \mathrm{e}^{-\frac{\log (p_{00}/q_2)}{\log p_{00}}\,s} \mathrm{e}^{-\frac{\log q_2}{\log p_{00}}\,t} &\textnormal{ for }& s < t ,  \\
\frac{\log (p_{00}/q_1) \log q_1}{\log^2(1/p_{00})}\, \mathrm{e}^{-\frac{\log q_1}{\log p_{00}}\,s} \mathrm{e}^{-\frac{\log (p_{00}/q_1)}{\log p_{00}}\,t} &\textnormal{ for }& s > t,   
\end{array} \right.\\
\acute{\lambda}^\star(s)&=\frac{\log (p_{00}/q_1 q_2)}{\log (1/p_{00})}\, \mathrm{e}^{-s}\quad \textnormal{for } s = t.\\
\end{split}
\end{align}
%Figure \ref{f: spread_prob_3} illustrates this last step in the construction of $\lambda^\star$. 
%%%%%%%%%%%%%%%%%%%%%%%%%%%
%\begin{figure}
%\graphic{spread_prob_3}
%\caption{The mass on the diagonal lattice points is spread over the entire diagonal.} 
%\label{f: spread_prob_3}
%\end{figure}
%%%%%%%%%%%%%%%%%%%%%%%%%%%
The above construction guarantees the following:
\begin{prop}\label{t: MO_Geo_same_meas_rectangles}
Let $\bl^\star$, $\lambda^\star$ and $\acute{\lambda}^\star$ be defined by (\ref{d: MO_Geo_cont_int_meas}) and (\ref{d: ls}). Then,\\
(i)$\,\, \displaystyle  \bl^\star \left(\Rstar\right) = \bpis\left(\Rstar\right)$, for any $(k^\star,l^\star) \in E^\star$,\\
(ii)$\,\, \displaystyle \int_{[-\log n, \infty)^2} \lambda^\star(s,t)\De s\De t + \int_{-\log n}^\infty \acute{\lambda}^\star(s)\De s= n$. \\ \qed 
\end{prop}
\begin{remark}
%With the affine transformations
%\begin{eqnarray*}
%\phi:\, E^\star &\to& E\\
%(x_1,x_2)& \mapsto& \phi(x_1,x_2) = \left( \frac{x_1+ \log n}{\log\left(1/p_{00}\right)}, \frac{x_2+ \log n}{\log\left(1/p_{00}\right)}\right),\\ 
%\tau:\, [-\log n, \infty)& \to &[0,\infty)\\
%x &\mapsto& \tau(x) = \frac{x + \log n}{\log\left(1/p_{00}\right)},
%\end{eqnarray*}
%we may express the new intensity functions
It can readily be shown that the new intensity functions 
$\lambda^\star$ and $\acute{\lambda}^\star$ 
may be expressed 
in the original coordinate system
by
%, Since $\lambda^\star(x_1,x_2) $ $= \left(\log\left(1/p_{00}\right)\right)^{-2}\lambda(\phi(x_1,x_2))$, for $x_1 \neq x_2$, and $\acute{\lambda}^\star(x)$
%$=\left( \log\left(1/p_{00}\right)\right)^{-1}\acute{\lambda}(\tau(x))$, for $x_1 = x_2 = x$ so that
\begin{align*}
\lambda(x,y)&=\lambda_n(x,y) = \left\{ 
\begin{array}{lll}
n \log (q_2) \log \left(\frac{p_{00}}{q_2}\right)   p_{00}^x q_2^{y-x} &\textnormal{ for } x <y,\\
n \log (q_1) \log \left(\frac{p_{00}}{q_1}\right) q_1^{x-y} p_{00}^y &\textnormal{ for } x>y,
\end{array} \right.\\
\acute{\lambda}(x) &= \acute{\lambda}_n(x)= n \log \left( \frac{p_{00}}{q_1 q_2}\right) p_{00}^x \quad \textnormal{ for } x=y,
\end{align*} 
for any $(x,y) \in [0,\infty)^2$. We recognise a weighted and continuous version of $P(X_1 \ge k,X_2 \ge l)$ from (\ref{d: MO_survival}). 
%Note the dependence on the sample size $n$. Section \ref{s: MO_Geo_d2} will clarify that we also have $\lambda^\star = \lambda^\star_n$.
\end{remark}
%-----------------------------------------------------------------------------------------------------------------------------------------------------------------------
%-----------------------------------------------------------------------------------------------------------------------------------------------------------------------
\subsubsection{Assumptions on the distributional parameters}\label{s: MO_Geo_cont_nicer_int_fct}
% Need it even for result from next section, not only for results from section after that one. 
The continuous intensity measure $\bl^\star$ defined by (\ref{d: MO_Geo_cont_int_meas}) and (\ref{d: ls}) depends on the parameters $q_1$, $q_2$ and $p_{00}$ of the Marshall-Olkin geometric distribution.
Our aim is to determine a bound on the error for the approximation of the Poisson process with mean measure $\mathbb{E}\Xi^\star_{A^\star}$, living on the lattice $E^\star$,
by a Poisson process with mean measure $\bl^\star$. 
As Section \ref{s: MO_Geo_d2} will show, the 
probability of simultaneous success, $p_{11}$, for the Marshall-Olkin geometric distribution, will have to tend to $0$ as $n \to \infty$.
Since $p_{00}+p_{01}+p_{10}+p_{11}=1$, this of course influences 
the distributional parameters $p_{00}, q_1$ and $q_2$ in that it also makes them dependent on $n$.
The continuous intensity functions $\lambda^\star$ and $\acute{\lambda}^\star$ thus have the drawback that, through their dependence on the 
parameters $p_{00}$, $q_1$ and $q_2$, they are also dependent on $n$.
We thus try to find other suitable continuous intensity functions that no longer vary with the sample size.  

For simplicity, we make the assumption that $p_{10}$ and $p_{01}$ vary at the same rate as $p_{11}=p_{11n}$; more precisely, assume
$p_{10} = p_{10n}=\gamma p_{11n}$ and $p_{01} = p_{01n}=\delta p_{11n}$, where $\gamma$ and $\delta$ are strictly positive real numbers, bounded such
that $p_{10}$ and $p_{01}$ are smaller than 1. We assume that $p_{11n}$ tends to $0$ as $n \to \infty$ at a rate that will be determined later, and 
express the distributional parameters as functions of it:
\begin{equation}\label{d: conditions}
\begin{split}
q_{1n} &=  1-(1+\gamma)p_{11n},\\
q_{2n} &= 1-(1+\delta)p_{11n},\\
p_{00n} &= 1-(1+\gamma+\delta)p_{11n}.
\end{split}
\end{equation}
Plugging into (\ref{d: ls}) and using the relation $\log(1-z) \sim -z$ for $|z| < 1$ and $z \to 0$, we find that $\lambda^\star(s,t)$ and $\acute{\lambda}^\star(s)$ are, 
for $p_{11n} \to 0$ as $n \to \infty$,
asymptotically equal to 
\begin{align}\label{d: lsnew}
\begin{split}
 \lsnew(s,t)&:= \left \{
 \begin{array}{lll} 
 \frac{\gamma(1+\delta)}{(\gade)^2}\,\mathrm{e}^{-\frac{\gamma}{\gade}\,s}\mathrm{e}^{-\frac{1 + \delta}{\gade}\,t} &\textnormal{ for }& s < t ,  \\
 \frac{\delta (1 + \gamma)}{(\gade)^2}\,\mathrm{e}^{-\frac{1+\gamma}{\gade}\,s} \mathrm{e}^{-\frac{\delta}{\gade}\,t} &\textnormal{ for }& s > t,  
 \end{array} \right.\\
 \text{and } \quad \acute{\lambda}^\star_{\gamma, \delta}(s)&:=\frac{1}{\gade}\, \mathrm{e}^{-s} \quad \textnormal{ for }  s = t,
\end{split}
\end{align}
respectively, for all $(s,t) \in [-\log n, \infty)^2$. At first glance $\lsnew$ and $\acute{\lambda}^\star_{\gamma, \delta}$ seem to be valid choices 
for continuous intensity functions independent of $n$. 
We will investigate in Section \ref{s: MO_Geo_nicer} whether a Poisson process with mean measure $\bl^\star$ on $A^\star$ may indeed be approximated by a 
Poisson process with mean measure 
\begin{equation}\label{d: blsnew}
\blsnew(B^\star) := \int \int_{A^\star \cap B^\star} \lsnew(s,t)\De s\De t + \int_{A^\star \cap B^\star \cap \{(s,t):\, s=t\}}\acute{\lambda}^\star_{\gamma, \delta}(s)\De s,
\end{equation}
for all $B^\star \in \mathcal{B}([-\log n, \infty)^2)$. 
To do the corresponding error calculations for a fixed sample size $n$ we first need to examine in further detail the differences between 
the exponent terms in $\lambda^\star(s,t)$ and $\lsnew(s,t)$:
\begin{lemma}\label{t: Lemma_cond} 
For each integer $n \ge 1$, let $p_{11n} \in (0,1)$ and let $q_{1n}$, $q_{2n}$, $p_{00n}$ $\in (0,1)$  be defined by (\ref{d: conditions}). Then,
\begin{align*}
(i) \qquad 0 \le \frac{1+\delta}{\gade} - \frac{\log q_{2n}}{\log p_{00n}} &\le \frac{\gamma p_{11n}}{1-(\gade)p_{11n}}\,,\\
(ii) \qquad 0 \le \frac{1+\gamma}{\gade} -\frac{\log q_{1n}}{\log p_{00n}} &\le \frac{\delta p_{11n}}{1-(\gade)p_{11n}}\,.
\end{align*}
Moreover,
\begin{align*}
(iii) \qquad 0 \le \log\left(\frac{q_{2n}}{p_{00n}}\right) \le \frac{\gamma p_{11n}}{1-(\gade)p_{11n}}\,, \\
(iv) \qquad 0 \le \log\left(\frac{q_{1n}}{p_{00n}}\right) \le \frac{\delta p_{11n}}{1-(\gade)p_{11n}}\,,
\end{align*}
and 
\[
(v)\qquad \log\left(\frac{1}{p_{00n}}\right)\log\left(\frac{p_{00n}}{q_{1n}q_{2n}}\right) \le \frac{(\gade)p_{11n}}{\{1-(\gade)p_{11n}\}^2}\,.
\]
\end{lemma}
\begin{proof}
(i) For ease of notation we omit the subscript $n$. Since, for all $|z| < 1$, $-\log(1-z)/z$ is increasing and $-(1-z)\log(1-z)/z$ is decreasing, we obtain the
following lower and upper bound, respectively, for $-(\log q_2)/(\log p_{00})$, where $q_2 < p_{00}$:
\[
- \frac{1+\delta}{\gade} \le - \frac{\log q_2}{\log p_{00}} \le - \frac{(1+\delta)\cdot [1-(\gade)p_{11}]}{[1-(1+\delta)p_{11}]\cdot(\gade)}.
\]
Therefore,
\begin{align*}
0 \le \frac{1+\delta}{\gade} - \frac{\log q_2}{\log p_{00}} 
&\le \frac{1+\delta}{\gade} \left\{ 1- \frac{1-(\gade)p_{11}}{1-(1+\delta)p_{11}}\right\}\\
&= \frac{(1+\delta)\gamma p_{11}}{(\gade) [1- (1+\delta)p_{11}]}
\le \frac{\gamma p_{11}}{1-(\gade)p_{11}}\,.
\end{align*}
(iii) Moreover, since $q_2 =p_{00}+p_{10}$, we have $\log(q_2/ p_{00}) \ge 0$. Using $\log(1+ z) \le z$ for positive $z$, we obtain
\[
\log(q_2/p_{00}) = \log \left( \frac{p_{00}+p_{10}}{p_{00}}\right) \le \frac{p_{10}}{p_{00}} = \frac{\gamma p_{11}}{1-(\gade)p_{11}}\,.
\]
(ii) and (iv) can be shown analogously to (i) and (iii), respectively.\\
(v) We have
\begin{eqnarray*}
&&\log\left(\frac{1}{p_{00n}}\right)\log\left(\frac{p_{00n}}{q_{1n}q_{2n}}\right) 
= (-\log p_{00}) \left\{ -\log(p_{00} + p_{01}) - \log(p_{00}+p_{10}) + \log p_{00}\right\}\\
&& \le (-\log p_{00}) \left\{ -\log p_{00} - \log p_{00} + \log p_{00} \right\} = (-\log p_{00})^2 \le \frac{(1-p_{00})^2}{p_{00}^2} \le \frac{1-p_{00}}{p_{00}^2} \\
&&= \frac{(\gade)p_{11}}{\left\{ 1- (\gade)p_{11}\right\}^2}\,. 
\end{eqnarray*}
\end{proof}
\noindent We will use Lemma \ref{t: Lemma_cond} to determine error estimates in Sections \ref{s: MO_Geo_d2} and \ref{s: MO_Geo_nicer}. 
\begin{remark}
We suppose here that $\gamma$ and $\delta$ do not vary with $n$. However, the asymptotic equivalence of (\ref{d: ls}) and (\ref{d: lsnew}),
and later results (i.e. Propositions \ref{t: MO_Geo_disc_to_cont_appl} and \ref{t: MO_Geo_cont_to_nicer_appl}, as well as Corollary \ref{t: MO_Geo_final_bound_d2}) 
also hold for the case $\gamma=\gamma_n$ and $\delta=\delta_n$. These results are thus actually stronger than we make them out to be.  
\end{remark}
%-----------------------------------------------------------------------------------------------------------------------------------------------------------------------
%-----------------------------------------------------------------------------------------------------------------------------------------------------------------------
\subsubsection{Approximation in $d_2$ by a Poisson process with continuous intensity}\label{s: MO_Geo_d2} 
We now determine the error of the approximation of the Poisson process with mean measure $\bpis$, living on lattice points $(k^\star, l^\star) \in A^\star \cap E^\star$, 
and the Poisson process with continuous mean measure $\bls$, living on $A^\star \cap [-\log n, \infty)^2$. 
Note that any not too small set $\As \in \mathcal{B}([-\log n, \infty)^2)$ contains subsets that are unions of coordinate rectangles $\Rstar$, i.e. of the form
\begin{equation}\label{d: TildeAstar}
 \bigcup_{(k^\star,l^\star) \in M^\star}  R^\star_{k^\star, l^\star} \subseteq A^\star, 
\end{equation}
where $M^\star$ is a countable subset of $E^\star$. 
Let $\tilde{A}^\star$ denote the biggest subset of $A^\star$ of the form (\ref{d: TildeAstar}).
In order to prove Theorem \ref{t: approx_disc_cont}, 
we rely on arguments used by \cite{Xia:1995} used to show {Proposition \ref{t: d2_two_PRM}}.
%-%-%-%-%-%-%-%-%- Extreme0330 version
%distinguish between the errors on
%$\tilde{\As}$ and $\As \setminus \tilde{\As}$. Even though $\bpis(\As)=\bls(\As)$ is not necessarily satisfied, Proposition \ref{t: MO_Geo_same_meas_rectangles} ensures that at least $\bpis(\tilde{\As})= \bls(\tilde{\As})$. We may therefore use Lemma \ref{t: helping_bound_s2} to bound the error on $\tilde{\As}$ by way of the $d_1$-distance between $\bpis$ and $\bls$ on $\tilde{\As}$. The size of the $d_1$-distance depends on the choice of the $d_0$-distance. We choose the Euclidean distance bounded by $1$. For the remaining error, we rely on the ``small" size of $\As \setminus \tilde{\As}$ and use Lemma \ref{t: Delta_1_upgamma_d2} for an upper bound on $\Delta_1 \upgamma$, where $\upgamma$ is the solution to an appropriate Stein equation. 
%-%-%-%-%-%-%-%-%-
\begin{figure}
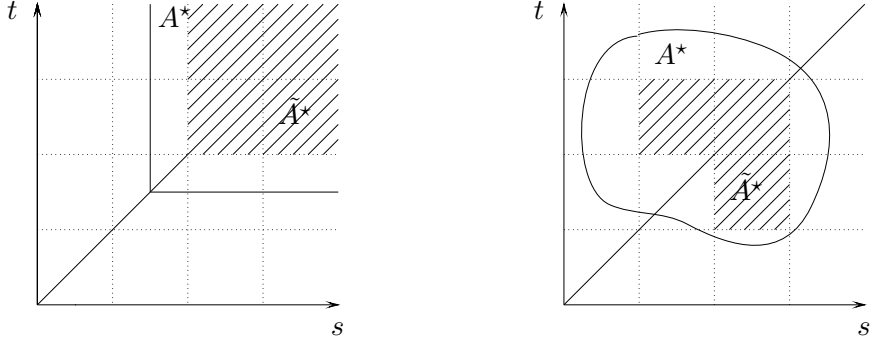

\centering
\begin{pgfpicture}{26.00mm}{65.39mm}{142.90mm}{120.00mm}
\pgfsetxvec{\pgfpoint{1.00mm}{0mm}}
\pgfsetyvec{\pgfpoint{0mm}{1.00mm}}
\color[rgb]{0,0,0}\pgfsetlinewidth{0.10mm}\pgfsetdash{}{0mm}
\pgfmoveto{\pgfxy(30.00,110.00)}\pgflineto{\pgfxy(30.00,70.00)}\pgfstroke
\pgfmoveto{\pgfxy(30.00,110.00)}\pgflineto{\pgfxy(29.70,108.80)}\pgflineto{\pgfxy(30.00,109.40)}\pgflineto{\pgfxy(30.30,108.80)}\pgflineto{\pgfxy(30.00,110.00)}\pgfclosepath\pgffill
\pgfmoveto{\pgfxy(30.00,110.00)}\pgflineto{\pgfxy(29.70,108.80)}\pgflineto{\pgfxy(30.00,109.40)}\pgflineto{\pgfxy(30.30,108.80)}\pgflineto{\pgfxy(30.00,110.00)}\pgfclosepath\pgfstroke
\pgfsetdash{{0.10mm}{0.50mm}}{0mm}\pgfmoveto{\pgfxy(40.00,110.00)}\pgflineto{\pgfxy(40.00,70.00)}\pgfstroke
\pgfmoveto{\pgfxy(50.00,110.00)}\pgflineto{\pgfxy(50.00,70.00)}\pgfstroke
\pgfmoveto{\pgfxy(60.00,110.00)}\pgflineto{\pgfxy(60.00,70.00)}\pgfstroke
\pgfmoveto{\pgfxy(30.00,80.00)}\pgflineto{\pgfxy(70.00,80.00)}\pgfstroke
\pgfmoveto{\pgfxy(30.00,90.00)}\pgflineto{\pgfxy(70.00,90.00)}\pgfstroke
\pgfmoveto{\pgfxy(30.00,100.00)}\pgflineto{\pgfxy(70.00,100.00)}\pgfstroke
\pgfsetdash{}{0mm}\pgfmoveto{\pgfxy(30.00,70.00)}\pgflineto{\pgfxy(70.00,110.00)}\pgfstroke
\pgfsetlinewidth{0.30mm}\pgfmoveto{\pgfxy(40.00,120.00)}\pgflineto{\pgfxy(40.00,120.00)}\pgfstroke
\pgfsetlinewidth{0.10mm}\pgfmoveto{\pgfxy(35.00,70.00)}\pgflineto{\pgfxy(35.00,70.00)}\pgfstroke
\pgfmoveto{\pgfxy(30.00,70.00)}\pgflineto{\pgfxy(70.00,70.00)}\pgfstroke
\pgfmoveto{\pgfxy(70.00,70.00)}\pgflineto{\pgfxy(68.80,70.30)}\pgflineto{\pgfxy(69.40,70.00)}\pgflineto{\pgfxy(68.80,69.70)}\pgflineto{\pgfxy(70.00,70.00)}\pgfclosepath\pgffill
\pgfmoveto{\pgfxy(70.00,70.00)}\pgflineto{\pgfxy(68.80,70.30)}\pgflineto{\pgfxy(69.40,70.00)}\pgflineto{\pgfxy(68.80,69.70)}\pgflineto{\pgfxy(70.00,70.00)}\pgfclosepath\pgfstroke
\pgfmoveto{\pgfxy(45.00,110.00)}\pgflineto{\pgfxy(45.00,85.00)}\pgfstroke
\pgfmoveto{\pgfxy(45.00,85.00)}\pgflineto{\pgfxy(70.00,85.00)}\pgfstroke
\pgfsetdash{{2.00mm}{1.00mm}}{0mm}\pgfmoveto{\pgfxy(52.00,87.00)}\pgfstroke
\pgfputat{\pgfxy(69.00,66.00)}{\pgfbox[bottom,left]{$s$}}
\pgfputat{\pgfxy(26.00,108.00)}{\pgfbox[bottom,left]{$t$}}
\pgfsetdash{}{0mm}\pgfmoveto{\pgfxy(50.00,92.00)}\pgflineto{\pgfxy(68.00,110.00)}\pgfstroke
\pgfmoveto{\pgfxy(50.00,94.00)}\pgflineto{\pgfxy(66.00,110.00)}\pgfstroke
\pgfmoveto{\pgfxy(50.00,96.00)}\pgflineto{\pgfxy(64.00,110.00)}\pgfstroke
\pgfmoveto{\pgfxy(50.00,98.00)}\pgflineto{\pgfxy(62.00,110.00)}\pgfstroke
\pgfmoveto{\pgfxy(50.00,100.00)}\pgflineto{\pgfxy(60.00,110.00)}\pgfstroke
\pgfmoveto{\pgfxy(50.00,102.00)}\pgflineto{\pgfxy(58.00,110.00)}\pgfstroke
\pgfmoveto{\pgfxy(50.00,104.00)}\pgflineto{\pgfxy(56.00,110.00)}\pgfstroke
\pgfmoveto{\pgfxy(50.00,106.00)}\pgflineto{\pgfxy(54.00,110.00)}\pgfstroke
\pgfmoveto{\pgfxy(50.00,108.00)}\pgflineto{\pgfxy(52.00,110.00)}\pgfstroke
\pgfmoveto{\pgfxy(52.00,90.00)}\pgflineto{\pgfxy(70.00,108.00)}\pgfstroke
\pgfmoveto{\pgfxy(54.00,90.00)}\pgflineto{\pgfxy(70.00,106.00)}\pgfstroke
\pgfmoveto{\pgfxy(56.00,90.00)}\pgflineto{\pgfxy(70.00,104.00)}\pgfstroke
\pgfmoveto{\pgfxy(58.00,90.00)}\pgflineto{\pgfxy(70.00,102.00)}\pgfstroke
\pgfmoveto{\pgfxy(60.00,90.00)}\pgflineto{\pgfxy(70.00,100.00)}\pgfstroke
\pgfmoveto{\pgfxy(62.00,90.00)}\pgflineto{\pgfxy(70.00,98.00)}\pgfstroke
\pgfmoveto{\pgfxy(64.00,90.00)}\pgflineto{\pgfxy(70.00,96.00)}\pgfstroke
\pgfmoveto{\pgfxy(66.00,90.00)}\pgflineto{\pgfxy(70.00,94.00)}\pgfstroke
\pgfmoveto{\pgfxy(68.00,90.00)}\pgflineto{\pgfxy(70.00,92.00)}\pgfstroke
\pgfputat{\pgfxy(62.00,94.00)}{\pgfbox[bottom,left]{$\tilde{A^\star}$}}
\pgfmoveto{\pgfxy(30.00,110.00)}\pgflineto{\pgfxy(30.00,70.00)}\pgfstroke
\pgfmoveto{\pgfxy(30.00,110.00)}\pgflineto{\pgfxy(29.70,108.80)}\pgflineto{\pgfxy(30.00,109.40)}\pgflineto{\pgfxy(30.30,108.80)}\pgflineto{\pgfxy(30.00,110.00)}\pgfclosepath\pgffill
\pgfmoveto{\pgfxy(30.00,110.00)}\pgflineto{\pgfxy(29.70,108.80)}\pgflineto{\pgfxy(30.00,109.40)}\pgflineto{\pgfxy(30.30,108.80)}\pgflineto{\pgfxy(30.00,110.00)}\pgfclosepath\pgfstroke
\pgfmoveto{\pgfxy(30.00,110.00)}\pgflineto{\pgfxy(30.00,70.00)}\pgfstroke
\pgfmoveto{\pgfxy(30.00,110.00)}\pgflineto{\pgfxy(29.70,108.80)}\pgflineto{\pgfxy(30.00,109.40)}\pgflineto{\pgfxy(30.30,108.80)}\pgflineto{\pgfxy(30.00,110.00)}\pgfclosepath\pgffill
\pgfmoveto{\pgfxy(30.00,110.00)}\pgflineto{\pgfxy(29.70,108.80)}\pgflineto{\pgfxy(30.00,109.40)}\pgflineto{\pgfxy(30.30,108.80)}\pgflineto{\pgfxy(30.00,110.00)}\pgfclosepath\pgfstroke
\pgfmoveto{\pgfxy(30.00,110.00)}\pgflineto{\pgfxy(30.00,70.00)}\pgfstroke
\pgfmoveto{\pgfxy(30.00,110.00)}\pgflineto{\pgfxy(29.70,108.80)}\pgflineto{\pgfxy(30.00,109.40)}\pgflineto{\pgfxy(30.30,108.80)}\pgflineto{\pgfxy(30.00,110.00)}\pgfclosepath\pgffill
\pgfmoveto{\pgfxy(30.00,110.00)}\pgflineto{\pgfxy(29.70,108.80)}\pgflineto{\pgfxy(30.00,109.40)}\pgflineto{\pgfxy(30.30,108.80)}\pgflineto{\pgfxy(30.00,110.00)}\pgfclosepath\pgfstroke
\pgfmoveto{\pgfxy(100.00,110.00)}\pgflineto{\pgfxy(100.00,70.00)}\pgfstroke
\pgfmoveto{\pgfxy(100.00,110.00)}\pgflineto{\pgfxy(99.70,108.80)}\pgflineto{\pgfxy(100.00,109.40)}\pgflineto{\pgfxy(100.30,108.80)}\pgflineto{\pgfxy(100.00,110.00)}\pgfclosepath\pgffill
\pgfmoveto{\pgfxy(100.00,110.00)}\pgflineto{\pgfxy(99.70,108.80)}\pgflineto{\pgfxy(100.00,109.40)}\pgflineto{\pgfxy(100.30,108.80)}\pgflineto{\pgfxy(100.00,110.00)}\pgfclosepath\pgfstroke
\pgfmoveto{\pgfxy(100.00,70.00)}\pgflineto{\pgfxy(140.00,70.00)}\pgfstroke
\pgfmoveto{\pgfxy(140.00,70.00)}\pgflineto{\pgfxy(138.80,70.30)}\pgflineto{\pgfxy(139.40,70.00)}\pgflineto{\pgfxy(138.80,69.70)}\pgflineto{\pgfxy(140.00,70.00)}\pgfclosepath\pgffill
\pgfmoveto{\pgfxy(140.00,70.00)}\pgflineto{\pgfxy(138.80,70.30)}\pgflineto{\pgfxy(139.40,70.00)}\pgflineto{\pgfxy(138.80,69.70)}\pgflineto{\pgfxy(140.00,70.00)}\pgfclosepath\pgfstroke
\pgfsetdash{{0.10mm}{0.50mm}}{0mm}\pgfmoveto{\pgfxy(110.00,110.00)}\pgflineto{\pgfxy(110.00,70.00)}\pgfstroke
\pgfmoveto{\pgfxy(120.00,70.00)}\pgflineto{\pgfxy(120.00,110.00)}\pgfstroke
\pgfmoveto{\pgfxy(130.00,70.00)}\pgflineto{\pgfxy(130.00,110.00)}\pgfstroke
\pgfmoveto{\pgfxy(100.00,100.00)}\pgflineto{\pgfxy(140.00,100.00)}\pgfstroke
\pgfmoveto{\pgfxy(100.00,90.00)}\pgflineto{\pgfxy(140.00,90.00)}\pgfstroke
\pgfmoveto{\pgfxy(100.00,80.00)}\pgflineto{\pgfxy(140.00,80.00)}\pgfstroke
\pgfputat{\pgfxy(139.00,66.00)}{\pgfbox[bottom,left]{$s$}}
\pgfputat{\pgfxy(96.00,108.00)}{\pgfbox[bottom,left]{$t$}}
\pgfsetdash{}{0mm}\pgfmoveto{\pgfxy(100.00,70.00)}\pgflineto{\pgfxy(140.00,110.00)}\pgfstroke
\pgfmoveto{\pgfxy(109.76,105.75)}\pgfcurveto{\pgfxy(100.90,105.28)}{\pgfxy(100.39,85.85)}{\pgfxy(106.06,83.33)}\pgfcurveto{\pgfxy(109.31,81.88)}{\pgfxy(113.06,82.51)}{\pgfxy(116.34,80.76)}\pgfcurveto{\pgfxy(122.48,77.48)}{\pgfxy(129.28,75.74)}{\pgfxy(132.80,82.92)}\pgfcurveto{\pgfxy(136.56,90.57)}{\pgfxy(137.14,99.80)}{\pgfxy(127.25,104.10)}\pgfcurveto{\pgfxy(122.14,106.32)}{\pgfxy(114.86,107.36)}{\pgfxy(109.86,105.96)}\pgfstroke
\pgfmoveto{\pgfxy(110.00,90.00)}\pgflineto{\pgfxy(120.00,100.00)}\pgfstroke
\pgfmoveto{\pgfxy(110.00,92.00)}\pgflineto{\pgfxy(118.00,100.00)}\pgfstroke
\pgfmoveto{\pgfxy(110.00,94.00)}\pgflineto{\pgfxy(116.00,100.00)}\pgfstroke
\pgfmoveto{\pgfxy(110.00,96.00)}\pgflineto{\pgfxy(114.00,100.00)}\pgfstroke
\pgfmoveto{\pgfxy(110.00,98.00)}\pgflineto{\pgfxy(112.00,100.00)}\pgfstroke
\pgfmoveto{\pgfxy(112.00,90.00)}\pgflineto{\pgfxy(120.00,98.00)}\pgfstroke
\pgfmoveto{\pgfxy(114.00,90.00)}\pgflineto{\pgfxy(124.00,100.00)}\pgfstroke
\pgfmoveto{\pgfxy(120.00,98.00)}\pgflineto{\pgfxy(122.00,100.00)}\pgfstroke
\pgfmoveto{\pgfxy(116.00,90.00)}\pgflineto{\pgfxy(126.00,100.00)}\pgfstroke
\pgfmoveto{\pgfxy(118.00,90.00)}\pgflineto{\pgfxy(128.00,100.00)}\pgfstroke
\pgfmoveto{\pgfxy(122.00,90.00)}\pgflineto{\pgfxy(130.00,98.00)}\pgfstroke
\pgfmoveto{\pgfxy(120.00,88.00)}\pgflineto{\pgfxy(122.00,90.00)}\pgfstroke
\pgfmoveto{\pgfxy(120.00,86.00)}\pgflineto{\pgfxy(130.00,96.00)}\pgfstroke
\pgfmoveto{\pgfxy(120.00,84.00)}\pgflineto{\pgfxy(130.00,94.00)}\pgfstroke
\pgfmoveto{\pgfxy(120.00,82.00)}\pgflineto{\pgfxy(130.00,92.00)}\pgfstroke
\pgfmoveto{\pgfxy(120.00,80.00)}\pgflineto{\pgfxy(130.00,90.00)}\pgfstroke
\pgfmoveto{\pgfxy(122.00,80.00)}\pgflineto{\pgfxy(130.00,88.00)}\pgfstroke
\pgfmoveto{\pgfxy(124.00,80.00)}\pgflineto{\pgfxy(130.00,86.00)}\pgfstroke
\pgfmoveto{\pgfxy(126.00,80.00)}\pgflineto{\pgfxy(130.00,84.00)}\pgfstroke
\pgfmoveto{\pgfxy(128.00,80.00)}\pgflineto{\pgfxy(130.00,82.00)}\pgfstroke
\pgfputat{\pgfxy(122.00,84.00)}{\pgfbox[bottom,left]{$\tilde{A^\star}$}}
\pgfputat{\pgfxy(112.00,102.00)}{\pgfbox[bottom,left]{$A^\star$}}
\pgfputat{\pgfxy(46.00,107.00)}{\pgfbox[bottom,left]{$A^\star$}}
\end{pgfpicture}%

\caption{Examples of sets $\tilde{A^\star}$.}
\label{f: TildeAstar}
\end{figure}
%-----------------------------------------------------------------------------------------------------------------------------------------------------------------------
%Moreover, 
More precisely, let $\upgamma$ be as defined in Proposition
\ref{t: upgamma_welldefined} for $h \in \mathcal{H}$. 
\begin{thm}\label{t: approx_disc_cont}
With the notations from Sections \ref{s: MO_Geo}-\ref{s: MO_Geo_cont_nicer_int_fct}, 
we obtain, for a set $A^\star \in \mathcal{B}([-\log n, \infty)^2)$,
\begin{equation}\label{t: approx_disc_cont_eq} 
d_2(\mathrm{PRM}(\bpis), \mathrm{PRM}(\bl^\star))
\le \sqrt{2} \log(1/ p_{00})+ \min\left\{1 ,\, \frac{1.65}{\sqrt{\bl^\star(A^\star)}} \right\} \bl^\star(A^\star \setminus \tilde{A}^\star),
\end{equation}
where $\tilde{A}^\star$ denotes the biggest subset of $A^\star$ of the form \eqref{d: TildeAstar} with $M^\star$ the biggest subset of 
$E^\star$ such that $\tilde{A}^\star \subseteq \As$.\\
\end{thm}
\begin{proof}
Let $\Xi_{\bpis} \sim \mathrm{PRM}(\bpis)$ and $\Xi_{\bls} \sim \mathrm{PRM}(\bls)$. We recall the standard techniques also employed in~\citet{Xia:1995}. Suppose that $Z= \{Z_t, t\in \mathbb{R}_+\}$ is an immigration-death process on $A^\star$ with 
immigration intensity $\bls$, unit per-capita death rate, equilibrium distribution $\mathcal{L}(\Xi_{\bls})$, and generator $\mathcal{A}$.
Furthermore, let $\mathcal{H}$ denote the set of functions $h:\, M_p(A^\star) \to \mathbb{R}$ such that (\ref{d: s_2h}) is satisfied
and let $\upgamma:\, M_p(A^\star) \to \mathbb{R}$ be defined by $\upgamma(\xi)$ $= $
$-\int_0^\infty \{\mathbb{E}^\xi h(Z_t) - \mathrm{PRM}(\bls)\}\De t$, for any $\xi \in M_p(A^\star)$. By Proposition \ref{t: upgamma_welldefined}, $\upgamma$ is well-defined, 
and by (\ref{d: Stein_eq_d2}), $|\mathrm{PRM}(\bpis)(h)- \mathrm{PRM}(\bls)(h)|$ equals $|\mathbb{E}(\mathcal{A}\upgamma)(\Xi_{\bpis})|$. 
Proceeding as in the proof of %\cite[Theorem 3.6]{Barbour_et_al.:1992}
Theorem 1.2 in \cite{Xia:1995}, we find
\[
\mathbb{E}(\mathcal{A}\upgamma)(\Xi_{\bpis})= \mathbb{E} \int_{A^\star} \left[ \upgamma(\Xi_{\bpis} + \delta_{\mathbf{z}}) - \upgamma(\Xi_{\bpis})\right]
\left( \bls(d\bz) - \bpis(d\bz)\right), 
\]
and thus 
\begin{align}
\frac{|\mathbb{E}(\mathcal{A}\upgamma)(\Xi_{\bpis})|}{s_2(h)} &\le
\frac{1}{s_2(h)}\, \mathbb{E} \left| \int_{\tilde{A^\star}} \left[ \upgamma(\Xi_{\bpis} + \delta_{\bz}) - \upgamma(\Xi_{\bpis})\right]
\left( \bls(d\bz) - \bpis(d\bz)\right)\right| \nonumber\\
& \quad + \frac{1}{s_2(h)}\, \mathbb{E} \left| \int_{A^\star \smallsetminus \tilde{A^\star}} \left[ \upgamma(\Xi_{\bpis} + \delta_{\bz}) - \upgamma(\Xi_{\bpis})\right]
\left( \bls(d\bz) - \bpis(d\bz)\right)\right|. \label{p: Xia}
\end{align}
Again by the proof of Theorem 1.2 in \cite{Xia:1995}, the first summand is bounded by   
\[
\left( 1 - e^{-\bl(\tilde{A}^\star)} \right) d_1(\bpis, \bls)) |_{\tilde{A^\star}} 
+ \left|\bl^\star(\tilde{A}^\star) - \bpis(\tilde{A}^\star)\right| 
\min\left\{1 ,\, \frac{1.65}{\sqrt{\bl^\star(\tilde{A^\star})}} \right\}
\le d_1(\bpis, \bls)) |_{\tilde{A^\star}},
\]
where $d_1(.,.)|_{\tilde{A}^\star}$ denotes the $d_1$-distance on $\tilde{A}^\star$ (instead of on $A^\star$) and where we used $\bl^\star(\tilde{A}^\star) = \bpis(\tilde{A}^\star) \ge 0$ for the last inequality. The $d_1$-distance between $\bls$  and $\bpis$ on ${\tilde{A}}^\star$ 
is given by
\[
 d_1(\bpis,\bls)|_{\tilde{A}^\star} = \frac{1}{\bls(\tilde{A}^\star)} \,
\sup_{\kappa \in \mathcal{K}} \frac{\left| 
\int_{\tilde{A}^\star}\kappa d\bpis-\int_{\tilde{A}^\star} \kappa d\bls \right|}{s_1(\kappa)}\,.
\]
As $\tilde{A}^\star$ is a union of coordinate rectangles $\Rstar$, we may express 
$\int_{\tilde{A}^\star}\kappa d\bpis-\int_{\tilde{A}^\star} \kappa d\bls $ as
\begin{equation} \label{p: d1diff}
  \sum_{(k^\star,l^\star) \in \tilde{A}^\star} \left\{ \int_{\Rstar} \kappa(\bz)\bpis(d\bz) 
 -\int_{\Rstar}\kappa(\bz)\bls(d\bz) \right\}.
\end{equation}
Furthermore, again by Proposition \ref{t: MO_Geo_same_meas_rectangles},
\[
 \int_{\Rstar}\kappa(\bz)\bpis(d\bz)  = \kappa((k^\star,l^\star)) \bpis(\Rstar) = \kappa((k^\star,l^\star)) \bls(\Rstar).
\]
Hence, we find the following upper bound for (\ref{p: d1diff}):
\[
 \sum_{(k^\star,l^\star) \in \tilde{A}^\star}  \int_{\Rstar} \left|\kappa((k^\star, l^\star)) -\kappa(\bz)\right|\bls(d\bz)
 \le  s_1(\kappa)d_0((k^\star,l^\star),\bz)\bls(\tilde{A}^\star),
 \]
where we also used the definition of the Lipschitz constant $s_1(k)$. 
Since the biggest possible Euclidean distance between the lower left corner point $(k^\star, l^\star)$ and any other point $\bz$ in the 
rectangle $\Rstar$ is given by the length $\sqrt{2}\log(1/ p_{00})$ of its diagonal, we have 
\begin{equation} \label{p: diaglength}
d_1(\bpis,\bls)|_{\tilde{A}^\star} \le \sqrt{2}\log (1/p_{00}).
\end{equation}
For the second summand in (\ref{p: Xia}), we note that, for any $\xi \in M_p(A^\star)$,
%\begin{multline}\label{p: partitionA}
%\left|\int_{\As} [\upgamma(\xi+\delta_{\bz})-\upgamma(\xi)](\bls(d\bz)-\bpis(d\bz))\right|\\ 
%\le \left|\int_{\tilde{A}^\star} [\upgamma(\xi+\delta_{\bz})-\upgamma(\xi)]
%(\bls(d\bz)-\bpis(d\bz)) \right|+\left|\int_{\As\smallsetminus\tilde{A}^\star} 
%[\upgamma(\xi+\delta_{\bz})-\upgamma(\xi)](\bls(d\bz)-\bpis(d\bz))\right|.
%\end{multline}
%The second summand may be bounded by 
\begin{align} 
  \left|\int_{\As\smallsetminus\tilde{A}^\star} 
[\upgamma(\xi+\delta_{\bz})-\upgamma(\xi)](\bls(d\bz)-\bpis(d\bz))\right|
 & \le  \int_{\As \smallsetminus \tilde{A}^\star} |\upgamma(\xi+\delta_{\bz})-\upgamma(\xi)|\cdot 
 |\bls (d\bz)-\bpis(d\bz)| \nonumber \\ 
 &\le \Delta_1\upgamma \int_{\As \smallsetminus \tilde{A}^\star}|\bls(d\bz)-\bpis(d\bz)| \le \Delta_1\upgamma \cdot \bls(\As \setminus \tilde{A}^\star) \label{p: summand2}
\end{align}
%Note that
%\[
% \int_{\As \smallsetminus \tilde{A}^\star}|\bls(d\bz)-\bpis(d\bz)| \le \bls(\As \setminus \tilde{A}^\star), 
%\]
and, completing the proof, that Lemma \ref{t: Delta_1_upgamma_d2} gives 
\begin{equation} \label{p: deltaone}
 \Delta_1\upgamma \le s_2(h)\min\left\{ 1 ,\, \frac{1.65}{\sqrt{\bls(\As)}} \right\}.
 \end{equation}
\end{proof}

%-----------------------------------------------------------------------------------------------------------------------------------------------------------------------
\noindent Theorem \ref{t: approx_disc_cont} gives sharp results only if the probability of simultaneous failure, $p_{00} = p_{00n}$ tends to $1$ as $n \to \infty$. This makes sense since $\log(1/p_{00})$, introduced as scaling factor of the original marginal geometric random variables, provides the side lengths of the rescaled lattice squares. The condition $p_{00n} \uparrow 1$  makes the side lengths of the coordinate squares tend to $0$ and thus causes the ``disappearance" of the lattice into the whole real subset $[-\log n, \infty)^2$. The same holds for the area $\As \setminus \tilde{\As}$, thereby also causing the disappearance of the second error term as $n \to \infty$.

For sets $\As$ that are unions of coordinate rectangles, there is no left-over area $\As \setminus \tilde{\As}$, and by consequence no second error term. 
%-----------------------------------------------------------------------------------------------------------------------------------------------------------------------
% {\color{red}
% \begin{cor}\label{t: approx_disc_cont_rect_set}
% Let $\As \in \mathcal{B}([-\log n, \infty)^2)$ be a union of coordinate rectangles, i.e. $A^\star = \cup_{(k^\star, l^\star) \in M^\star} R^\star_{k^\star,l^\star}$ 
% where $M^\star \subseteq E^\star$. Then, 
% \[
% d_2(\mathrm{PRM}(\bpis), \mathrm{PRM}(\bl^\star)) \le 2\sqrt{2} \log(1/ p_{00}) .
% \] \qed
% \end{cor}
% }
%-----------------------------------------------------------------------------------------------------------------------------------------------------------------------
%%%%%%%%%%%%%%%%%%%%%%%%%%%%%%%%%%%%
\begin{figure}
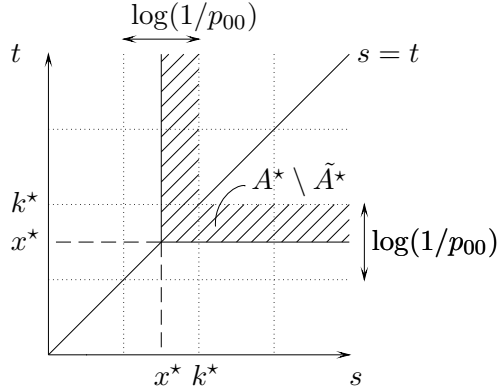

\centering
\begin{pgfpicture}{25.00mm}{65.39mm}{96.85mm}{120.00mm}
\pgfsetxvec{\pgfpoint{1.00mm}{0mm}}
\pgfsetyvec{\pgfpoint{0mm}{1.00mm}}
\color[rgb]{0,0,0}\pgfsetlinewidth{0.10mm}\pgfsetdash{}{0mm}
\pgfmoveto{\pgfxy(30.00,110.00)}\pgflineto{\pgfxy(30.00,70.00)}\pgfstroke
\pgfmoveto{\pgfxy(30.00,110.00)}\pgflineto{\pgfxy(29.70,108.80)}\pgflineto{\pgfxy(30.00,109.40)}\pgflineto{\pgfxy(30.30,108.80)}\pgflineto{\pgfxy(30.00,110.00)}\pgfclosepath\pgffill
\pgfmoveto{\pgfxy(30.00,110.00)}\pgflineto{\pgfxy(29.70,108.80)}\pgflineto{\pgfxy(30.00,109.40)}\pgflineto{\pgfxy(30.30,108.80)}\pgflineto{\pgfxy(30.00,110.00)}\pgfclosepath\pgfstroke
\pgfsetdash{{0.10mm}{0.50mm}}{0mm}\pgfmoveto{\pgfxy(40.00,110.00)}\pgflineto{\pgfxy(40.00,70.00)}\pgfstroke
\pgfmoveto{\pgfxy(50.00,110.00)}\pgflineto{\pgfxy(50.00,70.00)}\pgfstroke
\pgfmoveto{\pgfxy(60.00,110.00)}\pgflineto{\pgfxy(60.00,70.00)}\pgfstroke
\pgfmoveto{\pgfxy(30.00,80.00)}\pgflineto{\pgfxy(70.00,80.00)}\pgfstroke
\pgfmoveto{\pgfxy(30.00,90.00)}\pgflineto{\pgfxy(70.00,90.00)}\pgfstroke
\pgfmoveto{\pgfxy(30.00,100.00)}\pgflineto{\pgfxy(70.00,100.00)}\pgfstroke
\pgfsetdash{}{0mm}\pgfmoveto{\pgfxy(30.00,70.00)}\pgflineto{\pgfxy(70.00,110.00)}\pgfstroke
\pgfsetlinewidth{0.30mm}\pgfmoveto{\pgfxy(40.00,120.00)}\pgflineto{\pgfxy(40.00,120.00)}\pgfstroke
\pgfsetlinewidth{0.10mm}\pgfmoveto{\pgfxy(35.00,70.00)}\pgflineto{\pgfxy(35.00,70.00)}\pgfstroke
\pgfputat{\pgfxy(71.00,109.00)}{\pgfbox[bottom,left]{$s=t$}}
\pgfmoveto{\pgfxy(30.00,70.00)}\pgflineto{\pgfxy(70.00,70.00)}\pgfstroke
\pgfmoveto{\pgfxy(70.00,70.00)}\pgflineto{\pgfxy(68.80,70.30)}\pgflineto{\pgfxy(69.40,70.00)}\pgflineto{\pgfxy(68.80,69.70)}\pgflineto{\pgfxy(70.00,70.00)}\pgfclosepath\pgffill
\pgfmoveto{\pgfxy(70.00,70.00)}\pgflineto{\pgfxy(68.80,70.30)}\pgflineto{\pgfxy(69.40,70.00)}\pgflineto{\pgfxy(68.80,69.70)}\pgflineto{\pgfxy(70.00,70.00)}\pgfclosepath\pgfstroke
\pgfmoveto{\pgfxy(45.00,110.00)}\pgflineto{\pgfxy(45.00,85.00)}\pgfstroke
\pgfmoveto{\pgfxy(45.00,85.00)}\pgflineto{\pgfxy(70.00,85.00)}\pgfstroke
\pgfmoveto{\pgfxy(45.00,87.00)}\pgflineto{\pgfxy(50.00,92.00)}\pgfstroke
\pgfmoveto{\pgfxy(45.00,89.00)}\pgflineto{\pgfxy(50.00,94.00)}\pgfstroke
\pgfmoveto{\pgfxy(45.00,91.00)}\pgflineto{\pgfxy(50.00,96.00)}\pgfstroke
\pgfmoveto{\pgfxy(45.00,93.00)}\pgflineto{\pgfxy(50.00,98.00)}\pgfstroke
\pgfmoveto{\pgfxy(49.00,85.00)}\pgflineto{\pgfxy(54.00,90.00)}\pgfstroke
\pgfmoveto{\pgfxy(47.00,85.00)}\pgflineto{\pgfxy(52.00,90.00)}\pgfstroke
\pgfmoveto{\pgfxy(51.00,85.00)}\pgflineto{\pgfxy(56.00,90.00)}\pgfstroke
\pgfmoveto{\pgfxy(53.00,85.00)}\pgflineto{\pgfxy(58.00,90.00)}\pgfstroke
\pgfmoveto{\pgfxy(55.00,85.00)}\pgflineto{\pgfxy(60.00,90.00)}\pgfstroke
\pgfmoveto{\pgfxy(57.00,85.00)}\pgflineto{\pgfxy(62.00,90.00)}\pgfstroke
\pgfmoveto{\pgfxy(59.00,85.00)}\pgflineto{\pgfxy(64.00,90.00)}\pgfstroke
\pgfmoveto{\pgfxy(61.00,85.00)}\pgflineto{\pgfxy(66.00,90.00)}\pgfstroke
\pgfmoveto{\pgfxy(63.00,85.00)}\pgflineto{\pgfxy(68.00,90.00)}\pgfstroke
\pgfmoveto{\pgfxy(65.00,85.00)}\pgflineto{\pgfxy(70.00,90.00)}\pgfstroke
\pgfmoveto{\pgfxy(67.00,85.00)}\pgflineto{\pgfxy(70.00,88.00)}\pgfstroke
\pgfmoveto{\pgfxy(69.00,85.00)}\pgflineto{\pgfxy(70.00,86.00)}\pgfstroke
\pgfmoveto{\pgfxy(45.00,95.00)}\pgflineto{\pgfxy(50.00,100.00)}\pgfstroke
\pgfmoveto{\pgfxy(45.00,97.00)}\pgflineto{\pgfxy(50.00,102.00)}\pgfstroke
\pgfmoveto{\pgfxy(45.00,99.00)}\pgflineto{\pgfxy(50.00,104.00)}\pgfstroke
\pgfmoveto{\pgfxy(45.00,101.00)}\pgflineto{\pgfxy(50.00,106.00)}\pgfstroke
\pgfmoveto{\pgfxy(45.00,103.00)}\pgflineto{\pgfxy(50.00,108.00)}\pgfstroke
\pgfmoveto{\pgfxy(45.00,105.00)}\pgflineto{\pgfxy(50.00,110.00)}\pgfstroke
\pgfmoveto{\pgfxy(45.00,107.00)}\pgflineto{\pgfxy(48.00,110.00)}\pgfstroke
\pgfmoveto{\pgfxy(45.00,109.00)}\pgflineto{\pgfxy(46.00,110.00)}\pgfstroke
\pgfmoveto{\pgfxy(72.00,90.00)}\pgflineto{\pgfxy(72.00,80.00)}\pgfstroke
\pgfmoveto{\pgfxy(72.00,90.00)}\pgflineto{\pgfxy(71.70,88.80)}\pgflineto{\pgfxy(72.00,89.40)}\pgflineto{\pgfxy(72.30,88.80)}\pgflineto{\pgfxy(72.00,90.00)}\pgfclosepath\pgffill
\pgfmoveto{\pgfxy(72.00,90.00)}\pgflineto{\pgfxy(71.70,88.80)}\pgflineto{\pgfxy(72.00,89.40)}\pgflineto{\pgfxy(72.30,88.80)}\pgflineto{\pgfxy(72.00,90.00)}\pgfclosepath\pgfstroke
\pgfmoveto{\pgfxy(72.00,80.00)}\pgflineto{\pgfxy(72.30,81.20)}\pgflineto{\pgfxy(72.00,80.60)}\pgflineto{\pgfxy(71.70,81.20)}\pgflineto{\pgfxy(72.00,80.00)}\pgfclosepath\pgffill
\pgfmoveto{\pgfxy(72.00,80.00)}\pgflineto{\pgfxy(72.30,81.20)}\pgflineto{\pgfxy(72.00,80.60)}\pgflineto{\pgfxy(71.70,81.20)}\pgflineto{\pgfxy(72.00,80.00)}\pgfclosepath\pgfstroke
\pgfmoveto{\pgfxy(40.00,112.00)}\pgflineto{\pgfxy(50.00,112.00)}\pgfstroke
\pgfmoveto{\pgfxy(40.00,112.00)}\pgflineto{\pgfxy(41.20,111.70)}\pgflineto{\pgfxy(40.60,112.00)}\pgflineto{\pgfxy(41.20,112.30)}\pgflineto{\pgfxy(40.00,112.00)}\pgfclosepath\pgffill
\pgfmoveto{\pgfxy(40.00,112.00)}\pgflineto{\pgfxy(41.20,111.70)}\pgflineto{\pgfxy(40.60,112.00)}\pgflineto{\pgfxy(41.20,112.30)}\pgflineto{\pgfxy(40.00,112.00)}\pgfclosepath\pgfstroke
\pgfmoveto{\pgfxy(50.00,112.00)}\pgflineto{\pgfxy(48.80,112.30)}\pgflineto{\pgfxy(49.40,112.00)}\pgflineto{\pgfxy(48.80,111.70)}\pgflineto{\pgfxy(50.00,112.00)}\pgfclosepath\pgffill
\pgfmoveto{\pgfxy(50.00,112.00)}\pgflineto{\pgfxy(48.80,112.30)}\pgflineto{\pgfxy(49.40,112.00)}\pgflineto{\pgfxy(48.80,111.70)}\pgflineto{\pgfxy(50.00,112.00)}\pgfclosepath\pgfstroke
\pgfputat{\pgfxy(73.00,84.00)}{\pgfbox[bottom,left]{$\log(1/p_{00})$}}
\pgfputat{\pgfxy(73.00,84.00)}{\pgfbox[bottom,left]{$\log(1/p_{00})$}}
\pgfputat{\pgfxy(41.00,114.00)}{\pgfbox[bottom,left]{$\log (1/p_{00})$}}
\pgfputat{\pgfxy(70.00,66.00)}{\pgfbox[bottom,left]{$s$}}
\pgfputat{\pgfxy(25.00,109.00)}{\pgfbox[bottom,left]{$t$}}
\pgfsetdash{{2.00mm}{1.00mm}}{0mm}\pgfmoveto{\pgfxy(45.00,85.00)}\pgflineto{\pgfxy(45.00,70.00)}\pgfstroke
\pgfmoveto{\pgfxy(45.00,85.00)}\pgflineto{\pgfxy(30.00,85.00)}\pgfstroke
\pgfputat{\pgfxy(44.00,66.00)}{\pgfbox[bottom,left]{$x^\star$}}
\pgfputat{\pgfxy(49.00,66.00)}{\pgfbox[bottom,left]{$k^\star$}}
\pgfputat{\pgfxy(25.00,84.00)}{\pgfbox[bottom,left]{$x^\star$}}
\pgfputat{\pgfxy(25.00,89.00)}{\pgfbox[bottom,left]{$k^\star$}}
\pgfputat{\pgfxy(57.00,92.00)}{\pgfbox[bottom,left]{$A^\star \setminus \tilde{A^\star}$}}
\pgfmoveto{\pgfxy(52.00,87.00)}\pgfstroke
\pgfsetdash{}{0mm}\pgfmoveto{\pgfxy(52.21,87.07)}\pgfcurveto{\pgfxy(51.99,90.46)}{\pgfxy(54.72,92.68)}{\pgfxy(56.12,92.52)}\pgfcurveto{\pgfxy(56.12,92.50)}{\pgfxy(56.12,92.45)}{\pgfxy(56.12,92.42)}\pgfstroke
\end{pgfpicture}%

\caption{The set $\As \setminus \tilde{\As}$.}
\label{f: proof_cor_d2_discr_to_cont}
\end{figure}
%%%%%%%%%%%%%%%%%%%%%%%%%%%%%%%%%%%%
%-----------------------------------------------------------------------------------------------------------------------------------------------------------------------
We now apply Theorem \ref{t: approx_disc_cont} to the case where $\As =\As_n = [u_n^\star, \infty)^2$ and express the error estimate in terms of the threshold $u_n^\star$ and the probability of simultaneous success $p_{11n}$. 
To achieve this we assume that the distributional parameters $p_{00}, q_{1}$ and $q_2$ are defined as in Section \ref{s: MO_Geo_cont_nicer_int_fct}.
%-----------------------------------------------------------------------------------------------------------------------------------------------------------------------
\begin{prop}\label{t: MO_Geo_disc_to_cont_appl}
Let $p_{11n} \in (0,1)$ and assume that $q_{1n}$, $q_{2n}$ and $p_{00n}$ satisfy (\ref{d: conditions}). For any choice of $u_n^\star \ge -\log n$, 
define $A^\star = [u_n^\star,\infty)^2$.  
With the notations from Theorem \ref{t: approx_disc_cont},
\[
d_2(\mathrm{PRM}(\bpis), \mathrm{PRM}(\bl^\star))
\le \frac{(\gade)p_{11n}}{[1-(\gade)p_{11n}]^2}\left\{\sqrt{2} + 3\min\left\{\mathrm{e}^{-u_n^\star} ,\, 1.65 \mathrm{e}^{-u_n^\star/2} \right\}\right\}.
\]
\end{prop}
\begin{proof}
For ease of notation we omit the subscript $n$. We apply result (\ref{t: approx_disc_cont_eq}) from Theorem \ref{t: approx_disc_cont} to the special case 
$A^\star = [u^\star, \infty)^2$. Due to (\ref{d: conditions}) and $-\log(1-z) \le z/(1-z)$ for $|z| < 1$, we may bound the first of the two error terms in 
(\ref{t: approx_disc_cont_eq}) as follows:
\begin{equation}\label{p: MO_Geo_bubulala}
\sqrt{2}\log(1/p_{00}) \le \frac{\sqrt{2}(\gade)p_{11}}{1-(\gade)p_{11}} \le \frac{\sqrt{2}(\gade)p_{11}}{[1-(\gade)p_{11}]^2}\,.
\end{equation}
Direct computation yields $\bls(\As)= \mathrm{e}^{-u^\star}$. As illustrated by Figure \ref{f: proof_cor_d2_discr_to_cont}, 
$\bls(\As \setminus \tilde{\As})$ may be bounded by
\begin{multline*}
\int_{u^\star}^{\infty} \log (1/p_{00}) \sup_{s \in [u^\star,k^\star]} \ls(s,t)\De t \\
+ \int_{u^\star}^{\infty} \log (1/p_{00}) \sup_{t \in [u^\star,k^\star]} \ls(s,t)\De s
+ \int_{u^\star}^{k^\star}\sqrt{2} \log(1/p_{00}) \sup_{s \in [u^\star,k^\star]} \acute{\lambda}^\star(s) \De s.
\end{multline*}
Note that 
\begin{align*}
\sup_{s \in [u^\star,k^\star]} \exp{\left\{-\frac{\log(p_{00}/q_2)}{\log p_{00}}\,s\right\}} &\le \exp{\left\{-\frac{\log(p_{00}/q_2)}{\log p_{00}}\,u^\star\right\}},\\
\sup_{t \in [u^\star,k^\star]} \exp{\left\{-\frac{\log(p_{00}/q_1)}{\log p_{00}}\,t\right\}} &\le \exp{\left\{-\frac{\log(p_{00}/q_1)}{\log p_{00}}\,u^\star\right\}},
\end{align*}
and $\sup_{s \in [u^\star,k^\star]} \mathrm{e}^{-s} \le \mathrm{e}^{-u^\star}.$ Thus, by definition (\ref{d: ls}) of $\ls(s,t)$,
\[
\int_{u^\star}^{\infty} \log (1/p_{00}) \sup_{s \in [u^\star,k^\star]} \ls(s,t)\De t 
\le \frac{\log(p_{00}/q_2) \log q_2}{\log(1/p_{00})}\, \mathrm{e}^{-\frac{\log(p_{00}/q_2)}{\log p_{00}}\,u^\star} \int_{u^\star}^{\infty} \mathrm{e}^{-\frac{\log q_2}{\log p_{00}}\,t}\De t,
\]
which equals $\log(q_2/p_{00})\mathrm{e}^{-u^\star}$. Analogously,
\[
\int_{u^\star}^{\infty} \log (1/p_{00}) \sup_{t \in [u^\star,k^\star]} \ls(s,t)\De s \le \log(q_1/p_{00})\mathrm{e}^{-u^\star},
\]
whereas 
\[
\int_{u^\star}^{k^\star} \sqrt{2}\log(1/p_{00}) \sup_{s \in [u^\star,k^\star]} \acute{\lambda}^\star(s) \De s 
 \le  2\log^2(1/p_{00})  \,\frac{\log(p_{00}/q_1q_2)}{\log (1/p_{00})} \,\mathrm{e}^{-u^\star},
\]
since $k^\star - u^\star \le \sqrt{2}\log(1/p_{00})$. We obtain
\[
\bls(\As \setminus \tilde{\As}) 
 = \mathrm{e}^{-u^\star} \left\{ \log \left(\frac{q_2}{p_{00}}\right) + \log \left(\frac{q_1}{p_{00}}\right) +  2\log \left( \frac{1}{p_{00}}\right)
\log\left( \frac{p_{00}}{q_1 q_2}\right)\right\}.
\]
By Lemma \ref{t: Lemma_cond} (iii)-(v), the term in curly brackets may be bounded by
\[
\frac{(\gamma + \delta)p_{11}}{1-(\gade)p_{11}} + \frac{2(\gade)p_{11}}{\left\{ 1-(\gade)p_{11}\right\}^2} \le  \frac{3(\gade)p_{11}}{\left\{ 1-(\gade)p_{11}\right\}^2}\,.
\]
An upper bound for the second error term in (\ref{t: approx_disc_cont_eq}) is thus given by
\[
\min\left\{\mathrm{e}^{-u_n^\star} ,\, 1.65 \mathrm{e}^{-u_n^\star/2} \right\}\frac{3(\gade)p_{11n}}{[1-(\gade)p_{11n}]^2}\,. 
\]
By adding this to the bound in (\ref{p: MO_Geo_bubulala}) we obtain the result. 
\end{proof}
%-----------------------------------------------------------------------------------------------------------------------------------------------------------------------
\noindent The first of the error terms given by Proposition \ref{t: MO_Geo_disc_to_cont_appl}, i.e.
$\sqrt{2}(\gade)p_{11n}$ $/$ $[1-(\gade)p_{11n}]^2$,
%\[
% \frac{2\sqrt{2}(\gade)p_{11n}}{[1-(\gade)p_{11n}]^2},
%\]
is a bound on the error ${\sqrt{2}}\log(1/p_{00n})$ from Theorem \ref{t: approx_disc_cont}, where we used the assumption from Section \ref{s: MO_Geo_cont_nicer_int_fct} that $p_{00n} = 1-(\gade)p_{11n}$. This error term thus becomes small only if the probability of simultaneous success, $p_{11n}$, tends to $0$ as $n$ increases. The second error term, i.e. 
\[
\left\{\mathrm{e}^{-u_n^\star} ,\, 1.65 \mathrm{e}^{-u_n^\star/2} \right\}\frac{3(\gade)p_{11n}}{[1-(\gade)p_{11n}]^2}\,,
\]
is the bigger of the two, and determines the rate at which $p_{11n}$ must converge to $0$. The reason for that is that $p_{11n}$ must converge fast enough in order to offset the effect of the factor $\mathrm{e}^{-u_n^\star}$ which we will want to be increasing with increasing $n$, since $\mathrm{e}^{-u_n^\star} = \bls(\As_n)$ is 
the expected number of points in $\As_n$ of the approximating Poisson process, as well as more or less the expected number of threshold exceedances of the MPPE, for which we have $\mathrm{e}^{-u_n^\star}/ p_{00n} \le \bpis(\As_n) \le \mathrm{e}^{-u_n^\star}$. For instance, for a threshold $u_n^\star$ of size $- \log \log n$, the expected number of points in $\As$ of the two Poisson processes is $\log n$, the MPPE captures roughly the biggest $\log n$ points of its sample, and we need $p_{11n}= o(\log^{-1} n)$ for a sharp error bound. Suppose, for example, that $p_{11n} = n^{-1}$. 
Then, by (\ref{d: conditions}), the marginal probabilities of failure of $\mathbf{X}^\star_n$, $q_{1n}$ and $q_{2n}$, as well as the probability of simultaneous failure, $p_{00n}$, tend to $1$ very fast.

The mean measure $\bls$ is by definition dependent on the values of the distributional parameters. Since these need to vary with the sample size $n$ in order to obtain a small error for the approximation of $\mathrm{PRM}(\bpis)$ by $\mathrm{PRM}(\bls)$, it follows that $\bls = \bls_n$ (and of course also $\bpis = \bpis_n$). 
Though we have now achieved the goal of successfully approximating by a Poisson process with a continuous intensity, 
the conditions needed to accomplish this imply that we are not satisfied with our results yet, 
since we prefer to approximate by a Poisson process with continuous intensity that does not vary with $n$. 
As the next section will demonstrate, a suitable candidate is given by the Poisson process with intensity measure $\blsnew$ defined in (\ref{d: blsnew}). 
%-----------------------------------------------------------------------------------------------------------------------------------------------------------------------
\subsubsection{Approximation in $d_2$ and $d_{TV}$ by a Poisson process independent of $n$}\label{s: MO_Geo_nicer}
We determine an error estimate for the approximation of the Poisson process with intensity measure $\bls=\bls_n$ by the 
Poisson process with intensity measure $\blsnew$, defined in (\ref{d: blsnew}), that does not depend on the sample size $n$. 
Since both intensities are continuous, there is no special need to use the $d_2$-distance. We give the error in both the total 
variation and the $d_2$ distances.
For the error in total variation we may straightforwardly use %\cite[Theorem 3.6]{Barbour_et_al.:1992}
Proposition \ref{t: dTV_two_PRM} for the approximation of two Poisson processes. For the $d_2$-error, which will be smaller than the $d_{TV}$, we may additionally use Lemma 
\ref{t: Delta_1_upgamma_d2} for an upper bound on $\Delta_1 \upgamma$, where $\upgamma$ is the solution of an adequate Stein equation. This bound, containing the factor $\bls(\As)^{-1/2}$ (or $\blsnew({\As})^{-1/2}$), serves in reducing the $d_2$-error. 
%{\color{red} The error bounds given by Theorem \ref{t: approx_cont_nicer} will become small for large $n$ due to the 
%pointwise convergence of the intensity functions $\ls_n(s,t)$ and $\acute{\lambda}^\star_n(s)$ to the intensity functions 
%$\lsnew(s,t)$ and $\acute{\lambda}^\star_{\gamma,\delta}(s)$, respectively, as $n \to \infty$.{\em Is this remark necessary?}}
%-----------------------------------------------------------------------------------------------------------------------------------------------------------------------
\begin{thm}\label{t: approx_cont_nicer}
With the notations from Sections \ref{s: MO_Geo} and \ref{s: MO_Geo_dTV}-\ref{s: MO_Geo_cont_nicer_int_fct}, we obtain, for any set $A^\star \in \mathcal{B}([-\log n, \infty)^2)$,
\begin{align*}
(i) \qquad &d_{TV} \left(\mathrm{PRM}(\bl^\star), \mathrm{PRM}(\blsnew) \right) \le \int_{A^\star}| \bl^\star(d \mathbf{z})-\blsnew(d\mathbf{z})|,\\
(ii) \qquad&d_2 \left(\mathrm{PRM}(\bl^\star), \mathrm{PRM}(\blsnew) \right) \\
&\le \min\left\{ 1,\, 1.65 \min \left\{ \bl^\star(A^\star)^{-1/2} \, , \, \blsnew(A^\star)^{-1/2}\right\}\right\}
\int_{A^\star}| \bl^\star(d \mathbf{z})-\blsnew(d\mathbf{z})|.
\end{align*}
\end{thm}
%%%%%%%%%%%%%%%%%%%%%%%%%%%%% Maybe add third point to Thm saying that if A such that measures on it are equal then result from Prop 2.2.6 (d2 two PRMs). NO, doesn't work!
\begin{proof}
(i) By Proposition \ref{t: MO_Geo_same_meas_rectangles} (ii), $\bls$ is finite. Moreover, $\blsnew$ is finite since integration of 
$\lsnew$ and $\acute{\lambda}^\star_{\gamma,\delta}$ over $[u^\star, \infty)^2$ 
gives $\mathrm{e}^{-u^\star}$ 
% (delete the calculation below?)
% \begin{align*}
% &\int_{u^\star}^\infty \De t \int_{u^\star}^t ds\, \frac{\gamma(1+\delta)}{(\gade)^2}\,\mathrm{e}^{-\frac{\gamma}{\gade}\,s}\mathrm{e}^{-\frac{1 + \delta}{\gade}\,t}\\
% &+ \int_{u^\star}^\infty ds \int_{u^\star}^s \De t \, \frac{\delta (1 + \gamma)}{(\gade)^2}\,\mathrm{e}^{-\frac{1+\gamma}{\gade}\,s} \mathrm{e}^{-\frac{\delta}{\gade}\,t}
% + \int_{u^\star}^\infty ds \,\frac{1}{\gade}\, \mathrm{e}^{-s} \\
% =\,\,&\frac{\gamma}{\gade}\, \mathrm{e}^{-u^\star} + \frac{\delta}{\gade}\, \mathrm{e}^{-u^\star} + \frac{1}{\gade}\, \mathrm{e}^{-u^\star}= \mathrm{e}^{-u^\star},
% \end{align*}
% }
which equals $n$ for $u^\star = -\log n$. 
%\cite[Theorem 3.6]{Barbour_et_al.:1992} 
Proposition \ref{t: dTV_two_PRM} then immediately gives the result. \\
(ii) Using the same immigration-death process $Z$ and arguments as in the proof of Theorem \ref{t: approx_disc_cont}, we can show that for 
$\Xi^\star_{\gamma,\delta} \sim \mathrm{PRM}(\blsnew)$,
\[
 \mathbb{E}h(\Xi^\star_{\gamma,\delta})-\mathrm{PRM}(\bls)(h) = 
\mathbb{E} \left\{  \int_{\As} [\upgamma(\Xi^\star_{\gamma,\delta} + \delta_{\bz}) - \upgamma(\Xi^\star_{\gamma,\delta})] 
 (\bls(d\bz) - \blsnew(d\bz)) \right\}.
\]
Analogously to \eqref{p: summand2} and \eqref{p: deltaone}, the integrand may be bounded by 
\[
 \Delta_1\upgamma \int_{\As}|\bls(d\bz)-\blsnew(d\bz)| \le s_2(h)\min\left\{ 1 ,\, \frac{1.65}{\sqrt{\bls(\As)}}\right\} \int_{\As}|\bls(d\bz)-\blsnew(d\bz)|.
\]
Here, $1.65(\bls(\As))^{-1/2}$ may be replaced by $1.65(\blsnew(\As))^{-1/2}$ by going through the same arguments as before, but instead starting
with an immigration-death process over $\As$ with immigration intensity $\blsnew$, unit per-capita death rate, and equilibrium distribution 
$\mathrm{PRM}(\blsnew)$.
\end{proof}
%-----------------------------------------------------------------------------------------------------------------------------------------------------------------------
We now again assume that the distributional parameters $p_{00n}, q_{1n}$ and $q_{2n}$ satisfy (\ref{d: conditions}) 
%from Section \ref{s: MO_Geo_cont_nicer_int_fct} 
and apply Theorem \ref{t: approx_cont_nicer} to the case where $\As =\As_n = [u_n^\star, \infty)^2$.
We express the error bounds in terms of the threshold $u_n^{\star}$ and of the probability of simultaneous success $p_{11n}$.
%-----------------------------------------------------------------------------------------------------------------------------------------------------------------------
\begin{prop}\label{t: MO_Geo_cont_to_nicer_appl}
Let $p_{11n} \in (0,1)$ and assume that $q_{1n}$, $q_{2n}$ and $p_{00n}$ satisfy (\ref{d: conditions}). For any choice of $u_n^\star \ge -\log n$, 
define $A^\star = [u_n^\star,\infty)^2$.  
With the notations from Sections \ref{s: MO_Geo}-\ref{s: MO_Geo_cont_nicer_int_fct},
\begin{align*}
(i) \qquad &d_{TV} \left(\mathrm{PRM}(\bl^\star), \mathrm{PRM}(\blsnew) \right) \le \frac{4(\gade)^2 p_{11n}}{[1-(\gade)p_{11n}]^3}\, \mathrm{e}^{-u_n^\star}\,,\\
(ii) \qquad &d_2 \left(\mathrm{PRM}(\bl^\star), \mathrm{PRM}(\blsnew) \right)
\le \min\left\{ \mathrm{e}^{-u_n^\star} ,\, 1.65 \mathrm{e}^{-u_n^\star/2}\right\} 
\frac{4(\gade)^2 p_{11n}}{[1-(\gade)p_{11n}]^3}\,.
\end{align*}
\end{prop}
\begin{proof}
For ease of notation we again omit the subscript $n$. \\
(i) By Theorem \ref{t: approx_cont_nicer}, 
\begin{align*}
&d_{TV} \left(\mathrm{PRM}(\bl^\star), \mathrm{PRM}(\blsnew) \right) \le\,\, \int_{u^\star}^\infty  \int_{u^\star}^t  \left|\ls(s,t) -\lsnew(s,t)\right|\De s\De t\\
&+ \int_{u^\star}^\infty  \int_{u^\star}^s  \left|\ls(s,t) -\lsnew(s,t)\right|\De t\De s+ \int_{u^\star}^\infty \left| \acute{\lambda}^\star(s)-\acute{\lambda}^\star_{\gamma,\delta}(s)\right|\De s.\\
\end{align*}
Define 
\[
h := h(p_{11}) := \frac{1+\delta}{\gade} - \frac{\log q_{2}}{\log p_{00}}\, \quad \text{and} \quad
g:= g(p_{11}) := \frac{1+\gamma}{\gade} -\frac{\log q_{1}}{\log p_{00}}\,.
\]
We first consider the case $s=t$. Note that, with definitions (\ref{d: ls}) and (\ref{d: lsnew}),
\begin{align*}
\acute{\lambda}^\star(s)= \frac{\log(p_{00}/q_1q_2)}{\log(1/p_{00})}\, \mathrm{e}^{-s}
&= \left[\frac{\log q_1+\log q_2}{\log p_{00}} - \frac{2+ \gamma + \delta}{\gade} + \frac{1}{\gade}\right]\mathrm{e}^{-s}\\
&= \left[\frac{1}{\gade} - h(p_{11}) - g(p_{11})\right]\mathrm{e}^{-s},
\end{align*}
and that, since $h, g \ge 0$ by Lemma \ref{t: Lemma_cond} (i) and (ii), we thus have $\acute{\lambda}^\star(s) \le \acute{\lambda}^\star_{\gamma,\delta}(s)$.
Hence,
\[
\int_{u^\star}^\infty \left| \acute{\lambda}^\star(s)-\acute{\lambda}^\star_{\gamma,\delta}(s)\right|\De s = \int_{u^\star}^\infty \left(h + g\right)\mathrm{e}^{-s}\De s 
 \le \frac{(\gamma+\delta)p_{11}}{1-(\gade)p_{11}}\, \mathrm{e}^{-u^\star},
\]
again by Lemma \ref{t: Lemma_cond} (i) and (ii). For $s < t$, note that 
\begin{align*}
\ls(s,t)%&=\frac{\log (p_{00}/q_2) \log q_2}{\log^2(1/p_{00})}\, \mathrm{e}^{-\frac{\log (p_{00}/q_2)}{\log p_{00}}\,s} \mathrm{e}^{-\frac{\log q_2}{\log p_{00}}\,t}\\
&= \left[ \frac{\gamma}{\gade} + h\right] \left[ \frac{1+\delta}{\gade} - h\right]
\mathrm{e}^{- \frac{\gamma}{\gade}\,s} \mathrm{e}^{-\frac{1+\delta}{\gade}\,t}\mathrm{e}^{h(t-s)}\\
& = \lsnew(s,t) \mathrm{e}^{h(t-s)} + \left(\frac{1+ \delta - \gamma}{\gade} \, h - h^2\right)\mathrm{e}^{-\frac{\gamma}{\gade}\, s} \mathrm{e}^{-\frac{1+\delta}{\gade}\, t} \mathrm{e}^{h(t-s)},
\end{align*}
where $\ls(s,t)$ and $ \lsnew(s,t)$ are defined by (\ref{d: ls}) and (\ref{d: lsnew}), respectively.
Thereby,
\[
\left| \ls(s,t) - \lsnew(s,t)\mathrm{e}^{h(t-s)}\right| = \left| \frac{1+ \delta - \gamma}{\gade} \, h - h^2\right| 
\mathrm{e}^{-\frac{\gamma}{\gade}\, s} \mathrm{e}^{-\frac{1+\delta}{\gade}\, t} \mathrm{e}^{h(t-s)},
\]
where $\left| \frac{1+ \delta - \gamma}{\gade} \, h - h^2\right|  \le h + h^2 $. 
Note that we have
\begin{equation}\label{p: MO_Geo_nicer_appl_0}
\left| \ls(s,t) - \lsnew(s,t)\right| \le \left| \ls(s,t) - \lsnew(s,t)\mathrm{e}^{h(t-s)}\right| + \lsnew(s,t)\left| \mathrm{e}^{h(t-s)} - 1\right|.
\end{equation}
We first compute the following integral:
\begin{equation}\label{p: MO_Geo_d2_nicer_appl_integral}
\int_{u^\star}^\infty \int_{u^\star}^t \mathrm{e}^{-\frac{\gamma}{\gade}\, s-\frac{1+\delta}{\gade}\, t+ h(t-s)} \De s\De t
= \int_{u^\star}^\infty \int_{u^\star}^t \mathrm{e}^{-\frac{\log (p_{00}/q_2)}{\log p_{00}}\,s-\frac{\log q_2}{\log p_{00}}\,t}\De s\De t
= \frac{\log p_{00}}{\log q_2}\, \mathrm{e}^{-u^\star},
\end{equation}
where, using $z \le -\log(1-z) \le \frac{z}{1-z}$ for all $|z| \le 1$, and (\ref{d: conditions}),
\begin{align}\label{MO_Geo_nicer_appl_2}
\begin{split} 
\frac{\log p_{00}}{\log q_2} &=\frac{-\log[1-(\gade)p_{11}]}{- \log[1-(1+\delta)p_{11}]} \le \frac{\gade}{(1+\delta)[1-(\gade)p_{11}]}\,\\
&\le \frac{\gade}{1+ \delta}\left\{ 1+ \frac{(\gade) p_{11}}{[1-(\gade)p_{11}]^2}\right\}.
\end{split}
\end{align}
Moreover, by Lemma \ref{t: Lemma_cond} (i), 
\[
h+ h^2 \le \frac{\gamma p_{11}}{1-(\gade)p_{11}} +  \left[\frac{\gamma p_{11}}{1-(\gade)p_{11}} \right]^2 \le \frac{2\gamma p_{11}}{[1-(\gade)p_{11}]^2},
\]
since $\gamma p_{11} = p_{10} < 1$, and therefore $(\gamma p_{11})^2 \le \gamma p_{11}$. 
Then,
\begin{align}\label{p: MO_Geo_nicer_appl_4}
\begin{split}
\int_{u^\star}^\infty \int_{u^\star}^t  \left| \ls(s,t) - \lsnew(s,t)\mathrm{e}^{h(t-s)}\right|\De s\De t 
\le \frac{2 \gamma (\gade)p_{11}}{(1+\delta)[1-(\gade)p_{11}]^3}\, \mathrm{e}^{-u^\star},
\end{split}
\end{align}
which gives a bound for the integral of the first error term in (\ref{p: MO_Geo_nicer_appl_0}). For the second error term in (\ref{p: MO_Geo_nicer_appl_0}), note first that
$|\mathrm{e}^{h(t-s)}-1| = \mathrm{e}^{h(t-s)} -1$, since $h \ge 0$ and $t > s$. By (\ref{p: MO_Geo_d2_nicer_appl_integral}) and (\ref{MO_Geo_nicer_appl_2}), 
and with definition (\ref{d: lsnew}) of $\lsnew(s,t)$, we obtain 
\begin{align}\label{t: MO_Geo_nicer_appl_2}
\begin{split}
\int_{u^\star}^\infty \int_{u^\star}^t  \lsnew(s,t)\mathrm{e}^{h(t-s)} &= \frac{\gamma(1+\delta)\log p_{00}}{(\gade)^2 \log q_2}\, \mathrm{e}^{-u^\star}\\
&\le \frac{\gamma}{ \gade }\left\{ 1+ \frac{(\gade) p_{11}}{[1-(\gade)p_{11}]^2}\right\}\, \mathrm{e}^{-u^\star},
\end{split}
\end{align}
whereas
\begin{equation}\label{t: MO_Geo_nicer_appl_3}
\int_{u^\star}^\infty \int_{u^\star}^t \lsnew(s,t)\De s\De t = \frac{\gamma}{\gade}\, \mathrm{e}^{-u^\star}.
\end{equation}
By (\ref{t: MO_Geo_nicer_appl_2}) and (\ref{t: MO_Geo_nicer_appl_3}), we may thus bound the integral of the second error term in (\ref{p: MO_Geo_nicer_appl_0})
as follows:
\begin{align}\label{p: MO_Geo_nicer_appl_5}
\begin{split}
\int_{u^\star}^\infty \int_{u^\star}^t \lsnew(s,t)\left| \mathrm{e}^{h(t-s)} - 1\right|\De s\De t
&\le \frac{\gamma p_{11}}{[1-(\gade)p_{11}]^2}\,\mathrm{e}^{-u^\star}.
\end{split}
\end{align}
Hence, for $s < t$, (\ref{p: MO_Geo_nicer_appl_0}), (\ref{p: MO_Geo_nicer_appl_4}) and (\ref{p: MO_Geo_nicer_appl_5}) give
\[
\int_{u^\star}^\infty \int_{u^\star}^t  \left| \ls(s,t) - \lsnew(s,t)\right|\De s\De t
\le  \frac{\gamma p_{11}}{[1-(\gade)p_{11}]^3}\, \left\{ \frac{2(\gade)}{1+ \delta} + 1 \right\}\mathrm{e}^{-u^\star}.
\]
By proceeding analogously for $s>t$, we obtain
\[
\int_{u^\star}^\infty \int_{u^\star}^s  \left| \ls(s,t) - \lsnew(s,t)\right|\De t\De s
\le \frac{\delta p_{11}}{[1-(\gade)p_{11}]^3}\, \left\{ \frac{2(\gade)}{1+ \gamma} + 1 \right\}\mathrm{e}^{-u^\star}.
\]
The sum of the bounds for the three cases $s=t$, $s<t$ and $s>t$ yields 
\begin{eqnarray*}
&&\int_{\As} \left| \bls(d\bz) - \blsnew(d\bz)\right| \\
&&\le \frac{2p_{11}}{[1-(\gade)p_{11}]^3}\left\{ \frac{\gamma(\gade)}{1+\delta} + \frac{\delta(\gade)}{1+\gamma} + \gamma + \delta\right\}\mathrm{e}^{-u^\star}\le \frac{4(\gade)^2 p_{11}}{[1-(\gade)p_{11}]^3}\, \mathrm{e}^{-u^\star},
\end{eqnarray*}
where we used $(1+\gamma)^{-1}, (1+\delta)^{-1} < 1$, and $\gamma + \delta \le \gade \le (\gade)^2$ for the second inequality. \\
(ii) Direct computations give $\bls(\As)= \mathrm{e}^{-u^\star} = \blsnew(\As)$. Theorem \ref{t: approx_cont_nicer} (ii), together with the bound from (i), then
immediately gives the result. 
\end{proof}
%-----------------------------------------------------------------------------------------------------------------------------------------------------------------------
\noindent The error bounds established in Proposition \ref{t: MO_Geo_cont_to_nicer_appl} are similar to the error bound from Proposition \ref{t: MO_Geo_disc_to_cont_appl}. As before, $p_{11n}$ needs to converge to $0$ fast enough to make up for the factor $\mathrm{e}^{-u_n^\star}$ which increases the size of the error as soon as $u_n^\star < 0$. And since $u_n^\star \ge 0$ gives $1$ or no points in $\As$, the mean number of points in $\As$ being given by $\mathrm{e}^{-u_n^\star}$ for either process, we would %certainly
want the threshold $u_n^\star$ to be negative. 

The biggest difference between the $d_2$-bounds from Propositions \ref{t: MO_Geo_disc_to_cont_appl} and  \ref{t: MO_Geo_cont_to_nicer_appl} is that 
the former contains the multiplicative factor $[1-(\gade)p_{11n}]^{-2}$ and the latter the bigger factor $[1-(\gade)p_{11n}]^{-3}$. However, 
since we need $p_{11n} \to 0$ as $n \to \infty$, we will have $(\gade)p_{11n} \le 1/2$ for all $n$ large enough. 
Then $[1-(\gade)p_{11n}]^{-3} \le 2 [1-(\gade)p_{11n}]^{-2}$ so that 
both error bounds will be of the same rate. Hence, for large enough $n$, the approximation by a further Poisson process does not add an error of a bigger size 
than the one that arises from the approximation by only $\mathrm{PRM}(\bls)$.
%-----------------------------------------------------------------------------------------------------------------------------------------------------------------------
\subsubsection{Final bound in the $d_2$-distance}\label{s: MO_Geo_final_bound_d2}
The following corollary summarises the results from Sections \ref{s: MO_Geo_dTV}, \ref{s: MO_Geo_d2} and \ref{s: MO_Geo_nicer}. 
It gives an estimate for the error in the $d_2$-distance of the approximation of the law of an MPPE $\Xi^\star_{A^\star}$ with \iid Marshall-Olkin geometric marks, living on a lattice of points contained in $\As \cap[-\log n, \infty)^2$, by the law of a Poisson process with a continuous intensity measure $\blsnew$ over $\As \cap[-\log n, \infty)^2$, where $\As = [u^\star, \infty)^2$ for some choice of threshold $u^\star \ge -\log n$.
%-----------------------------------------------------------------------------------------------------------------------------------------------------------------------
\begin{cor}\label{t: MO_Geo_final_bound_d2} Let $p_{11n} \in (0,1)$ and assume that $q_{1n}$, $q_{2n}$ and $p_{00n}$ satisfy (\ref{d: conditions}). For any choice of $u_n^\star \ge -\log n$, 
define $A^\star = [u_n^\star,\infty)^2$.  
With the notations from Sections \ref{s: MO_Geo} and \ref{s: MO_Geo_dTV}-\ref{s: MO_Geo_cont_nicer_int_fct},
\begin{eqnarray*}
&&d_2\left(\mathcal{L}(\Xi^\star_{\As}), \mathrm{PRM}(\blsnew)\right) \le \frac{\mathrm{e}^{-u^\star_n}}{n} + \frac{(\gade)^2p_{11n}}{[1-(\gade)p_{11n}]^3} \left\{ \sqrt{2} + 7\min\left\{  \mathrm{e}^{-u_n^\star},\,1.65 \mathrm{e}^{-u_n^\star/2}\right\} \right\}.
\end{eqnarray*}
\begin{proof}
We have 
\begin{align*}
&d_2\left(\mathcal{L}(\Xi^\star_{\As}), \mathrm{PRM}(\blsnew)\right) \\
&\le  d_2\left(\mathcal{L}(\Xi^\star_{\As}), \mathrm{PRM}(\bpis)\right)
+    d_2\left(\mathrm{PRM}(\bpis), \mathrm{PRM}(\bls)\right)
+    d_2\left(\mathrm{PRM}(\bls), \mathrm{PRM}(\blsnew)\right).
\end{align*}
By Theorem \ref{t: MO_Geo_dTV_lattice},
\[
d_2\left(\mathcal{L}(\Xi^\star_{\As}), \mathrm{PRM}(\bpis)\right) \le d_{TV}\left(\mathcal{L}(\Xi^\star_{\As}), \mathrm{PRM}(\bpis)\right) \le \frac{\mathrm{e}^{-u_n^\star}}{n}\,.
\]
Furthermore, with the results from Propositions \ref{t: MO_Geo_disc_to_cont_appl} and \ref{t: MO_Geo_cont_to_nicer_appl}, and using $(\gade)\le(\gade)^2$ and 
$[1-(\gade)p_{11}]^{-2} \le [1-(\gade)p_{11}]^{-3}$, we obtain
\begin{multline*}
d_2\left(\mathrm{PRM}(\bpis), \mathrm{PRM}(\bls)\right) + d_2\left(\mathrm{PRM}(\bls), \mathrm{PRM}(\blsnew)\right) \\
\le \frac{(\gade)^2p_{11n}}{[1-(\gade)p_{11n}]^3} \left\{ {\sqrt{2}} + 7 \min\left\{ \mathrm{e}^{-{u^\star_n}} ,\, 1.65 \mathrm{e}^{-{u^\star_n}/2}\right\}\right\}.
\end{multline*}
\end{proof}
\end{cor}
\noindent By far the smallest component of the error estimate from Corollary \ref{t: MO_Geo_final_bound_d2} is given by $\mathrm{e}^{-u_n^\star}/n$, the error arising from approximating $\mathcal{L}(\Xi^\star_{A^\star})$ by $\mathrm{PRM}(\mathbb{E}\Xi^\star_{A^\star})$, which lives on the lattice $\As \cap E^\star$ just as $\Xi^\star_{\As}$. 
%This part of the error corresponds exactly to the overall error estimate that we obtained in Section \ref{s: MO_Exp_joint} for the approximation of an MPPE $\Xi^\star_{A^\star}$ with Marshall-Olkin exponential marks by a Poisson process with mean measure $\mathbb{E}\Xi^\star_{A^\star}$. Yet the Marshall-Olkin exponential is a continuous distribution and the mean measure $\mathbb{E}(\Xi^\star_{A^\star})$ is thereby also continuous.
%As for MPPE's with univariate geometric marks (see Section \ref{s: MPPE_geo}), a
A
far bigger error emerges for the MPPE with Marshall-Olkin geometric marks when going from the Poisson process on the lattice to a Poisson process on $\As \cap [-\log n, \infty)^2$ with continuous intensity. 
This error can only be small if the probability of simultaneous success of the Marshall-Olkin geometric distribution, $p_{11}$, and thereby also the marginal success probabilities $1-q_{1n}$ and $1-q_{2n}$, tend to zero as $n \to \infty$ at a rate fast enough to compensate for the factor $\mathrm{e}^{-u_n^\star}$, the (rough) number of points expected in $\As_n$ for each of the processes. For instance, for $\As_n = [-\log \log n, \infty)^2$ and $p_{11n}=1/n$, we expect $\log n$ joint threshold exceedances, and obtain
\begin{align*}
&d_2\left(\mathcal{L}(\Xi^\star_{\As}), \mathrm{PRM}(\blsnew)\right) \\
&\le \frac{\log n}{n} + \frac{(\gade)^2}{n[1-(\gade)/n]^3} \left\{ {\sqrt{2}} + 7\min\left\{  \log n ,\, 1.65 \sqrt{\log n}\right\} \right\} \le \frac{C \log n}{n},
\end{align*}
where $C$ is some constant. With the (very strong) condition $p_{11n}=1/n$, we thus obtain an error of the same size as the error that we obtain when approximating $\mathcal{L}(\Xi^\star_{\As})$ only by $\mathrm{PRM}(\mathbb{E}\Xi^\star_{A^\star})$.
%................................................................................................................
\section*{Acknowledgements} The thorough and precise reports of two anonymous referees, which led to several improvements of the paper, is kindly acknowledged.%.........................................................................................................................................
\bibliography{Bibliography}
\end{document}